\documentclass{amsart}[11pt]

\usepackage{amsmath, amsthm, amssymb, amsfonts, verbatim}
\usepackage{array}
\usepackage[active]{srcltx}
\usepackage{epsfig}
\usepackage{color}
\usepackage{comment}
\usepackage{setspace}
\usepackage{textcomp}
\usepackage{multirow}
\usepackage{xspace,color}
\usepackage{subcaption}
\usepackage{tikz}

\usepackage{hyperref,mathabx,color,enumerate,chngpage}
\usepackage{stmaryrd}
\usepackage{graphicx}
\usepackage{mathrsfs}
\usepackage{indentfirst}
\usepackage{pdfsync}
\newtheorem{thm}{Theorem}[section]

\newcommand{\pmat}[1]{\begin{pmatrix} #1 \end{pmatrix}}

\newcommand{\algn}[1]{\begin{align} #1 \end{align}}
\newcommand{\algns}[1]{\begin{align*} #1 \end{align*}}

\newcommand{\gat}[1]{\begin{gather} #1 \end{gather}}

\newcommand{\LRp}[1]{\left( #1 \right)}

\newcommand{\LRa}[1]{\left< #1 \right>}

\newcommand{\jump}[1] {\ensuremath{\left[\!\left[{#1}\right]\!\right]}}

\newcommand{\R}{{\mathbb R}}

\newcommand{\curl}{\operatorname{curl}}
\newcommand{\rot}{\operatorname{rot}}
\renewcommand{\div}{\operatorname{div}}

\newcommand{\spn}{\operatorname{span}}
\newcommand{\supp}{\operatorname{supp}}
\newcommand{\osc}{\operatorname{osc}}

\newcommand{\bs}{\boldsymbol}

\newcommand{\pd}{\partial}

\newcommand{\nb}{\bs{n}}

\newcommand{\tnorm}[2][]{%
	\mathopen{#1|\mkern-1.5mu#1|\mkern-1.5mu#1|}
	#2
	\mathclose{#1|\mkern-1.5mu#1|\mkern-1.5mu#1|}
}

\newcommand{\tred}[1]{\textcolor{black}{#1}}

\newtheorem{theorem}{Theorem}[section]
 
\newtheorem{lemma}[theorem]{Lemma}
\newtheorem{cor}[theorem]{Corollary}

\newtheorem{rmk}[theorem]{Remark}

\newcommand{\ub}{{\bs{u}}}

\newcommand{\vb}{{\bs{v}}}
\newcommand{\Vb}{{\bs{V}}}

\title[A posteriori estimate of Darcy flow with Robin interface conditions]{A posteriori error estimates of Darcy flows with Robin-type jump interface conditions}

\begin{document}

\author{Jeonghun J. Lee}
\address{Department of Mathematics, Baylor University, Waco, Texas, USA }

\email{jeonghun\_lee@baylor.edu}

\begin{abstract}
  In this work we develop an a posteriori error estimator for mixed finite element methods of Darcy flow problems with Robin-type jump interface conditions. We construct an energy-norm type a posteriori error estimator using the Stenberg post-processing. The reliability of the estimator is proved using an interface-adapted Helmholtz-type decomposition and an interface-adapted Scott--Zhang type interpolation operator. A local efficiency and the reliability of post-processed pressure are also proved. Numerical results illustrating adaptivity algorithms using our estimator are included.
\end{abstract}

\date{April, 2023}
\maketitle

\section{Introduction}

Fluid flow in porous media appears in various fields of science and engineering applications. Therefore, mathematical modeling and numerical methods for finding accurate numerical solutions of porous media flow have been important problems in computational mathematics. Recently, mathematical models in which porous media domains have low-dimensional fault (or fracture) structures are considered for accurate descriptions of more realistic porous media flow. 
In \cite{Martin-Jaffre-Roberts:2005}, some porous media flow models with fault/fracture structures were proposed in which fluid flow on fractures and in surrounding porous media are governed by separate partial differential equations with coupling conditions. In \cite{Lee-Tan-Villa-Ghattas:2022}, a reduced model was derived under the assumption that there is no fluid flow along fault/fracture structures because of very low permeability on fault/fracture. In the reduced models, the fluid flow and the pressure jump on faults are related by a
Robin-type interface condition. We remark that similar models with nonlinear extensions are used for porous media flows with semi-permeable membrane structures in consideration of their applications to chemical processes in biological tissues (see, e.g., \cite{Cangiani-Georgoulis-Jensen:2013,Cangiani-Georgoulis-Sabawi:2018}).



The purpose of this paper is to obtain a posteriori error estimators for the model derived in \cite{Lee-Tan-Villa-Ghattas:2022} with the dual mixed form of finite element methods. 
We remark that a posteriori error estimate results for the more complex models in \cite{Martin-Jaffre-Roberts:2005,Angot-Boyer-Hubert:2009} (see \cite{Chen-Sun:2017,Hecht-Mghazli-Roberts:2019,Zhao-Chung:2022} for a posteriori error estimates), do not  imply a posteriori error estimate results for the model that we are interested in. 
This is because a less number of error terms makes local efficiency of a posteriori error estimators more difficult. 

We also remark that the problem in this paper can be viewed as a generalization of mixed finite element methods for Poisson equations with Robin boundary conditions which was studied in \cite{Konno-Schotzau-Stenberg:2011}. A priori and a posteriori error estimates are done in \cite{Konno-Schotzau-Stenberg:2011} using the mesh-dependent norm approach (cf. \cite{Babuska-Osborn-Pitkaranta:1980,Lovadina-Stenberg:2006,Stenberg:1991}) but the saturation assumption is necessary for the reliability of the estimator. The analysis in this paper does not need such an assumption for reliability, and it also gives a new reliability estimate for post-processed pressure.

The paper is organized as follows. In Section 2 we present background notions on function spaces, governing equations, finite element discretization. We define our a posteriori error estimator and prove its reliability and local efficiency in Section 3. In Section 4 and 5, we present numerical experiment results which show performance of our a posteriori error estimator, and conclusion with summary. Finally, some calculus identities which are used in our analysis are explained in Section 6 as appendix.

\section{Preliminaries for governing equations}

\subsection{Notation and definitions}

For a bounded domain $D \subset \mathbb{R}^n$ ($n=2,3$) with positive $n$-dimensional Lebesgue measure, 
we use the convention that $(u , v)_{D} = \int_{D} uv \, dx$ for a subdomain $D \subset \overline{\Omega}$ which has positive $n$-dimensional Lebesgue measure. Similarly, $\LRa{u , v}_D= \int_{D} u v \,dS$ for a subdomain $D \subset \overline{\Omega}$ which has positive $(n-1)$ or $(n-2)$-dimensional Lebesgue measure up to context.

\subsection{Governing equations and variational formulation}
Let $\Omega \subset \mathbb{R}^n$, $n=2,3$ be a homologically trivial bounded domain with polygonal/polyhedral boundary. 
We assume that a fault $\Gamma$ is a union of disjoint ($n-1$)-dimensional piecewise linear submanifolds in $\Omega$. Each connected component of $\Gamma$ is a union of linear segments (if $n=2$) or as a union of planar domains such that the boundary of each planar domain is a union of linear segments. We also assume that there are two open subdomains $\Omega_+,\Omega_- \subset (\Omega \setminus \Gamma)$ with Lipschitz boundaries such that 
\algns{
	\overline{\Omega}= \overline{\Omega_+} \cup \overline{\Omega_-}, \qquad \Gamma \subset \partial \Omega_+ \cap \partial \Omega_-,
}
and only one side of $\Gamma$ is in contact with $\Omega_+$ or $\Omega_-$. 
Let $\bs{n}_+$ and $\bs{n}_-$ be the two unit normal vector fields on $\Gamma$ with opposite directions ($\bs{n}_+ =-\bs{n}_-$) such that $\bs{n}_{\pm}$ correspond to the unit outward normal vector fields from $\Omega{\pm}$ (see~Figure~\ref{fig:model-domain}). 

Suppose that $\Gamma_D$, $\Gamma_N$ are disjoint ($n-1$)-dimensional open submanifolds in $\partial \Omega$ such that $\overline{\Gamma_D} \cup \overline{\Gamma_N} = \partial \Omega$. 
In this paper we assume the following:
\begin{align} 
	\notag 
    &\text{ For any }q \in L^2(\Omega) \text{ there exists } \bs{w} \in H^1(\Omega;\R^n) 
    \\ 
	\label{eq:assumption1}
    &\text{ such that } \bs{w}|_{\Gamma \cup \Gamma_N}= 0, \div \bs{w}=q \text{ and } \|\bs{w}\|_1 \le C\|q\|_0 
    \\ 
	\notag
	&\text{ with a constant } C>0 \text{ depending on } \Omega, \Gamma, \Gamma_N .
\end{align}
The assumption \eqref{eq:assumption1} is a weak assumption. For example, if both of $\partial \Omega_+ \cap \partial \Omega \cap \tred{\Gamma_D}$ and $\partial \Omega_- \cap \partial \Omega \cap \tred{\Gamma_D}$ have positive ($n-1$)-dimensional Lebesgue measures, then \eqref{eq:assumption1} holds. \tred{To see this, suppose that $q \in L^2(\Omega)$ is given. Note that there exist $\boldsymbol{w}_+ \in H^1(\Omega_+; \mathbb{R}^n) $ such that $\div \bs{w}_+ = q|_{\Omega_+}$, $\bs{w}_+|_{\pd \Omega_+ \setminus \Gamma_D} = 0$, and $\| \bs{w}_+ \|_1 \le C \| q|_{\Omega_+}\|_0$ (see, e.g., \cite[Lemma~B.1]{Lee-Baerland-Mardal-Winther:2017}). Similarly, there exists $\boldsymbol{w}_- \in H^1(\Omega_-; \mathbb{R}^n)$, $\div \bs{w}_- = q|_{\Omega_-}$, $\bs{w}_-|_{\pd \Omega_- \setminus \Gamma_D} =0$, and $\| \bs{w}_- \|_1 \le C \| q|_{\Omega_-}\|_0$. Then, $\bs{w}$ defined by $\bs{w}|_{\Omega_\pm} = \bs{w}_{\pm}$, satisfies \eqref{eq:assumption1}. }

%

For any $q \in L^2(\Omega)$ with sufficient regularity, we use $q_+|_{\Gamma}$ and $q_-|_{\Gamma}$ to denote the traces of $q|_{\Omega_+}$ and $q|_{\Omega_-}$ on $\Gamma$. For simplicity we use $\jump{q}|_{\Gamma}:= q_+|_{\Gamma} - q_-|_{\Gamma}$.  
Note that the continuity of $q$ on $\Gamma$ is not assumed in general, so $\jump{q}|_{\Gamma} \not = 0$.
For a vector-valued function $\bs{v}$ on $\Omega$ with enough regularity (e.g., $\bs{v} \in H^s(\Omega \setminus \Gamma; \mathbb{R}^n)$ with $s > 1/2$), $\bs{v}_+|_{\Gamma}$ and $\bs{v}_-|_{\Gamma}$ are well-defined as the traces of $\bs{v}$ from $\Omega_+$ and $\Omega_-$. We say that $\bs{v}$ satisfies normal continuity on $\Gamma$ if $\bs{v}_+|_{\Gamma} \cdot \bs{n}_+ = -\bs{v}_-|_{\Gamma}\cdot \bs{n}_-$ on $\Gamma$.


\begin{figure}
 \begin{center}
 \begin{tikzpicture}[scale=2.5]
     \draw [color=black,fill = blue!20!green](0,0) rectangle (2,2);
     \draw [dotted,color=black] (1,0)--(1,0.4);
     \draw [dotted,color=black] (1,1.6)--(1,2);
     \draw [thin,color=black] (1,0.4)--(1,1.6)
     node[ near start, below, right ]{$\Gamma$};
     \draw[->] (1,1)--(1.3,1) node[midway,above]{$n_+$};
     \draw[->] (1,1)--(0.7,1) node[midway,above]{$n_-$};
     \draw [color=black] (0,1)--(0,2)--(2,2)--(2,1);
     \node at (0.5,0.5) {$\Omega_+$};
     \node at (1.5,0.5) {$\Omega_-$};
 \end{tikzpicture}
 \end{center}
 \caption{A model domain with interface $\Gamma$}
 \label{fig:model-domain} 
 \end{figure}

For governing equations assume that $\kappa$ is a symmetric positive definite tensor on $\Omega$. 
In Darcy flow problems, the pressure $p$ and fluid flow $\bs{u}$ satisfy Darcy's law $\bs{u} = - \kappa \nabla p$ in $\Omega$.
Conservation of mass gives $\div \bs{u} = f$ for given source/sink function $f$ on $\Omega$. 
The pressure and flux boundary conditions are given by 
\algns{
	p = g_D \text{  on  }\Gamma_D, \quad \ub \cdot \nb = g_N \text{  on  }\Gamma_N, 
}
and the interface condition on $\Gamma$ is $\alpha \bs{u}_+ \cdot \bs{n}_+ - \jump{p}|_{\Gamma} = 0$ with $\alpha >0$. Summarizing these equations and conditions, 
a strong form of equations with a dual mixed formulation of the Darcy flow equation in domain $\Omega$ with fault $\Gamma$ reads:
\gat{
	\label{eq:equations}
	\kappa^{-1} \bs{u}+\nabla p = 0  \text{ in } \Omega, \qquad     \div \bs{u} = f  \text{ in } \Omega,
	\\
	\label{eq:boundary-conditions}
	\bs{u}\cdot\bs{n}=g_N  \text{ on } \Gamma_N, \quad p=g_D  \text{ on } \Gamma_D,
	\\
	\label{eq:interface-condition}
	\alpha \bs{u}_+\cdot\bs{n}_+ - \jump{p} = 0 \text{ on } \Gamma.
}
Throughout this paper we assume that $\alpha$ is constant on $\Gamma$ and 
\algn{ \label{eq:alpha-large}
	\tred{0 < \alpha_0 \le \alpha \le \alpha_1 < \infty}
}
with a uniform \tred{lower and upper bounds $\alpha_0, \alpha_1$}, and we do {\it not} consider the limit cases $\alpha \to 0^+$ or $\alpha \to +\infty$. The limit case $\alpha=0$ becomes the classical Darcy flow problems without fault which does not need the interface condition \eqref{eq:interface-condition}. The $\alpha = +\infty$ case corresponds to the problem that no fluid flows across $\Gamma$ which needs $\bs{u} \cdot \bs{n}|_{\Gamma} = 0$ as an interface condition. This case needs a completely different way to implement the interface condition $\bs{u} \cdot \bs{n}|_{\Gamma} = 0$ with the dual mixed finite element methods, the numerical method that we use in this paper. Therefore, $\alpha = +\infty$ case cannot be covered by the work in this paper. However, our analysis does not need a uniform upper bound of $\alpha$, so the results in the paper are valid for nearly impermeable $\Gamma$, i.e., for arbitrarily large but finite $\alpha$.

Hereafter, we assume $\pd \Omega = \Gamma_D$, $g_D= 0$, $\kappa = 1$ for simplicity of discussions. Let $Q=L^2(\Omega)$, and $H(\div,\Omega)$ be the space of $\R^n$-valued $L^2$ functions on $\Omega$ such that its distributional divergence is in $L^2(\Omega)$. We define 
\begin{align*}
	\bs{V} := \{\bs{v} \in H(\div,\Omega):\bs{v}_+\cdot \bs{n}_+|_\Gamma = -\bs{v}_-\cdot \bs{n}_-|_\Gamma \in L^2(\Gamma)\}
\end{align*}
with two norms
\algn{
	\label{eq:fault-weighted-L2}
    \tnorm{\bs{v} } &= \LRp{ \int_{\Omega} \kappa^{-1} \bs{v} \cdot \bs{v} \,dx + \sum_{F \subset \Gamma} \int_F \alpha (\bs{v}\cdot\bs{n}) (\bs{v}\cdot\bs{n}) \,ds }^\frac12 ,
    \\
	\label{eq:weighted-Hdiv}
    \|\bs{v}\|_{\bs{V}} &= \LRp{ \tnorm{\bs{v}}^2 + \|\div \bs{v}\|_{0}^{2} }^{\frac 12} .
}

By multiplying $\bs{v} \in \bs{V}$ to the first equation in \eqref{eq:equations} and taking the integration by parts with $g_D = 0$, $\kappa=1$, 
\algns{
	&\int_{\Omega} \bs{u} \cdot \bs{v} \,dx + \int_{\Omega} \nabla p \cdot \bs{v} \,dx 
	\\
	&= \int_{\Omega} \bs{u} \cdot \bs{v} \,dx + \int_{\Gamma} p_+  \bs{v}_+ \cdot \bs{n}_+ \,ds + \int_{\Gamma} p_- \bs{v}_- \cdot \bs{n}_- \,ds 
	\\
	&\quad - \int_{\Omega} p \div \bs{v} \,dx .
}
After using $\bs{v}_+\cdot\bs{n}_+ =-\bs{v}_-\cdot\bs{n}_-$ on $\Gamma$, and the interface condition \eqref{eq:interface-condition}, we obtain
%
%
%
\algns{
	\int_\Omega \bs{u}\cdot\bs{v}\,dx - \int_\Omega p\div \bs{v}\,dx+ \int_{\Gamma} \alpha (\bs{u}_+\cdot\bs{n}_+) (\bs{v}_+\cdot\bs{n}_+) \,ds = 0 
}
which can be written as $\LRp{\bs{u}, \bs{v}}_{\Omega} - \LRp{p, \div \bs{v}}_{\Omega} + \LRa{ \alpha \bs{u}_+ \cdot \bs{n}_+, \bs{v}_+ \cdot \bs{n}_-}_{\Gamma}=0$. In the following, we use $\LRa{\alpha \bs{u}\cdot\bs{n},\bs{v}\cdot\bs{n}}_\Gamma$ to denote $\LRa{\alpha \bs{u}_+\cdot\bs{n}_+,\bs{v}_+\cdot\bs{n}_+}_\Gamma$. 
From this and an immediate variational form of the second equation in \eqref{eq:equations}, we have a variational problem to find $(\ub, p)\in \bs{V} \times Q$ such that 
\begin{subequations} \label{eq:variational-eqs}
	\algn{ 
		\label{eq:variational-eq1}
    	\LRp{\bs{u},\bs{v}}_{\Omega} + \LRa{\alpha \bs{u}\cdot\bs{n},\bs{v}\cdot\bs{n}}_{\Gamma} - \LRp{p,\div \bs{v}}_{\Omega}  &= 0 & &\forall\bs{v} \in \bs{V},
    	\\
		\label{eq:variational-eq2}
	    \LRp{\div \bs{u},q}_{\Omega} &= \LRp{f, q}_{\Omega} & &\forall q \in Q.
	}
\end{subequations}
The stability of this system with an inf-sup condition
\algn{ \label{eq:continuous-inf-sup}
	\inf_{\bs{v} \in \bs{V}} \sup_{q \in Q} \frac{\LRp{q, \div \bs{v}}_{\Omega} }{\| \bs{v} \|_{\bs{V}} \| q \|_0} \ge C > 0
}
%
was studied in \cite{Lee-Tan-Villa-Ghattas:2022}. 

\subsection{Discretization with finite elements}

We introduce finite element spaces for discretization. In the rest of the paper we assume that $k \ge 1$ is a fixed integer.
For a $d$-dimensional simplex $D \subset \R^n$ ($d=n,n-1,n-2$), $\mathcal{P}_k(D)$ is the space of polynomials on $D$ of degree $\le k$. Similarly, $\mathcal{P}_k(D; \mathbb{R}^d)$ is the space of $\mathbb{R}^d$-valued polynomials of degree $\le k$ on the $d$-dimensional simplex $D$. 

\tred{Let $\mathcal{T}_h$ be a set of $n$-dimensional simplices whose interiors are disjoint such that if any two simplices $T_1, T_2 \in \mathcal{T}_h$ are not disjoint, then $T_1 \cap T_2$ is a subsimplex of $T_1$ and $T_2$. If $n=3$, let $\mathcal{F}_h$ denote the set of all $(n-1)$-dimensional subsimplices $F$ of the simplices in $\mathcal{T}_h$, and $\mathcal{F}_h^{\partial} = \{ F \in \mathcal{F}_h \,: \, F \subset \partial \Omega\}$. We assume that $\mathcal{T}_h$ is matching with the fracture $\Gamma$ in the sense that $\Gamma = \cup_{F \in \mathcal{F}_h^{\Gamma}} F$ for some $\mathcal{F}_h^{\Gamma} \subset \mathcal{F}_h$, so $\mathcal{F}_h^{\Gamma}$ forms a triangulation of $\Gamma$. 
We also define $\mathcal{F}_h^0$ by $\mathcal{F}_h^0 := \mathcal{F}_h \setminus (\mathcal{F}_h^{\Gamma} \cup \mathcal{F}_h^{\partial})$. If $n=2$, $\mathcal{E}_h$, $\mathcal{E}_h^{\pd}$, $\mathcal{E}_h^{\Gamma}$, $\mathcal{E}_h^0$ are similarly defined. Finally, $h_K$ denotes the diameter of a simplex $K$ which can be a tetrahedron, a triangle, or an edge. }

For given $k\ge 1$ and $n$-dimensional simplex $T\in \mathcal{T}_h$ let us define 
\algn{
	\label{eq:RTN-local}
  \Vb_{k-1}^{RTN}(T) &= \mathcal{P}_{k-1}(T; \mathbb{R}^n) +  \begin{pmatrix}
  x_1 \\
  \vdots \\
  x_n
  \end{pmatrix} \mathcal{P}_{k-1}(T), 
  \\
	\label{eq:BDM-local}
  \bs{V}_k^{BDM}(T) &= \mathcal{P}_k(T; \mathbb{R}^n) .
}
The Raviart--Thomas element (for $n=2$, \cite{RT75}) and the first kind of N\'ed\'elec $H(\div)$ element (for $n=3$, \cite{Nedelec80}) are defined by
\algn{ \label{eq:RTN-space}
  \Vb_h^{RTN} &= \{ \vb \in \bs{V} \,:\, \vb|_T \in \Vb_{k-1}^{RTN}(T), \quad \forall T \in \mathcal{T}_h \} .
}
The Brezzi--Douglas--Marini element (for $n=2$, \cite{BDM85}) and the second kind of N\'ed\'elec $H(\div)$ element (for $n=3$, \cite{Nedelec86}) are defined by 
\algn{ \label{eq:BDM-space}
  \bs{V}_h^{BDM} &= \{ \vb \in \bs{V} \,:\, \vb|_T \in \bs{V}_k^{BDM}(T), \quad \forall T \in \mathcal{T}_h \} .
}
The finite element spaces $Q_h^0$, $Q_h$, $Q_h^*$ are defined by 
\algn{
	\notag 
	Q_h^0 &= \{ q \in Q \,:\, q|_T \in \mathcal{P}_{0}(T) \quad \forall T \in \mathcal{T}_h \} ,
	\\
	\notag
	Q_h &= \{ q \in Q \,:\, q|_T \in \mathcal{P}_{k-1}(T) \quad \forall T \in \mathcal{T}_h \} ,
	\\
	\label{eq:Qhstar-def}
	Q_h^* &= \{ q \in Q \,:\, q|_T \in \mathcal{P}_{k+1}(T) \quad \forall T \in \mathcal{T}_h \} ,
}
and $P_h, P_h^0$ are the $L^2$ orthogonal projections to $Q_h$, $Q_h^0$. For face $F$ or edge $E$ with integer $m\ge 0$, $P_F^m$ and $P_E^m$ are the $L^2$ orthogonal projections to $\mathcal{P}_m(F)$ and $\mathcal{P}_m(E)$.

\tred{For $q \in Q_h$, let $\bs{w}\in H^1(\Omega;\mathbb{R}^n)$ be a function satisfying \eqref{eq:assumption1}. It is well-known that the interpolation operator $\Pi_h:H^1(\Omega; \mathbb{R}^n) \rightarrow \bs{V}_h$ defined by the canonical degrees of freedom fulfills $\div \Pi_h \bs{w} = q$, $\Pi_h \bs{w} \cdot \bs{n} = 0$ on $\Gamma$, and $\|\Pi_h \bs{w} \|_0 + \| \div \Pi_h \bs{w} \|_0 \le C \| \bs{w} \|_1$ for some $C>0$ (cf. \cite{Brezzi-Fortin-book}). This is sufficient to prove that the pair $(\bs{V}_h, Q_h)$ satisfies } 
%
\algn{ \label{eq:inf-sup}
	\inf_{0 \not = q \in Q_h} \sup_{0 \not = \bs{v} \in \bs{V}_h} \frac{\LRp{ q, \div \bs{v}}_{\Omega} }{\| q \|_0 \| \bs{v} \|_{\bs{V}} } \ge C > 0
}
with $C>0$ independent of mesh sizes. 
In the rest of this paper our discussions are common for $\bs{V}_h = \bs{V}_h^{RTN}$ or $\bs{V}_h = \bs{V}_h^{BDM}$ unless we specify $\bs{V}_h$ in our statements.

The discrete problem of \eqref{eq:variational-eqs} with $\bs{V}_h \times Q_h$ is to find $(\bs{u}_h, p_h) \in \bs{V}_h \times Q_h$ such that 
\begin{subequations} \label{eq:discrete-eqs}
	\algn{ 
		\label{eq:discrete-eq1}
    	\LRp{\bs{u}_h,\bs{v}}_{\Omega} + \LRa{\alpha \bs{u}_h\cdot\bs{n},\bs{v}\cdot\bs{n}}_{\Gamma} - \LRp{p_h,\div \bs{v}}_{\Omega}  &= 0 & & \forall\bs{v} \in \bs{V}_h,
    	\\
		\label{eq:discrete-eq2}
	    \LRp{\div \bs{u}_h,q}_{\Omega} &= \LRp{f, q}_{\Omega} & & \forall q \in Q_h.
	}
\end{subequations}
An a priori error analysis for $(\bs{u}_h, p_h)$ 
is proved in \cite{Lee-Tan-Villa-Ghattas:2022}.

We finish this section by introducing a post-processed numerical solution of $p$, a variant of the post-processing in \cite{Stenberg:1991}.
Suppose that $(\bs{u}_h, p_h) \in \bs{V}_h \times Q_h$ is a solution of \eqref{eq:discrete-eqs}. Following the idea in \cite{Kim:2012,Cockburn-Zhang:2014}, 
a post-processed solution $p_h^* \in Q_h^*$ is defined by 
\begin{subequations} \label{eq:post-processing-eqs}
\algn{
	\label{eq:post-processing-eq1}
	\LRp{ \nabla p_h^*, \nabla q}_{T} &= -\LRp{ \ub_h, \nabla q}_T & & \forall q \in Q_h^* (T) , 
	\\
	\label{eq:post-processing-eq2}	
	\LRp{ p_h^*, q}_T &= \LRp{ P_h^0 p_h, q}_T & & \forall q \in Q_h^0 (T)
}
\end{subequations}
for $T \in \mathcal{T}_h$. 

Here we remark that $Q_h^*$ in \eqref{eq:Qhstar-def} contains all piecewise quadratic polynomials because $k \ge 1$. 
This will be used in the proof of Lemma~\ref{lemma:jump-mean-value-zero} which allows an elegant local efficiency proof of a posteriori error estimator.


\section{A posteriori error estimate}

In this section we define a posteriori error estimator and prove its reliability and local efficiency. Let $(\bs{u}_h, p_h)$ be a solution of \eqref{eq:discrete-eqs} and $p_h^*$ be the post-processed pressure defined in \eqref{eq:post-processing-eqs}. Let 
\algn{ \label{eq:l-def}
	m = k-1 \quad \text{ if }\bs{V}_h = \bs{V}_h^{RTN} \quad \text{ and } \quad m = k \quad \text{ if }\bs{V}_h = \bs{V}_h^{BDM} .
}
Then, a posteriori error estimator $\eta$ for $n=3$ is defined by 
\algn{ \label{eq:eta-def}
	\eta = \LRp{\sum_{T \in \mathcal{T}_h} \eta_T^2 + \sum_{F \in \mathcal{F}_h} \eta_F^2}^{\frac 12}
}
where 
\algn{
	\label{eq:eta-T-def}
	\eta_T &:= \| \bs{u}_h + \nabla p_h^* \|_{0,T},
	\\
	\label{eq:eta-F-def}
	\eta_F &:= 
		\begin{cases}
			h_F^{-1/2} \| \jump{p_h^*} \|_{0,F} & \text{ if } F \in \mathcal{F}_h^0
			\\
			\alpha^{-1/2} \| (I - P_F^m) \jump{p_h^*} \|_{0,F} & \text{ if } F \in \mathcal{F}_h^{\Gamma} 
		\end{cases} .
}
If $n=2$, $\eta$ is similarly defined by replacing $\eta_F$ by $\eta_E$ for edges in $\mathcal{E}_h^0$ and $\mathcal{E}_h^{\Gamma}$ with the same formula in \eqref{eq:eta-F-def}, so we omit the detailed definition.

\subsection{Reliability estimate}
We prove the reliability of the a posteriori error estimator $\eta$ in \eqref{eq:eta-def}.
The main result is the following. 
\begin{thm} Suppose that $(\ub, p)$, $(\bs{u}_h, p_h)$ are solutions of \eqref{eq:variational-eqs}, \eqref{eq:discrete-eqs}, and $\eta$ is defined by \eqref{eq:eta-def}. Then, there exists $C>0$ independent of mesh sizes and $\alpha$ in \eqref{eq:alpha-large} such that 
	\algn{ \label{eq:velocity-reliability}
		\tnorm{\bs{u} - \bs{u}_h } \le C \eta +  \frac {1}{\pi} \LRp{ \sum_{T \in \mathcal{T}_h} \osc(f,T)^2 }^{\frac 12}, \quad \osc(f,T):= {h_T} \| f - P_h f \|_{0,T}.
	}
\end{thm}
To prove this theorem, we split $\bs{u} - \bs{u}_h$ into two components which are orthogonal with an inner product given by the bilinear form of $\bs{u}$ and $\bs{v}$ in \eqref{eq:variational-eq1}. In the lemma below, we first show that there is an orthogonal decomposition of $\bs{u} - \bs{u}_h$. 
\begin{lemma} \label{lemma:orthogonal-decomposition}
	Suppose that $(\tilde{\ub}, \tilde{p})$ is the solution of 
	\begin{subequations}
		\label{eq:variational-project-eqs}
		\algn{ 
			\label{eq:variational-project-eq1}
    		\LRp{\tilde{\bs{u}},\bs{v}}_{\Omega}- \LRp{\tilde{p}, \div \bs{v}}_{\Omega} + \LRa{\alpha 	\tilde{\bs{u}}\cdot\bs{n},\bs{v}\cdot\bs{n}}_{\Gamma} &= 0 & &\forall\bs{v} \in \bs{V},
    		\\
			\label{eq:variational-project-eq2}
		    \LRp{\div \tilde{\bs{u}}, q}_{\Omega} &= \LRp{P_h f, q}_{\Omega} & &\forall q \in Q.
		}
	\end{subequations}
	Then, 
	\algn{ \label{eq:orthogonal-decomposition}
		\tnorm{\bs{u} - \bs{u}_h}^2 = \tnorm{\bs{u} - \tilde{\bs{u}}}^2 + \tnorm{ \tilde{\bs{u}} - \bs{u}_h }^2.
	}
\end{lemma}
\begin{proof}
	Then, it is easy to see that 
	\begin{subequations}
		\label{eq:variational-diff-eqs}
		\algn{ 
    		\label{eq:variational-diff-eq1}
			\LRp{\bs{u} - \tilde{\bs{u}},\bs{v}}_{\Omega}- \LRp{p - \tilde{p}, \div \bs{v}}_{\Omega} + \LRa{\alpha (\bs{u} - \tilde{\bs{u}}) \cdot\bs{n},\bs{v}\cdot\bs{n}}_{\Gamma} &= 0 ,
	    	\\
			\label{eq:variational-diff-eq2}    	
	    	\LRp{\div (\bs{u} - \tilde{\bs{u}}), q}_{\Omega} &= \LRp{f - P_h f, q}_{\Omega} 
		}
	\end{subequations}
	for all $\bs{v} \in \bs{V}$ and $q \in Q$.
	If $\bs{v} = \tilde{\bs{u}} - \bs{u}_h \in \bs{V}$, then 
	\algns{ 
		\LRp{\bs{u} - \tilde{\bs{u}},\tilde{\bs{u}} - \bs{u}_h}_{\Omega} + \LRa{\alpha (\bs{u} - \tilde{\bs{u}}) \cdot\bs{n},(\tilde{\bs{u}} - \bs{u}_h)\cdot\bs{n}}_{\Gamma} = 0 
	}
	because $\div (\tilde{\bs{u}} - \bs{u}_h) = 0$. Then, \eqref{eq:orthogonal-decomposition} follows from this orthogonality.
\end{proof}
As a consequence of \eqref{eq:orthogonal-decomposition}, it suffices to estimate $\tnorm{\bs{u} - \tilde{\bs{u}}}^2$ and $\tnorm{\bs{u}_h - \tilde{\bs{u}}}^2$ by the right-hand side terms in \eqref{eq:velocity-reliability}. We first estimate $\tnorm{\bs{u} - \tilde{\bs{u}}}$. 

\begin{lemma} 
	Suppose that $(\tilde{\ub}, \tilde{p})$ is defined as in Lemma~\ref{lemma:orthogonal-decomposition}. Then, 
	\algn{ \label{eq:u-utilde-estm}
		\tnorm{ \bs{u} - \tilde{\bs{u}}} \le \frac {1}{\pi} \LRp{\sum_{T \in \mathcal{T}_h} \osc(f,T)^2}^{\frac 12} .
	}
\end{lemma}
\begin{proof}
	Taking arbitrary test function vanishing near $\Gamma$ in \eqref{eq:variational-diff-eq1}, we obtain
	\algns{
		\bs{u} - \tilde{\bs{u}} = \nabla (p - \tilde{p}) \quad \text{ in } {L^2(\Omega_0; \mathbb{R}^n)}
	}
	for every open set $\Omega_0$ such that $\overline{\Omega_0} \subset \Omega \setminus \Gamma$. By the dominated convergence theorem, $\bs{u} - \tilde{\bs{u}} = \nabla(p - \tilde{p})$ on every $T \in \mathcal{T}_h$.
	If $\bs{v} = \bs{u} - \tilde{\bs{u}}$ in \tred{\eqref{eq:variational-diff-eq1}}, then 
	\algn{
		\label{eq:u-utilde-diff-estm1}
		\tnorm{ \bs{u} - \tilde{\bs{u}}}^2 = \LRp{p - \tilde{p}, f - P_h f }_{\Omega} 
	}
	by \eqref{eq:fault-weighted-L2} and \eqref{eq:variational-diff-eq2}. By the Cauchy--Schwarz and element-wise Poincare inequalities for mean-value zero functions (cf. \cite{Bebendorf:2003}), 
	\algn{
		\notag 
		\LRp{p - \tilde{p}, f - P_h f }_{\Omega} &= \LRp{p - \tilde{p} - P_h(p - \tilde{p}), f - P_h f }_{\Omega}
		\\
		\notag &\le \sum_{T \in \mathcal{T}_h} \frac{h_T}{\pi} \| \nabla (p - \tilde{p}) \|_{0,T} \| f - P_h f \|_{0,T}
		\\
		\notag
		&\le \frac 1{\pi} \| \nabla (p - \tilde{p}) \|_{0} \LRp{ \sum_{T \in \mathcal{T}_h} {h_T^2} \| f - P_h f \|_{0,T}^2 }^{\frac 12} 
		\\
		\label{eq:u-utilde-diff-estm2}
		&\le \frac {1}{\pi} \| \bs{u} - \tilde{\bs{u}} \|_{0} \LRp{\sum_{T \in \mathcal{T}_h} \osc(f,T)^2}^{\frac 12} . 
	}
	Combining \eqref{eq:u-utilde-diff-estm1} and \eqref{eq:u-utilde-diff-estm2}, we can obtain \eqref{eq:u-utilde-estm}.
\end{proof}

To estimate $\tnorm{\bs{u}_h - \tilde{\bs{u}}}$ by the a posteriori error estimator $\eta$, we need an auxiliary finite element space $S_h$.
We choose different $S_h$ for $\bs{V}_h = \bs{V}_h^{RTN}$ and $\bs{V}_h = \bs{V}_h^{BDM}$ and for $n=2,3$.

We first define $S_h$ for $\bs{V}_h^{BDM}$.
Note that $\bs{V}_h^{BDM} = \mathcal{P}_k \Lambda^{n-1}(\mathcal{T}_h)$ in the language of the finite element exterior calculus \tred{(\cite{AFW05-1,AFW06,AFW10})}. If $\bs{V}_h = \bs{V}_h^{BDM}$, then 
\algn{ \label{eq:Sh-def1}
	S_h = \mathcal{P}_{k+1} \Lambda^{n-2}(\mathcal{T}_h), 
}
which is the Lagrange finite element of degree $k+1$ if $n=2$ and is the N\'{e}d\'{e}lec edge element of the 2nd kind (\cite{Nedelec86}) with degree $k+1$ if $n=3$.

If $\bs{V}_h = \bs{V}_h^{RTN}$ and $n=2$, then 
\algn{ \label{eq:Sh-def2}
	S_h = \mathcal{P}_{k} \Lambda^{0}(\mathcal{T}_h) ,
}
the Lagrange finite element of degree $k$. 

If $\bs{V}_h = \bs{V}_h^{RTN}$ and $n=3$, then we define a new finite element space $S_h$ obtained by enriching $\mathcal{P}_{k} \Lambda^{1}(\mathcal{T}_h)$ with curl-free edge bubble functions which will be described below.

For an edge $E$ in a triangulation $\mathcal{T}_h$ for $n=3$, define $M_E$ by
\algn{
	M_E = \bigcup_{T \in \mathcal{T}_h, E \subset \pd T} T.
}
Note that each face $F \subset \partial M_E$ does not contain $E$. Denoting the barycentric coordinate which vanishes on $F$ by $\lambda_F$, we define $b_E$ by 
\algn{
	b_E = \prod_{F \subset \pd M_E} \lambda_F .
}
Since every tetrahedron in $M_E$ does not have more than two distinct faces which do not contain $E$, $b_E|_T \in \mathcal{P}_2(T)$ for every tetrahedron $T \subset M_E$. In the discussion below, $\mathcal{P}_{k+1}^{\perp}(E)$ is the space of polynomials with degree $(k+1)$ on $E$ which are orthogonal to all polynomials with degree $k$ on $E$, and $\bar{\mathcal{P}}_{k+1}^{\perp}(E)$ is the space of polynomials on $M_E$ which are constant on every plane perpendicular to $E$ and the restriction of the polynomials on $E$ are in $\mathcal{P}_{k+1}^{\perp}(E)$.

In the following lemma, for an edge $E$, $\bs{t}_E$ is a unit tangential vector of $E$ and $\frac{\pd}{\pd \bs{t}_E}$ is the derivative along the direction of $\bs{t}_E$.
\begin{lemma}
	For an edge $E$, 
	\algn{ \label{eq:edge-bubble-def}
		\mathcal{B}(E) = \spn \left\{ \nabla \LRp{\frac{\pd q}{\pd \bs{t}_E} b_E} \,:\, q \in \bar{\mathcal{P}}_{k+1}^{\perp}(E) \right\} .	
	}
	For a tetrahedron $T$ and a fixed face $\tilde{F} \subset \pd T$ let $S(T,\tilde{F})$ be 
	\algns{
		S(T, \tilde{F}) = \mathcal{P}_{k}(T; \mathbb{R}^3) + \oplus_{E \subset \pd \tilde{F}} \mathcal{B}(E) .
	}
	and a set of local degrees of freedom for $\tau \in S(T,\tilde{F})$ is given by
	\algn{
		\label{eq:Sh-dof1}
		\tau &\longmapsto \int_{E} \tau \cdot \bs{t}_E \tilde{q} \,dl	& &  \forall \tilde{q} \in \mathcal{P}_k(E) \text{ if } E \not \subset \pd \tilde{F},
		\\
		\label{eq:Sh-dof2}
		\tau &\longmapsto \int_{E} \tau \cdot \bs{t}_E \tilde{q} \,dl	& & \forall \tilde{q} \in \mathcal{P}_{k+1}(E) \text{ if } E \subset \pd \tilde{F},
		\\
		\label{eq:Sh-dof3}
		\tau &\longmapsto \int_{F} (\tau \times \bs{n}_F) \cdot \bs{q} \,ds & & \forall \bs{q} \in \bs{V}_{k-1}^{RTN}(F) \text{ if } k \ge 2,
		\\
		\label{eq:Sh-dof4}
		\tau &\longmapsto \int_{T} \tau \cdot \xi \,dx	 & & \forall \xi \in \bs{V}_{k-2}^{RTN}(T) \text{ if } k \ge 3.
	}
	Then, $\tau \in S(T, \tilde{F})$ is uniquely determined by \eqref{eq:Sh-dof1}, \eqref{eq:Sh-dof2}, \eqref{eq:Sh-dof3}, \eqref{eq:Sh-dof4}.
\end{lemma}
\begin{proof}
Suppose that $\tau \in S(T, \tilde{F})$ and all DOFs of $\tau$ given by \eqref{eq:Sh-dof1}, \eqref{eq:Sh-dof2}, \eqref{eq:Sh-dof3}, \eqref{eq:Sh-dof4} vanish. Let $E_i$, $i=1,2,3$ be the edges of $\tilde{F}$. We can write $\tau$ as $\tau = \sum_{i=0}^3 \tau_i$ with $\tau_0 \in \mathcal{P}_{k}(T; \mathbb{R}^3)$ and 
\algn{
\tau_i =  \nabla\LRp{\frac{\pd q_i}{\pd \bs{t}_{E_i}} b_{E_i}} \in \mathcal{B}(E_i), \quad q_i \in \bar{\mathcal{P}}_{k+1}^{\perp}(E_i), i=1,2,3. 
}
%
By the vanishing DOFs assumption, for $E_i \subset \pd \tilde{F}$, 
\algns{
	\int_{E_i} \tau \cdot \bs{t}_{E_i} \tilde{q}\,dl = \int_{E_i} \tau_0 \cdot \bs{t}_{E_i} \tilde{q}\,dl + \sum_{j=1}^3 \int_{E_i} \tau_j \cdot \bs{t}_{E_i} \tilde{q}\,dl  = 0
}
for all $\tilde{q} \in \mathcal{P}_{k+1}(E_i)$. By the definition \eqref{eq:edge-bubble-def}, $\tau_i \cdot \bs{t}_{E_j}|_{E_j} = 0$ if $j \not = i$, so we get 
\algn{ \label{eq:edge-vanish-1}
	\int_{E_i} \tau_0 \cdot \bs{t}_{E_i} \tilde{q}\,dl + \int_{E_i} \tau_i \cdot \bs{t}_{E_i} \tilde{q}\,dl  = 0 .
}
Consider the decomposition $\tilde{q} = \tilde{q}_0 + \tilde{q}_1 \in \mathcal{P}_k(E_i) \oplus \mathcal{P}_{k+1}^{\perp}(E_i)$. Then,  
\algn{ \label{eq:edge-ortho1}
	\int_{E_i} \tau_0 \cdot \bs{t}_{E_i} \tilde{q}_1\,dl = 0 . 
}
%
Taking the integration by parts twice gives 
\algn{
	\notag 
	\int_{E_i} \tau_i \cdot \bs{t}_{E_i} \tilde{q}_0\,dl &= \int_{E_i} \frac{\pd}{\pd \bs{t}_{E_i}} \LRp{\frac{\pd q_i}{\pd \bs{t}_{E_i}} b_{E_i}} \tilde{q}_0 \,dl 	
	\\
	\label{eq:edge-ortho2}
	&= - \int_{E_i} \LRp{\frac{\pd q_i}{\pd \bs{t}_{E_i}} b_{E_i}} \frac{\pd \tilde{q}_0}{\pd \bs{t}_{E_i}} \,dl
	\\
	\notag
	&= \int_{E_i} q_i \frac{\pd}{\pd \bs{t}_{E_i}} \LRp{ b_{E_i} \frac{\pd \tilde{q}_0}{\pd \bs{t}_{E_i}} } \,dl 
	\\
	\notag
	&=0
}
where the last identity follows from $\tilde{q}_0 \in \mathcal{P}_k({E_i})$, $b_{E_i}|_{E_i} \in \mathcal{P}_2({E_i})$, and $q_i|_{E_i} \in \mathcal{P}_{k+1}^{\perp}({E_i})$. Therefore, \eqref{eq:edge-vanish-1} is reduced to 
\algns{
	\int_{E_i} \tau_0 \cdot \bs{t}_{E_i} \tilde{q}_0\,dl + \int_{E_i} \tau_i \cdot \bs{t}_{E_i} \tilde{q}_1\,dl  = 0 .
}
If $\tilde{q}_0 = \tau_0 \cdot \bs{t}_{E_i}$, $\tilde{q}_1 = - q_i$ in this formula, and use the identity 
\algns{
	\int_{E_i} \tau_i \cdot \bs{t}_{E_i} \tilde{q}_1\,dl &= - \int_{E_i} \frac{\pd}{\pd \bs{t}_{E_i}} \LRp{\frac{\pd q_i}{\pd \bs{t}_{E_i}} b_{E_i}} {q}_i \,dl = \int_{E_i} \LRp{\frac{\pd q_i}{\pd \bs{t}_{E_i}}}^2 b_{E_i} \,dl 	,
}
then 
\algns{
	\int_{E_i} \tau_0 \cdot \bs{t}_{E_i} \tilde{q}_0\,dl + \int_{E_i} \tau_i \cdot \bs{t}_{E_i} \tilde{q}_1\,dl  &= \int_{E_i} (\tau_0 \cdot \bs{t}_{E_i})^2 \,dl + \int_{E_i} \LRp{\frac{\pd q_i}{\pd \bs{t}_{E_i}}}^2 b_{E_i} \,dl 
	\\
	&= 0.
}
Since $b_{E_i} >0$ on $E_i$, $\frac{\pd q_i}{\pd \bs{t}_{E_i}}|_{E_i} = 0$, so $q_i|_{E_i}$ is constant. Furthermore, $q_i|_{E_i} \in \mathcal{P}_{k+1}^{\perp}(E_i)$, so $q_i = 0$. As a consequence, $\tau_i = 0$ for $i=1,2,3$. Then, $\tau =0$ follows by a standard unisolvency proof of the N\'ed\'elec edge elements of the 2nd kind.
\end{proof}

We now define $S_h$ for $n=3$ and $\bs{V}_h = \bs{V}_h^{RTN}$. The enriched $H(\curl)$ element $S_h$ with the shape functions 
\algn{ \label{eq:Sh-def3}
	S_h = \mathcal{P}_k(\mathcal{T}_h; \mathbb{R}^3) + \bigoplus_{E \in \mathcal{E}_h^{\Gamma} }\mathcal{B}(E)
}
and the global degrees of freedom \eqref{eq:Sh-dof1} for $E \in \mathcal{E}_h \setminus \mathcal{E}_h^{\Gamma}$, 
\eqref{eq:Sh-dof2} for $E \in \mathcal{E}_h^{\Gamma}$, \eqref{eq:Sh-dof3} for $F \in \mathcal{F}_h$, \eqref{eq:Sh-dof4} for $E \in \mathcal{T}_h$. 

For $S_h$ defined by \eqref{eq:Sh-def1}, \eqref{eq:Sh-def2}, \eqref{eq:Sh-def3} depending on $n$ and $\bs{V}_h$, we can check that 
\algn{ \label{eq:curl-Sh-inclusion}
	\curl S_h \subset \bs{V}_h .
}
%
%

Here we show existence of an appropriate interpolation.
\begin{lemma} \label{lemma:Ih-def}
	Suppose that $\Psi \in H^1(\Omega)$ and $\curl \Psi \cdot \bs{n}|_{\Gamma} \in L^2(\Gamma)$. 
	Let $S_h$ be defined by \eqref{eq:Sh-def1}, \eqref{eq:Sh-def2}, \eqref{eq:Sh-def3} depending on $n$ and $\bs{V}_h$, and recall $l$ defined in \eqref{eq:l-def}.
	%
	Then, there exists $I_h \Psi \in S_h$ such that 
	\algn{
		\label{eq:Ih-property0}
		\tnorm{ \curl I_h \Psi} &\le C (\| \Psi \|_{1} + \| \alpha^{1/2} \curl \Psi \cdot \bs{n} \|_{0,\Gamma}), 
		\\
		\label{eq:Ih-property1}
		\| \Psi - I_h \Psi \|_{0} &\le C (\| \Psi \|_{1} + \| \alpha \curl \Psi \cdot \bs{n} \|_{0,\Gamma}), 
		\\
		\label{eq:Ih-property2}
		\curl I_h \Psi \cdot \bs{n}|_E &= P_E^{m} (\curl \Psi \cdot \bs{n}|_E) , \quad \forall E \in \mathcal{E}_h^{\Gamma}, \text{ if } n=2,
		\\
		\label{eq:Ih-property3}
		\curl I_h \Psi \cdot \bs{n}|_F &= P_F^{m} (\curl \Psi \cdot \bs{n}|_F) , \quad \forall F \in \mathcal{F}_h^{\Gamma}, 	\text{ if } n=3,
	}
	and for $E \in \mathcal{E}_h^{\Gamma}$, 
	\algn{ \label{eq:Ih-property4}
		\begin{cases}
			(I_h \Psi - \Psi)(v) &= 0 \quad \text{ if } n=2, v \in \pd E ,
			\\
			\int_E (I_h \Psi - \Psi) \cdot \bs{t}_E \tilde{q} \,dl &= 0 \quad \text{ if } n=3, \tilde{q} \in \mathcal{P}_{k+1}(E) .
		\end{cases}
	}
\end{lemma}
\begin{proof}
	Suppose that $n=2$. In this case, $\curl \Psi \cdot \bs{n}_{\Gamma} \in L^2(\Gamma)$ means that the tangential derivative of $\Psi$ along $\Gamma$ is in $L^2(\Gamma)$ (cf. \eqref{eq:IBP-2d-2} in Appendix). Since $\Psi|_{\Gamma} \in L^2(\Gamma)$, we have $\Psi|_{\Gamma} \in H^1(\Gamma)$. Let $\mathcal{N}_h$ be the set of vertex nodes in $\overline{\Omega}$ which determines the degrees of freedom of $S_h$, a Lagrange finite element. By the Sobolev embedding on the 1-dimensional submanifold $\Gamma$, vertex evaluation of $\Psi$ on the nodes on $\Gamma$ is well-defined. On $E \in \mathcal{E}_h^{\Gamma}$, we define $I_h^{\Gamma} \Psi|_E$ by 
	\algn{
		\label{eq:Ih-Gamma-2d-1}
		I_h^{\Gamma} \Psi(v) &= \Psi(v) & & v \in \pd E, 
		\\
		\label{eq:Ih-Gamma-2d-2}
		\int_E I_h^{\Gamma} \Psi  \frac{\pd}{\pd \bs{t}_E} \tilde{q} \, dl &= \int_E \Psi \frac{\pd}{\pd \bs{t}_E} \tilde{q} \,dl & & \forall \tilde{q} \in \mathcal{P}_{m}(E), \tilde{q} \perp \mathcal{P}_0(E) 
	}
	for $m$ defined in \eqref{eq:l-def}, and  
	%
	%
	$I_h^{\Gamma} \Psi(v) = 0$ for $v \in \mathcal{N}_h \setminus \Gamma$. Then, by a standard scaling argument, 
	\algns{
		\| I_h^{\Gamma} \Psi \|_0 + \| \curl I_h^{\Gamma} \Psi \|_0 \le C (\| \Psi \|_1 + \| \curl \Psi \cdot \bs{n} \|_{0, \Gamma})	
	}
	with $C>0$ independent of mesh sizes. 
	Let $I_h^{SZ}$ be a Scott--Zhang interpolation (cf. \cite{Scott-Zhang:1990}) which takes $\Gamma$ as a vanishing interface and satisfies
	\algns{
		\| I_h^{SZ} \Phi \|_0 + \| \curl I_h^{SZ} \Phi \|_0 \le C \| \Phi \|_1 
	}
	for $\Psi \in H^1(\Omega)$. 	If we define $I_h \Psi$ by 
	\algn{ \label{eq:Ih-def}
		I_h \Psi = I_h^{SZ} ( \Psi - I_h^{\Gamma} \Psi) + I_h^{\Gamma} \Psi , 
	}
	\tred{then} $I_h \Psi$ is bounded by $\| \Psi \|_1 + \| \curl \Psi \cdot \bs{n} \|_{0,\Gamma}$. 

	We now check \eqref{eq:Ih-property0}, \eqref{eq:Ih-property1}, \eqref{eq:Ih-property2}, \eqref{eq:Ih-property4}. First, \eqref{eq:Ih-property4} is a consequence of \eqref{eq:Ih-Gamma-2d-1} and \eqref{eq:Ih-def}. By \eqref{eq:Ih-Gamma-2d-1}, \eqref{eq:Ih-Gamma-2d-2}, $\curl I_h \Psi \cdot \bs{n}|_{\Gamma} = I_h^{\Gamma} \Psi \cdot \bs{n}|_{\Gamma}$, and the integration by parts on every $E \in \mathcal{E}_h^{\Gamma}$, \eqref{eq:Ih-property2} follows.

Furthermore, if $\Psi \in S_h$, then $I_h^{\Gamma} \Psi = \Psi|_{\Gamma}$, so $I_h$ is the identity map on $S_h$ because $I_h^{SZ}$ is the identity map for the elements in $S_h$ which vanish on $\Gamma$. By the Bramble--Hilbert lemma, $\| \Psi - I_h \Psi \|_0 \le Ch (\| \Psi \|_1 + \| \curl \Psi \cdot \bs{n} \|_{0,\Gamma} )$. 

	
	Suppose that $n=3$. 
	First, $\Psi \times \bs{n}_{\Gamma} \in L^2(\Gamma; \mathbb{R}^2)$ and $\curl \Psi \cdot \bs{n}_{\Gamma} \in L^2(\Gamma)$ imply that the tangential component of $\Psi$ on $\Gamma$ is in the rotated $H(\div)$ space on $\Gamma$ (cf. \eqref{eq:3d-curl-identity} in Appendix). Since $\Psi|_{\Gamma} \in H^{s}(\Gamma)$ with $s>0$ as a trace of $H^1(\Omega; \mathbb{R}^3)$, $\Psi \times \bs{n}_{\Gamma} \in L^r(\Gamma)$ for $r>2$ by Sobolev embedding, so 	%
	\algn{
		\label{eq:Ih-Gamma-3d-1}
		\int_E I_h^{\Gamma} \Psi \cdot \bs{t}_E \tilde{q} \,dl &= \int_E \Psi \cdot \bs{t}_E \tilde{q} \,dl, & & E \in \mathcal{E}_h^{\Gamma}, \tilde{q} \in \mathcal{P}_{k+1}(E),
		\\
		\label{eq:Ih-Gamma-3d-2}
		\int_F I_h^{\Gamma}\Psi \times \bs{n}_F \cdot \xi \, ds &= \int_F (\Psi \times \bs{n}_F) \cdot \xi \,ds, & & F \in \mathcal{F}_h^{\Gamma}, \xi \in \mathcal{P}_{m-2}(F; \mathbb{R}^2)
	}
	are well-defined (cf. \cite{Brezzi-Fortin-book}). $I_h \Psi \in S_h$ is defined by taking zeros for all other degrees of freedom, and $\| I_h^{\Gamma} \Psi\|_0$ is bounded by $\| \Psi \|_1 + \| \curl \Psi \cdot \bs{n} \|_{0,\Gamma}$. There exists a Scott--Zhang type interpolation $I_h^{SZ}$ for $H(\curl)$ elements (see \cite{Gawlik-Holst-Licht:2021}) with vanishing interface $\Gamma$, so define $I_h$ as in \eqref{eq:Ih-def}. By an argument similar to the proof for $n=2$, \eqref{eq:Ih-property1} can be obtained, and \eqref{eq:Ih-property4} follows from \eqref{eq:Ih-Gamma-3d-1} and \eqref{eq:Ih-def}. Finally, \eqref{eq:Ih-property3} follows from \eqref{eq:Ih-Gamma-3d-1}, \eqref{eq:Ih-Gamma-3d-2}, and the integration by parts on every face $F \in \mathcal{F}_h^{\Gamma}$ (see \eqref{eq:IBP-3d-1} in Appendix for details).
	%
\end{proof}
We now prove a reliability estimate of $\tnorm{\bs{u}_h - \tilde{\bs{u}} }$.
\begin{thm} Suppose that $(\tilde{\ub}, \tilde{p})$ is defined as in Lemma~\ref{lemma:orthogonal-decomposition}. Then, there exists $C>0$ independent of mesh sizes such that 
	\algn{ \label{eq:reliability2}
		\tnorm{\bs{u}_h - \tilde{\bs{u}} } \le C \eta .
	}
\end{thm}
\begin{proof}
	We now estimate $\tnorm{\bs{u}_h - \tilde{\bs{u}}}^2$. First, since $\Omega$ is homologically trivial, there exists $\Psi \in H^1(\Omega, \mathbb{R}^n)$ such that $\bs{u}_h - \tilde{\bs{u}} = \curl \Psi$ and 
\algns{
	\| \Psi \|_{1} &\le C \| \bs{u}_h - \tilde{\bs{u}} \|_{0},
	\\
	\LRa{ \alpha \curl \Psi \cdot \bs{n}, \curl \Psi \cdot \bs{n}}_{\Gamma} &= \LRa{ \alpha (\bs{u}_h - \tilde{\bs{u}}) \cdot \bs{n} , (\bs{u}_h - \tilde{\bs{u}}) \cdot \bs{n} }_{\Gamma} 
}
because $\div (\bs{u}_h - \tilde{\bs{u}}) = 0$. Thus, 
\algn{ \label{eq:uh-tildeu-diff}
	\tnorm{\bs{u}_h - \tilde{\bs{u}}}^2 = \LRp{\bs{u}_h - \tilde{\bs{u}}, \curl \Psi}_{\Omega} + \LRa{\alpha (\bs{u}_h - \tilde{\bs{u}}) \cdot\bs{n},\curl \Psi \cdot\bs{n}}_{\Gamma} .
}
Since 
\begin{subequations}
	\label{eq:variational-project2-eqs}
	\algn{ 
    	\label{eq:variational-project2-eq1}
		\LRp{\bs{u}_h - \tilde{\bs{u}}, \bs{v}}_{\Omega}- \LRp{p_h - \tilde{p}, \div \bs{v}}_{\Omega} + \LRa{\alpha (\bs{u}_h - \tilde{\bs{u}}) \cdot\bs{n},\bs{v}\cdot\bs{n}}_{\Gamma} &= 0 
	}
\end{subequations}
for all $\bs{v} \in \bs{V}_h$, we have 
\algn{ \label{eq:Ih-orthogonality}
	\LRp{\bs{u}_h - \tilde{\bs{u}}, \curl I_h \Psi }_{\Omega} + \LRa{\alpha (\bs{u}_h - \tilde{\bs{u}}) \cdot\bs{n},\curl I_h \Psi \cdot\bs{n}}_{\Gamma} = 0 
}
for $I_h$ in Lemma~\ref{lemma:Ih-def} because of $I_h \Psi \in S_h$ and \eqref{eq:curl-Sh-inclusion}. Applying \eqref{eq:Ih-orthogonality} to \eqref{eq:uh-tildeu-diff}, 
\algn{ \label{eq:uh-tildeu-diff2}
	\tnorm{\bs{u}_h - \tilde{\bs{u}}}^2 &= \LRp{\bs{u}_h - \tilde{\bs{u}}, \curl (\Psi - I_h \Psi)}_{\Omega} 
	\\
	\notag &\quad + \LRa{\alpha (\bs{u}_h - \tilde{\bs{u}}) \cdot\bs{n},\curl (\Psi - I_h \Psi) \cdot\bs{n}}_{\Gamma} .
}
Since $\tilde{\bs{u}} = -\nabla \tilde{p}$ and $\alpha \tilde{\bs{u}} \cdot \bs{n} = \jump{\tilde{p}}$ on $\Gamma$, we can further obtain
\algn{ \label{eq:uh-tildeu-diff3}
	\tnorm{\bs{u}_h - \tilde{\bs{u}}}^2 &= \LRp{\bs{u}_h , \curl (\Psi - I_h \Psi)}_{\Omega} 
	\\
	\notag &\quad + \LRa{\alpha \bs{u}_h \cdot\bs{n},\curl (\Psi - I_h \Psi) \cdot\bs{n}}_{\Gamma} 
}
by the integration by parts. For $m$ defined in \eqref{eq:l-def}, $(\alpha \bs{u}_h \cdot \bs{n})|_E \in \mathcal{P}_m(E)$ for $E \in \mathcal{E}_h^{\Gamma}$ if $n=2$ and $(\alpha \bs{u}_h \cdot \bs{n})|_F \in \mathcal{P}_m(F)$ for $F \in \mathcal{F}_h^{\Gamma}$ if $n=3$. Therefore, \eqref{eq:uh-tildeu-diff3} is reduced to 
\algn{ \label{eq:uh-tildeu-diff3-reduced}
	\tnorm{\bs{u}_h - \tilde{\bs{u}}}^2 &= \LRp{\bs{u}_h , \curl (\Psi - I_h \Psi)}_{\Omega} 
}
by \eqref{eq:Ih-property2} and \eqref{eq:Ih-property3}. 

If $n=2$, a simple algebra and triangle-wise integration by parts give 
\algn{ 
	\notag \tnorm{\bs{u}_h - \tilde{\bs{u}}}^2 &= \LRp{\bs{u}_h + \nabla p_h^* , \curl (\Psi - I_h \Psi)}_{\Omega} - \LRp{ \nabla p_h^* , \curl (\Psi - I_h \Psi)}_{\Omega} 
	\\
	\label{eq:uh-tildeu-diff4-2d}	&=  \LRp{\bs{u}_h + \nabla p_h^* , \curl (\Psi - I_h \Psi)}_{\Omega} - \sum_{T \in \mathcal{T}_h} \LRa{ \nabla p_h^* \cdot \bs{t}_{\pd T} , \Psi - I_h \Psi}_{\pd T } .
}
By edge-wise integration by parts using $(\Psi - I_h\Psi)(v) = 0$ for every endpoint $v$ of edges $E \in \mathcal{E}_h^{\Gamma}$, 
\algn{ 
	\notag 
	&\sum_{T \in \mathcal{T}_h} \LRa{ \nabla p_h^* \cdot \bs{t}_{\pd T}, \Psi - I_h \Psi }_{\pd T}
	\\
	\notag 
	&\quad = \pm \sum_{E \in \mathcal{E}_h} \LRa{ \nabla \jump{p_h^*} \cdot \bs{t}_{E}, \Psi - I_h \Psi }_{E}
	\\
	\label{eq:uh-tildeu-diff5-2d}	
	&\quad = \pm \sum_{E \in \mathcal{E}_h^{\Gamma}} \LRa{\jump{p_h^*} , \curl(\Psi - I_h \Psi) \cdot \bs{n} }_{E} \pm \sum_{E \in \mathcal{E}_h \setminus \mathcal{E}_h^{\Gamma}} \LRa{ \nabla \jump{p_h^*} \cdot \bs{t}_{E}, \Psi - I_h \Psi }_{E} 
	\\
	\notag
	&\quad = \pm \sum_{E \in \mathcal{E}_h^{\Gamma}} \LRa{ (I - P_E^m) \jump{p_h^*}, \curl (\Psi - I_h \Psi) \cdot \bs{n}}_{E} 
	\\
	\notag 
	&\qquad \pm \sum_{E \in \mathcal{E}_h \setminus \mathcal{E}_h^{\Gamma}} \LRa{ \nabla \jump{p_h^*} \cdot \bs{t}_{E}, \Psi - I_h \Psi }_{E}  .
}
Here we use $\pm$ due to sign ambiguity of the definitions of $\jump{p_h^*}$ and $\bs{t}_E$. However, we will use the Cauchy--Schwarz inequality to estimate the terms that this ambiguous sign is involved, so the exact sign is not important in the rest of discussions. By this and \eqref{eq:uh-tildeu-diff4-2d}, 
\algn{ 
	\notag \tnorm{\bs{u}_h - \tilde{\bs{u}}}^2 &=  \LRp{\bs{u}_h + \nabla p_h^* , \curl (\Psi - I_h \Psi)}_{\Omega} \pm \sum_{E \in \mathcal{E}_h \setminus \mathcal{E}_h^{\Gamma} } \LRa{ \nabla \jump{p_h^*} \cdot \bs{t}_E, \Psi - I_h \Psi}_{E } 
	\\
	\label{eq:uh-tildeu-diff6-2d}	
	&\quad \pm \sum_{E \in \mathcal{E}_h^{\Gamma} } \LRa{ (I-P_E^m) \jump{p_h^*} , \curl(\Psi - I_h \Psi) \cdot \bs{n} }_{E} 
	\\
	\notag &=: I_{2,a} + I_{2,b} + I_{2,c} .
}
By the Cauchy--Schwarz inequality 
\algn{
	\notag |I_{2,a}| &\le \| \bs{u}_h + \nabla p_h^* \|_0 \| \curl( \Psi - I_h \Psi) \|_0
	\\
	\label{eq:I2a-estm} &\le C\LRp{\sum_{T \in \mathcal{T}_h} \eta_T^2}^{\frac 12} \LRp{ \| \Psi \|_1 + \| \curl \Psi \cdot \bs{n} \|_{L^2(\Gamma)} }
	\\
	\notag &\le C\LRp{\sum_{T \in \mathcal{T}_h} \eta_T^2}^{\frac 12} \tnorm{\bs{u}_h - \tilde{\bs{u}} }.
}
By element-wise inverse inequality and an approximation property of $\Psi - I_h \Psi$, 
\algn{ 
	\notag |I_{2,b}| &\le \sum_{E \in \mathcal{E}_h^0 \setminus \mathcal{E}_h^{\Gamma} } \left| \LRa{ \nabla \jump{p_h^*} \cdot \bs{t}_E , \Psi - I_h \Psi}_{E}  \right|
	\\
	\label{eq:I2b-estm}	
	&\le C \LRp{ \sum_{E \in \mathcal{E}_h \setminus \mathcal{E}_h^{\Gamma}} h_E^{-1} \| \jump{p_h^*} \|_{0,E}^2}^{\frac 12}  \| \Psi \| _{1}  	
	\\
	\notag &\le C \LRp{ \sum_{E \in \mathcal{E}_h \setminus \mathcal{E}_h^{\Gamma}} \eta_E^2}^{\frac 12} \tnorm{\bs{u}_h - \tilde{\bs{u}} }.
}
For $I_{2,c}$, 
%
%
%
\algn{
	\notag 
	&\left| \LRa{(I-\tred{P_E^m}) \jump{p_h^*}, \curl(\Psi - I_h \Psi)\cdot \bs{n} }_E \right|
	\\
	\label{eq:I2c-estm}
	&\quad \le \alpha^{-1/2} \| (I-P_E^m) \jump{p_h^*} \|_{0,E} \| \alpha^{1/2} \curl (\Psi - I_h \Psi) \cdot \bs{n} \|_{0,E} 
	\\
	\notag
	&\quad \le 2 \alpha^{-1/2} \| (I-P_E^m) \jump{p_h^*} \|_{0,E} \| \alpha^{1/2} \curl \Psi \cdot \bs{n} \|_{0,E} .
}
%
Combining \eqref{eq:uh-tildeu-diff6-2d}, \eqref{eq:I2a-estm}, \eqref{eq:I2b-estm}, \eqref{eq:I2c-estm}, we obtain %
\algns{ 
	\tnorm{\bs{u}_h - \tilde{\bs{u}} } \le C \eta . 
}

If $n=3$, then 
\algn{ 
	\notag \tnorm{\bs{u}_h - \tilde{\bs{u}}}^2 &= \LRp{\bs{u}_h + \nabla p_h^* , \curl (\Psi - I_h \Psi)}_{\Omega} - \LRp{ \nabla p_h^* , \curl (\Psi - I_h \Psi)}_{\Omega} 
	\\
	\label{eq:uh-tildeu-diff4-3d}	&=  \LRp{\bs{u}_h + \nabla p_h^* , \curl (\Psi - I_h \Psi)}_{\Omega} - \sum_{T \in \mathcal{T}_h} \LRa{ \nabla p_h^* , \bs{n} \times (\Psi - I_h \Psi)}_{\pd T } 
}
by tetrahedron-wise integration by parts. 
By face-wise integration by parts \eqref{eq:IBP-3d-1} and by the property \eqref{eq:Ih-property4}, 
\algn{
	\notag
	\sum_{T \in \mathcal{T}_h} \LRa{ \nabla p_h^* , \bs{n} \times (\Psi - I_h \Psi)}_{\pd T \cap \Gamma} &= \pm \sum_{F \in \mathcal{F}_h^{\Gamma}} \LRa{ \nabla \jump{p_h^*} , \bs{n} \times (\Psi - I_h \Psi)}_{F} 
	\\
	\label{eq:face-wise-IBP}
	&=\pm \sum_{F \in \mathcal{F}_h^{\Gamma}} \LRa{ \jump{p_h^*} , \curl (\Psi - I_h \Psi) \cdot \bs{n}}_{F} 
	\\
	\notag
	&= \pm \sum_{F \in \mathcal{F}_h^{\Gamma}} \LRa{ (I-P_F^m) \jump{p_h^*} , \curl (\Psi - I_h \Psi) \cdot \bs{n}}_{F} .	
}
We can proceed using this and \eqref{eq:uh-tildeu-diff4-3d} to obtain
\algn{ 
	\notag 
	\tnorm{\bs{u}_h - \tilde{\bs{u}}}^2 &=  \LRp{\bs{u}_h + \nabla p_h^* , \curl (\Psi - I_h \Psi)}_{\Omega} \pm \sum_{F \in \mathcal{F}_h \setminus \mathcal{F}_h^{\Gamma} } \LRa{ \nabla \jump{p_h^*} \times \bs{n}, (\Psi - I_h \Psi)}_{F} 
	\\
	\label{eq:uh-tildeu-diff5-3d}	
	&\quad \pm \sum_{F \in \mathcal{F}_h^{\Gamma} } \LRa{ (I-P_F^m) \jump{p_h^*} ,\curl (\Psi - I_h \Psi) \cdot\bs{n}}_{F} 
	\\
	\notag
	&=: I_{3,a} + I_{3,b} + I_{3,c} .
}
By the arguments which are completely similar to the ones for $n=2$, we can obtain 
\algn{
	\label{eq:I3a-estm}
	|I_{3,a}| &\le C\LRp{\sum_{T \in \mathcal{T}_h} \|  \bs{u}_h + \nabla p_h^* \|_{0,T}^2 }^{\frac 12} \| \bs{u}_h - \tilde{\bs{u}} \|_0,
	\\
	\label{eq:I3b-estm}	
	|I_{3,b}| &\le C \LRp{ \sum_{F \in \mathcal{F}_h \setminus \mathcal{F}_h^{\Gamma}  } h_F^{-1} \| \jump{p_h^*} \|_{0,F}^2 }^{\frac 12} \| \bs{u}_h - \tilde{\bs{u}} \|_0,
	\\
	\label{eq:I3c-estm}
	|I_{3,c}| &\le C \LRp{ \sum_{F \in \mathcal{F}_h^{\Gamma}} \alpha^{-1} \| (I-P_F^m) \jump{p_h^*} \|_{0,F}^2 }^{\frac 12} 
	\| \alpha^{1/2} (\bs{u}_h - \tilde{\bs{u}}) \cdot \bs{n} \|_{0,\Gamma}.
}
Applying these estimates to \eqref{eq:uh-tildeu-diff5-3d}, we can obtain \eqref{eq:reliability2}.
\end{proof}
\begin{rmk}
	The new $S_h$ space in \eqref{eq:Sh-def3} for $n=3$, $\bs{V}_h = \bs{V}_h^{RTN}$, is necessary for \eqref{eq:face-wise-IBP}.
	More precisely, the new $S_h$ allows an interpolation $I_h$ satisfying \eqref{eq:Ih-property4} for $n=3$, $\bs{V}_h = \bs{V}_h^{RTN}$, which is necessary for the first equality in \eqref{eq:face-wise-IBP}.
\end{rmk}

\subsection{Local Efficiency}

In this subsection we show local efficiency of the a posteriori error estimator.
We give a detailed proof for $n=3$ because the two-dimensional case is almost same.

We prove a lemma employing the techniques in \cite{Vohralik:2007,Kim:2012}. 
\begin{lemma} \label{lemma:jump-mean-value-zero}
	For $F \in \mathcal{F}_h \setminus \mathcal{F}_h^{\Gamma}$, 
	\algn{ 
		\label{eq:ph-star-meanzero}
		\int_F \jump{p_h^*} \,ds &= 0 \quad \text{ if } F \in \mathcal{F}_h \setminus \mathcal{F}_h^{\Gamma},
		\\
		\label{eq:flux-meanzero}
		\int_F (\alpha \bs{u}_h \cdot \bs{n} - \jump{p_h^*}) \,ds &= 0 \quad \text{ if } F \in \mathcal{F}_h^{\Gamma} .
	}
\end{lemma}
\begin{proof}
	For $F \in \mathcal{F}_h$ let $\bs{v}_F$ be a test function in the lowest order Raviart--Thomas finite element such that 
	\algns{
		\bs{v}_F \cdot \bs{n}|_{F'} = 
			\begin{cases} 
				1	& \text{ if } F' = F 
				\\
				0	& \text{ if } F' \not = F
			\end{cases} \quad \text{ for } F' \in \mathcal{F}_h .
	}
	If we take this $\bs{v}_F$ in \eqref{eq:discrete-eq1} and $F \in \mathcal{F}_h \setminus \mathcal{F}_h^{\Gamma}$, then 
	\algns{
    	\LRp{\bs{u}_h,\bs{v}_F}_{\Omega} - \LRp{p_h,\div \bs{v}_F}_{\Omega}  &= 0 .
	}
	Since $\div \bs{v}_F$ is a piecewise constant function and the mean-values of $p_h^*$, $p_h$ on every tetrahedron are same, we have
	\algns{
    	\LRp{\bs{u}_h,\bs{v}_F}_{\Omega} - \LRp{p_h^*,\div \bs{v}_F}_{\Omega}  = 0 ,
	}
	and the element-wise integration by parts gives
	\algns{
    	\LRp{\bs{u}_h,\bs{v}_F}_{\Omega} - \LRa{ \jump{p_h^*}, 1 }_F + \LRp{\nabla p_h^*, \bs{v}_F}_{\Omega}  = 0 .
	}
	By the forms of shape functions of the lowest order Raviart--Thomas elements (\eqref{eq:RTN-local} with $k=1$), $\curl \bs{v}_F = 0$ on every $T \in \mathcal{T}_h$, so there exists $\phi_T \in \mathcal{P}_2(T)$ for every $T \subset \supp \bs{v}_F$ such that $\bs{v}_F|_T = \nabla \phi_T$. By this, \eqref{eq:post-processing-eq1}, and the above equation, we have 
	\algns{
    	- \LRa{ \jump{p_h^*}, 1 }_F = 0 ,
	}
	so \eqref{eq:ph-star-meanzero} follows.
	
	If $F \in \mathcal{F}_h^{\Gamma}$, then 
	\algns{
    	\LRp{\bs{u}_h,\bs{v}_F}_{\Omega} +\LRa{ \alpha \bs{u}_h \cdot \bs{n}, \bs{v}_F \cdot \bs{n} }_F - \LRp{p_h,\div \bs{v}_F}_{\Omega}  &= 0 .
	}
	The integration by parts gives 
	\algns{
    	\LRp{\bs{u}_h,\bs{v}_F}_{\Omega} +\LRa{ \alpha \bs{u}_h \cdot \bs{n} - \jump{p_h^*}, \bs{v}_F \cdot \bs{n} }_F - \LRp{\nabla p_h^*, \bs{v}_F}_{\Omega}  &= 0 ,
	}
	and we can conclude \eqref{eq:flux-meanzero} by a same argument for the proof of \eqref{eq:ph-star-meanzero}.
\end{proof}
%
%
\begin{thm}
For $\alpha$ satisfying \eqref{eq:alpha-large} with a uniform lower bound, there exist $C>0$ independent of mesh sizes and $\alpha$ such that the following local efficiency holds:
	\algn{
		\label{eq:eta-T-efficiency}
		\eta_T &\le C \| \bs{u} - \bs{u}_h \|_{0,T},
		\\
		\label{eq:eta-F-efficiency-1}
		\eta_F &\le C \sum_{F \subset \pd T} \| \bs{u} - \bs{u}_h \|_{0,T},\quad \text{ if } F \in \mathcal{F}_h \setminus \mathcal{F}_h^{\Gamma} ,
		\\
		\label{eq:eta-F-efficiency-2}
		\eta_F &\le C \LRp{ \sum_{F \subset \pd T} \| \bs{u} - \bs{u}_h \|_{0,T} + \alpha^{1/2} \| (\bs{u} - \bs{u}_h) \cdot \bs{n} \|_{0,F}} ,\quad \text{ if } F \in \mathcal{F}_h^{\Gamma} .
	}
\end{thm}
%
%
\begin{proof}
\tred{We first show \eqref{eq:eta-T-efficiency}.} 
\tred{By Lemma~3.7} in \cite{Cockburn-Zhang:2014}, there exists $C>0$ such that 
\algns{ 
	\| \bs{u}_h + \nabla p_h^* \|_{0,T} \le C \| \bs{u}_h - \bs{u} \|_{0,T},
}
so we have
\algn{ \label{eq:eta-T-estm}
	\eta_T = \|\bs{u}_h + \nabla p_h^* \|_{0,T} \le C \| \bs{u}_h - \bs{u} \|_{0,T}. 
}

We now prove \eqref{eq:eta-F-efficiency-1}. For $F \in \mathcal{F}_h \setminus \mathcal{F}_h^{\Gamma}$, using \eqref{eq:ph-star-meanzero} and $\jump{p}|_F =0$, we have (see \cite{Ainsworth:2007} or \cite[Lemma~3.5]{Cockburn-Zhang:2014})
\algn{
	\notag 
	\eta_F^2 &= \int_F h_F^{-1} ((I-P_F^0) \jump{p_h^* - p})^2 \, ds 
	\\
	\notag
	&\le 2 \sum_{T \in \mathcal{T}_h, F \subset \pd T} h_F^{-1} \int_{\pd T \cap F} ((I-P_F^0)(p - p_h^*))^2 \,ds  
	\\
	\label{eq:eta-F-intermediate-1}
	&\le C \sum_{T \in \mathcal{T}_h, F \subset \pd T} \| \nabla (p - p_h^*) \|_{0,T}^2 
	\\
	\notag
	&= C \sum_{T \in \mathcal{T}_h, F \subset \pd T} \| \bs{u} + \nabla p_h^* \|_{0,T}^2 .
}
By the triangle inequality $\| \bs{u} + \nabla p_h^* \|_{0,T} \le \| \bs{u} - \bs{u}_h \|_{0,T} + \| \bs{u}_h + \nabla p_h^* \|_{0,T}$ and \eqref{eq:eta-T-estm}, 
\algn{ \label{eq:eta-F-estm-1}
	\eta_F^2 &\le C\sum_{T \in \mathcal{T}_h, F \subset \pd T} \| \bs{u} - \bs{u}_h \|_{0,T}^2 .
}

For \eqref{eq:eta-F-efficiency-2}, using $\alpha \bs{u} \cdot \bs{n} = \jump{p}$ on $\Gamma$ and $(I - P_F^m) (\alpha \bs{u}_h \cdot \bs{n}) = 0$, 
%
%
\algns{
	\eta_F^2 &= \int_F \alpha^{-1} \LRp{(I - P_F^m) \jump{p_h^*}}^2 \, ds 
	\\
	&= \int_F \alpha^{-1} \LRp{(I - P_F^m) (\jump{p_h^*} - \jump{p} + \alpha \bs{u} \cdot \bs{n} - \alpha \bs{u}_h \cdot \bs{n} )}^2 \, ds .
}
Using \eqref{eq:alpha-large}, we can obtain
\algn{
	\eta_F^2 &\le 2 \sum_{T \in \mathcal{T}_h, F \subset \pd T} \int_{\pd T \cap F} \LRp{(I - P_F^m) (p_h^* - p)}^2 \,ds
	\\
	\notag
	&\quad + 2 \int_F \alpha (\bs{u} \cdot \bs{n} - \bs{u}_h \cdot \bs{n} )^2 \, ds .
}
Then, the first term on the right-hand side can be estimated as in \eqref{eq:eta-F-intermediate-1} with a natural assumption \tred{$h_F \le 1$}, so \eqref{eq:eta-F-efficiency-2} follows.
\end{proof}

\subsection{Reliability of post-processed pressure}

In this subsection we prove that $\eta$ is a realiable a posteriori error estimator for the error of post-processed pressure $p_h^*$.
We show a proof for $n=3$ because the same argument works for $n=2$. 

\begin{thm}
	For $p_h^*$ in \eqref{eq:post-processing-eqs} and $\eta$ in \eqref{eq:eta-def}, there exists $C>0$ independent of mesh sizes such that 
	\algn{ \label{eq:pressure-reliability}
		\| p - p_h^* \|_0 \le C \eta + \frac {C}{\pi} \LRp{ \sum_{T \in \mathcal{T}_h} {h_T^2} \| f - P_h f \|_{0,T}^2 }^{\frac 12}.
	}
\end{thm}
\begin{proof}
By the assumption \eqref{eq:assumption1}, there exists $\bs{w} \in H^1(\Omega; \mathbb{R}^3)$ such that 
\algn{ \label{eq:w-def}
	\div \bs{w} = p - p_h^*, \quad \| \bs{w} \|_1 \le C \| p - p_h^* \|_0, \quad \bs{w}\cdot \bs{n}|_{\Gamma} = 0 .
}
Using this $\bs{w}$, the equation $-\nabla p =\bs{u}$ in $\Omega$, continuity of $p$ on $\Omega \setminus \Gamma$, and $p = 0$ on $\pd \Omega$, 
\algns{
	\| p - p_h^* \|_0^2 &= \LRp{ p - p_h^*, \div \bs{w}}_{\Omega} 
	\\
	&= - \LRp{\nabla( p - p_h^*), \bs{w} }_{\Omega} \pm \sum_{F \in \mathcal{F}_h \setminus \mathcal{F}_h^{\Gamma} } \LRa{\jump{p - p_h^*}, \bs{w} \cdot \bs{n} }_F 
	\\
	&= \LRp{\bs{u} + \nabla p_h^*, \bs{w} }_{\Omega} \pm \sum_{F \in \mathcal{F}_h \setminus \mathcal{F}_h^{\Gamma} } \LRa{\jump{p_h^*}, \bs{w} \cdot \bs{n} }_F 
	\\
	&= \LRp{\bs{u} - \bs{u}_h, \bs{w} }_{\Omega} + \LRp{\bs{u}_h + \nabla p_h^*, \bs{w} }_{\Omega} 
	\pm \sum_{F \in \mathcal{F}_h \setminus \mathcal{F}_h^{\Gamma} } \LRa{\jump{p_h^*}, \bs{w} \cdot \bs{n} }_F .
}
Then, \eqref{eq:pressure-reliability} follows by 
\algns{
	|\LRp{\bs{u} - \bs{u}_h, \bs{w} }_{\Omega}| &\le \| \bs{u} - \bs{u}_h \|_0 \| \bs{w} \|_0, 
	\\
	| \LRp{ \bs{u}_h + \nabla p_h^*, \bs{w} }_{\Omega} | &\le \LRp{\sum_{T \in \mathcal{T}_h} \eta_T^2 }^{\frac 12} \| \bs{w} \|_0, 
	\\
	\left|  \sum_{F \in \mathcal{F}_h \setminus \mathcal{F}_h^{\Gamma} } \LRa{\jump{p_h^*}, \bs{w} \cdot \bs{n} }_F \right| &\le \LRp{ \sum_{F \in \mathcal{F}_h \setminus \mathcal{F}_h^{\Gamma}} \eta_F^2}^{\frac 12} \| \bs{w} \|_1 ,	
}
and \eqref{eq:velocity-reliability}.
\end{proof}
With an assumption of (partial) elliptic regularity, we show that an improved estimate of $\| p - p_h^* \|_0$ is obtained.
Consider the dual problem to find $(\bar{\bs{u}}, \bar{p}) \in \bs{V} \times Q$ satisfying
\begin{subequations}
	\label{eq:dual-problem-eqs}
	\algn{ 
		\label{eq:dual-problem-eq1}
    	\LRp{\bar{\bs{u}},\bs{v}}_{\Omega}- \LRp{\bar{p}, \div \bs{v}}_{\Omega} + \LRa{\alpha 	\bar{\bs{u}}\cdot\bs{n},\bs{v}\cdot\bs{n}}_{\Gamma} &= 0 & &\forall\bs{v} \in \bs{V},
    	\\
		\label{eq:dual-problem-eq2}
	    \LRp{\div \bar{\bs{u}}, q}_{\Omega} &= \LRp{p - p_h^* , q}_{\Omega} & &\forall q \in Q, 
	}
\end{subequations}
and assume that 
\algn{ \label{eq:partial-regularity}
	\bar{\bs{u}} \in H^{\beta}(\Omega; \mathbb{R}^n) , \quad \frac 12 < \beta \le 1 , \qquad \| \bs{u} \|_{\beta} \le C \| p - p_h^* \|_0 .
}
The boundary condition of this problem is $\bar{p} = 0$ on $\pd \Omega$. By the inf-sup condition \eqref{eq:continuous-inf-sup}, there exists $C>0$ such that $ \| \bar{p} \|_0 \le C \tnorm{\bar{\bs{u}}}$. 
Furthermore, by taking $\bs{v} = \bar{\bs{u}}$, $q = \bar{p}$, we can obtain
\algns{
	\tnorm{\bar{\bs{u}}}^2 \le \LRp{ p - p_h^*, \bar{p}}_{\Omega} \le C \| p - p_h^* \|_0 \| \bar{p} \|_0, 
}
so, 
\algn{ \label{eq:dual-stability}
	\tnorm{\bar{\bs{u}}} \le C \| p - p_h^* \|_0 .
}
\begin{cor} \label{cor:duality}
Suppose that \eqref{eq:partial-regularity} holds and $\bs{V}_h = \bs{V}_h^{RTN}$ with $k=1$. Then, there exists $C>0$ independent of mesh sizes such that 
\algns{
	\| p - p_h^* \|_0 &\le \LRp{ \sum_{T \in \mathcal{T}_h} h_T^{2 \beta} \eta_T^2 }^{\frac 12}  + \LRp{ \sum_{T \in \mathcal{T}_h} \osc(f,T)^2 }^{\frac 12} 
	\\
	&\quad +\LRp{ \sum_{F \in \mathcal{F}_h^{\Gamma}} \min \{ \eta_F, \alpha^{1/2} h_F^{\beta} \eta_F \}^2 + \sum_{F \in \mathcal{F}_h \setminus \mathcal{F}_h^{\Gamma}} h_F^{2\beta-1} \eta_F^2}^{\frac 12} .
}
\end{cor}
\begin{proof}
For $\bar{\bs{u}}$ in \eqref{eq:dual-problem-eqs} and $\tilde{\bs{u}}$ in \eqref{eq:variational-project-eqs}
\algn{
	\notag
	\| p - p_h^* \|_0^2 &= \LRp{ p - p_h^*, \div \bar{\bs{u}} }_{\Omega} 
	\\
	\notag
	&= - \LRp{\nabla( p - p_h^*), \bar{\bs{u}}}_{\Omega} + \sum_{T \in \mathcal{T}_h } \LRa{ {p - p_h^*}, \bar{\bs{u}}\cdot \bs{n} }_{\pd T} 
	\\
	\notag
	&= \LRp{\bs{u} + \nabla p_h^*, \bar{\bs{u}} }_{\Omega} + \sum_{T \in \mathcal{T}_h } \LRa{{p-p_h^*}, \bar{\bs{u}} \cdot \bs{n} }_{\pd T} .
	\\
	\label{eq:duality-estm1}
	&= \LRp{\bs{u} - \bs{u}_h, \bar{\bs{u}} }_{\Omega} + \LRp{ \bs{u}_h + \nabla p_h^*, \bar{\bs{u}} }_{\Omega} + \sum_{T \in \mathcal{T}_h } \LRa{{p-p_h^*}, \bar{\bs{u}} \cdot \bs{n} }_{\pd T} .
}
%
%
Taking $\bs{v} = \bs{u} - \bs{u}_h$ in \eqref{eq:dual-problem-eq1}, 
\algn{
	\notag
	\LRp{\bs{u} - \bs{u}_h, \bar{\bs{u}} }_{\Omega} &= - \LRa{ \alpha \bar{\bs{u}} \cdot \bs{n}, (\bs{u} - \bs{u}_h)\cdot \bs{n} }_{\Gamma} + \LRp{\bar{p}, \div (\bs{u} - \bs{u}_h)}_{\Omega}
	\\
	\label{eq:duality-estm2}
	&= - \LRa{ \alpha \bar{\bs{u}} \cdot \bs{n}, (\bs{u} - \bs{u}_h)\cdot \bs{n} }_{\Gamma} + \LRp{\bar{p}, f - P_h f }_{\Omega} .
}
By \eqref{eq:duality-estm1}, \eqref{eq:duality-estm2}, \eqref{eq:interface-condition}, and $\jump{p}|_F = 0$ for $F \in \mathcal{F}_h \setminus \mathcal{F}_h^{\Gamma}$, 
\algns{
	\| p - p_h^* \|_0^2 &= \LRp{ \bs{u}_h + \nabla p_h^*, \bar{\bs{u}} }_{\Omega} + \LRa{ \alpha \bs{u}_h \cdot \bs{n} - \jump{p_h^*}, \bar{\bs{u}}\cdot \bs{n} }_{\Gamma}
	\\
	&\quad \pm \sum_{F \in \mathcal{F}_h \setminus \mathcal{F}_h^{\Gamma} } \LRa{\jump{p_h^*}, \bar{\bs{u}} \cdot \bs{n} }_F + \LRp{\bar{p}, f - P_h f }_{\Omega} 
	\\
	&=: J_1 + J_2 + J_3 + J_4 .
}

To estimate $J_1$, note that $\nabla : \mathcal{P}_1(T) \to \mathcal{P}_0(T; \mathbb{R}^n)$ is surjective, so we get
\algns{
	|\LRp{ \bs{u}_h + \nabla p_h^*, \bar{u}}_T| &= \LRp{ \bs{u}_h + \nabla p_h^*, \bar{u} - \nabla \phi}_T \quad \forall \phi \in \mathcal{P}_1(T) 
	\\
	&\le C \| \bs{u}_h + \nabla p_h^* \|_{0,T} h_T^{\beta} \| \bar{\bs{u}} \|_{H^{\beta}(T)}
	\\
	&\le C h_T^{\beta} \eta_T \| \bar{\bs{u}} \|_{H^{\beta}(T)}
}
where we used \eqref{eq:post-processing-eq1}, a Poincare inequality for fractional Sobolev spaces (\cite[Lemma~3.1]{Bellido-Mora-Corral:2014}) with a standard scaling argument, and \eqref{eq:eta-T-def}. Then, by the Cauchy--Schwarz inequality and the integral form of fractional Sobolev norm \cite[Chapter~14]{Brenner-Scott-book}),
\algns{
	|J_1| &= \sum_{T \in \mathcal{T}_h} |\LRp{ \bs{u}_h + \nabla p_h^*, \bar{u}}_T| 
	\\
	&\le C \LRp{ \sum_{T \in \mathcal{T}_h} h_T^{2\beta} \eta_T^2 }^{\frac 12} \| \bar{\bs{u}} \|_{H^{\beta}(\Omega; \mathbb{R}^n)} 
	\\
	&\le C \LRp{ \sum_{T \in \mathcal{T}_h} h_T^{2\beta} \eta_T^2 }^{\frac 12} \| p - p_h^* \|_0 .
}

To estimate $J_2$ and $J_3$, we first note that one can obtain 
\algn{ \label{eq:beta-approx}
	\| \bar{\bs{u}} \cdot \bs{n} - P_F^0 \bar{\bs{u}} \cdot \bs{n} \|_{0,F} \le C h^{\beta} \| \bar{\bs{u}} \|_{H^{\beta}(T; \mathbb{R}^n)}, \quad F \subset \pd T
}
by a trace theorem of fractional Sobolev spaces (see, e.g., \cite{Renardy-Rogers-book}), the Poincare inequality for fractional Sobolev spaces, and a standard scaling argument. 
Since $\bs{u}_h \cdot \bs{n}|_F \in \mathcal{P}_0(F)$ and $P_F^0 (\alpha \bs{u}_h \cdot \bs{n} - \jump{p_h^*}) = 0$ for $F \in \mathcal{F}_h^{\Gamma}$ by \eqref{eq:flux-meanzero}, 
\algn{
	\label{eq:dual-flux-estm1}
	\LRa{ \alpha \bs{u}_h \cdot \bs{n} - \jump{p_h^*}, \bar{\bs{u}} \cdot \bs{n}}_F &= \LRa{ \jump{p_h^*} - P_F^0 \jump{p_h^*}, \bar{\bs{u}} \cdot \bs{n} }_F  
	\\
	\label{eq:dual-flux-estm2}
	&= \LRa{ \jump{p_h^*} - P_F^0 \jump{p_h^*}, \bar{\bs{u}} \cdot \bs{n} - P_F^0 \bar{\bs{u}} \cdot \bs{n}}_F 	
}
for $F \in \mathcal{F}_h^{\Gamma}$. From \eqref{eq:dual-flux-estm1} and \eqref{eq:dual-flux-estm2}, we can obtain either
\algns{
	| \LRa{\alpha \bs{u}_h \cdot \bs{n} - \jump{p_h^*}, \bar{\bs{u}} \cdot \bs{n}}_F| &\le \eta_F \| \alpha^{1/2} \bar{\bs{u}} \cdot \bs{n} \|_{0,F}
}
or 
\algns{
	| \LRa{\alpha \bs{u}_h \cdot \bs{n} - \jump{p_h^*}, \bar{\bs{u}} \cdot \bs{n}}_F| &\le C \alpha^{1/2} h_F^{\beta} \eta_F \| \bs{u} \|_{H^{\beta}(T; \mathbb{R}^n)}, \quad F \subset \pd T. 
}
The mean-value zero property \eqref{eq:ph-star-meanzero} gives 
\algns{
	\LRa{ \jump{p_h^*}, \bar{\bs{u}} \cdot \bs{n}}_F = \LRa{ \jump{p_h^*}, \bar{\bs{u}} \cdot \bs{n} - P_F^0 \bar{\bs{u}} \cdot \bs{n}}_F \quad \forall F \in \mathcal{F}_h \setminus \mathcal{F}_h^{\Gamma} ,
}
so we can obtain
\algns{
	| \LRa{ \jump{p_h^*}, \bar{\bs{u}} \cdot \bs{n}}_F | &\le C \eta_F h_F^{\beta-1/2} \| \bs{u} \|_{H^{\beta}(T; \mathbb{R}^n)}, \quad F \subset \pd T
}
for $F \in \mathcal{F}_h \setminus \mathcal{F}_h^{\Gamma}$. Combining all of these estimates, 
\algns{
	|J_2 + J_3| &\le C \LRp{ \sum_{F \in \mathcal{F}_h^{\Gamma}} \min \{ \eta_F, \alpha^{1/2} h_F^{\beta} \eta_F \}^2 + \sum_{F \in \mathcal{F}_h \setminus \mathcal{F}_h^{\Gamma}} h_F^{2\beta-1} \eta_F^2}^{\frac 12}
	\\
	&\quad \times \| p - p_h^* \|_0 .
}
Finally, $J_4$ can be estimated as
\algns{
	|J_4| &\le \sum_{T \in \mathcal{T}_h} \frac{h_T}{\pi} \| \nabla \bar{p} \|_{0,T} \| f - P_h f \|_{0,T}
	\\
	&\le \| \bar{\bs{u}} \|_0 \frac 1{\pi} \LRp{ \sum_{T \in \mathcal{T}_h} \osc(f,T)^2 }^{\frac 12} 
	\\
	&\le C \| p - p_h^* \|_0 \frac 1{\pi} \LRp{ \sum_{T \in \mathcal{T}_h} \osc(f,T)^2 }^{\frac 12} 
}
by \eqref{eq:dual-stability}, so the conclusion follows.
\end{proof}
\begin{rmk}
In Corollary~\ref{cor:duality}, if $\alpha \ll h_F^{2\beta}$, then an improved error bound with the a posteriori error estimator terms are obtained. However, the data oscillation error term with $\osc(f,T)$ is \tred{the} same, so data oscillation can be a dominant factor in error bounds. If $f \in H^1(\Omega)$, then an improved bound $\osc(f,T) \le \frac{\tred{h_T^2}}{\pi} \| f \|_{1,T}$ can be obtained for each $\osc(f,T)$. This observation can explain convergence of $\| p - p_h^* \|_0$ faster than the one of $\| \bs{u} - \bs{u}_h \|_0$. 
\end{rmk}

\section{Numerical results}
In this section we present results of numerical experiments. All experiments are done with the finite element package FEniCS (version 2019.1.0 \cite{fenicsbook}). In particular, the marked elements for refinement are refined using the built-in adaptive mesh refinement algorithm in FEniCS. 

In the first numerical experiment let $\Omega = [0,1]\times [0,1]$ with fault $\Gamma = \{1/2\} \times [1/4, 3/4] \subset \Omega$ (see the first figure in Figure~\ref{fig:exact-pressure}). The manufactured solution (see the middle and right figures in Figure~\ref{fig:exact-pressure}) for this test case is given by
\algn{ \label{eq:exact-pressure}
p(x,y) = 
\begin{cases}
  0 & \text{ if } y < \frac 14 \text{ or } y > \frac 34 \\
  \sin \frac{3 \pi x}{2} \cos^2 \LRp{2\pi \LRp{y-\frac 12}} & \text{ if } x < \frac 12 \text{ and } \frac 14 \le y \le \frac 34 \\
  -\sin \frac{3 \pi (1-x)}{2} \cos^2 \LRp{2 \pi \LRp{y-\frac 12}} & \text{ if } x > \frac 12 \text{ and } \frac 14 \le y \le \frac 34.
\end{cases}
}
\begin{figure}[h]  
	\centering
	\begin{subfigure}[b]{0.3\textwidth}        
	    \includegraphics[trim=0 0 0 0,clip,width=\textwidth]{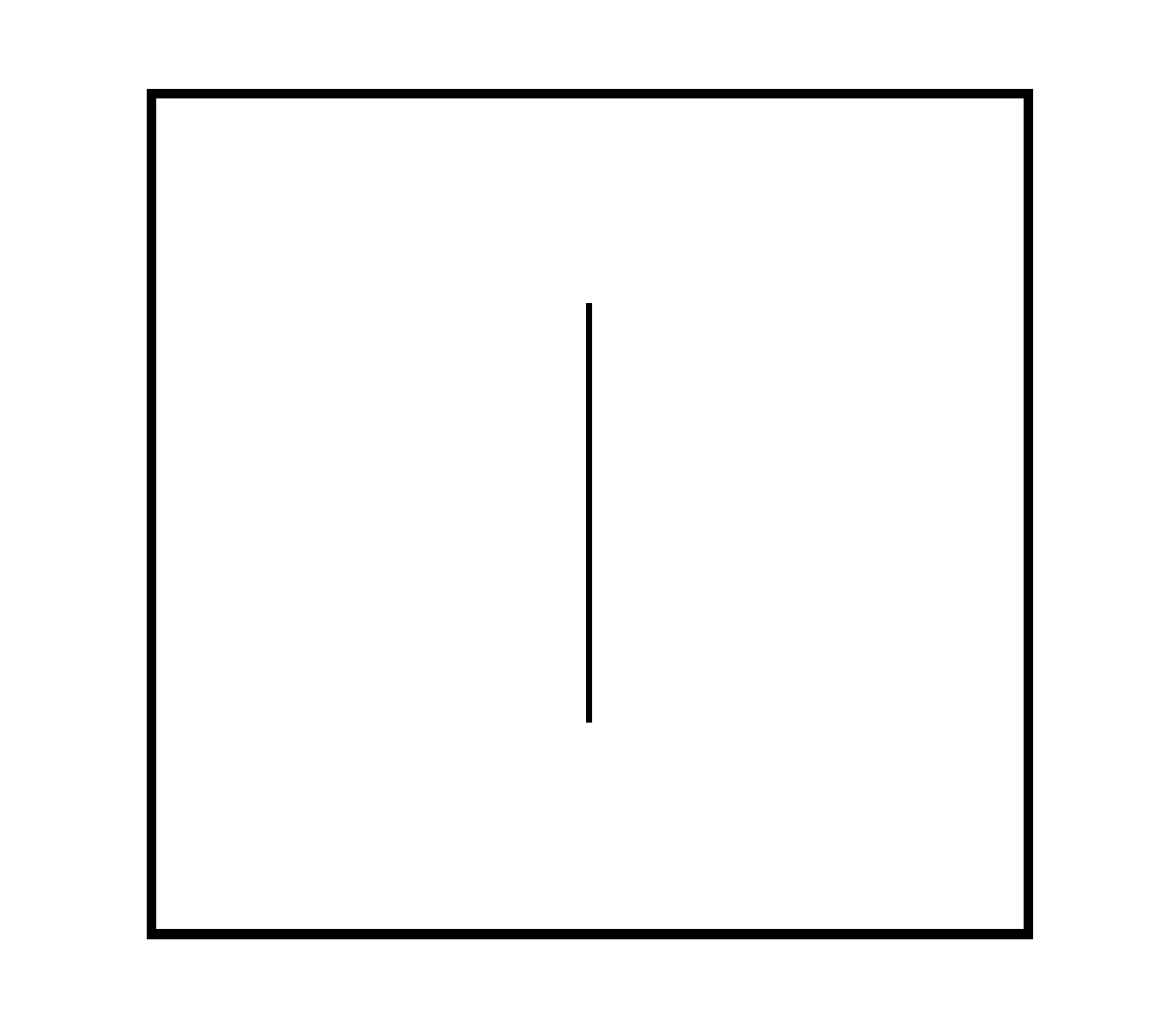} 
    \end{subfigure}
        ~
	\begin{subfigure}[b]{0.225\textwidth}
  			\includegraphics[trim=0 -80 0 0,width=\textwidth]{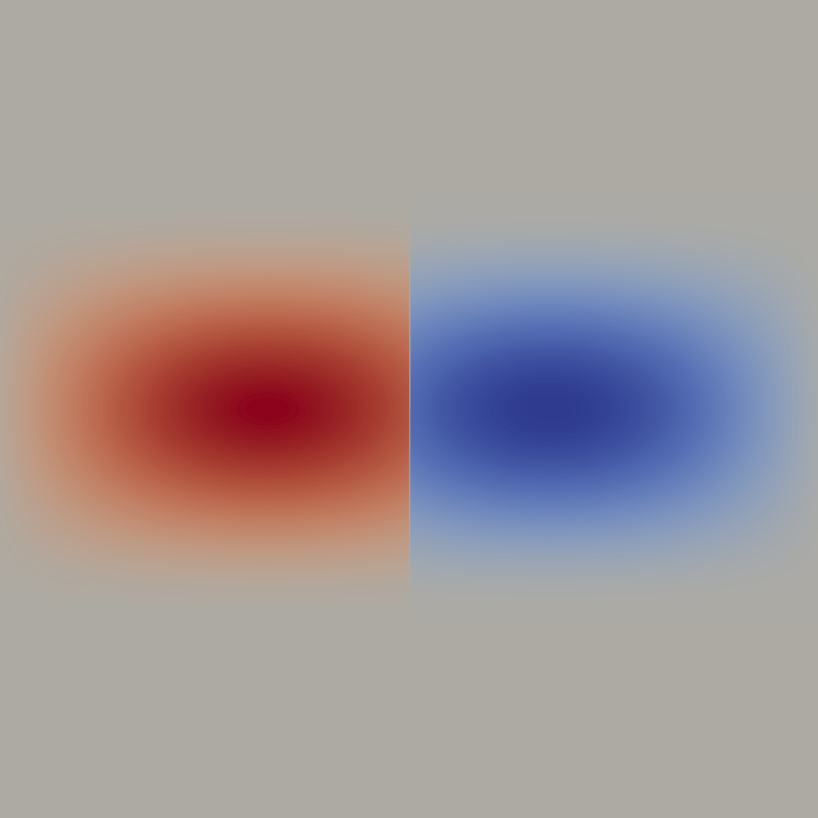} 
    \end{subfigure} 
  		~
  	\begin{subfigure}[b]{0.3\textwidth}
  			\includegraphics[trim=-50 0 0 0,width=\textwidth, height=\textwidth]{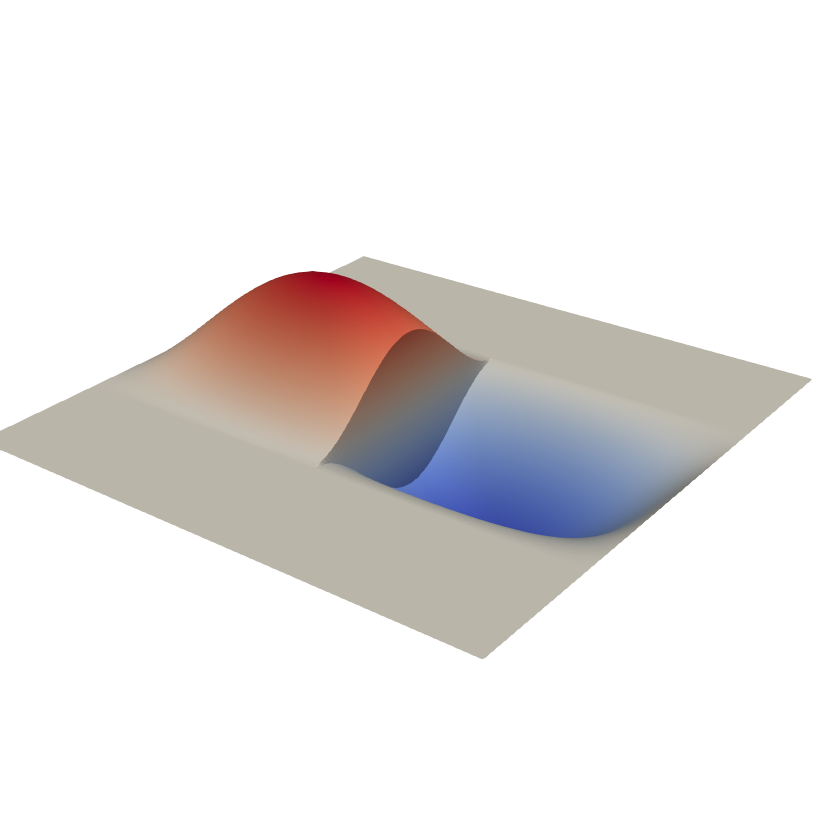}
  	\end{subfigure}
  	\caption{The domain with a vertical fault in numerical experiments (left figure) and the graphs of the pressure field in \eqref{eq:exact-pressure} (middle and right figures). }
  	\label{fig:exact-pressure}
\end{figure}

\begin{figure}[h]
	\begin{center}
 		\includegraphics[width=1.\linewidth]{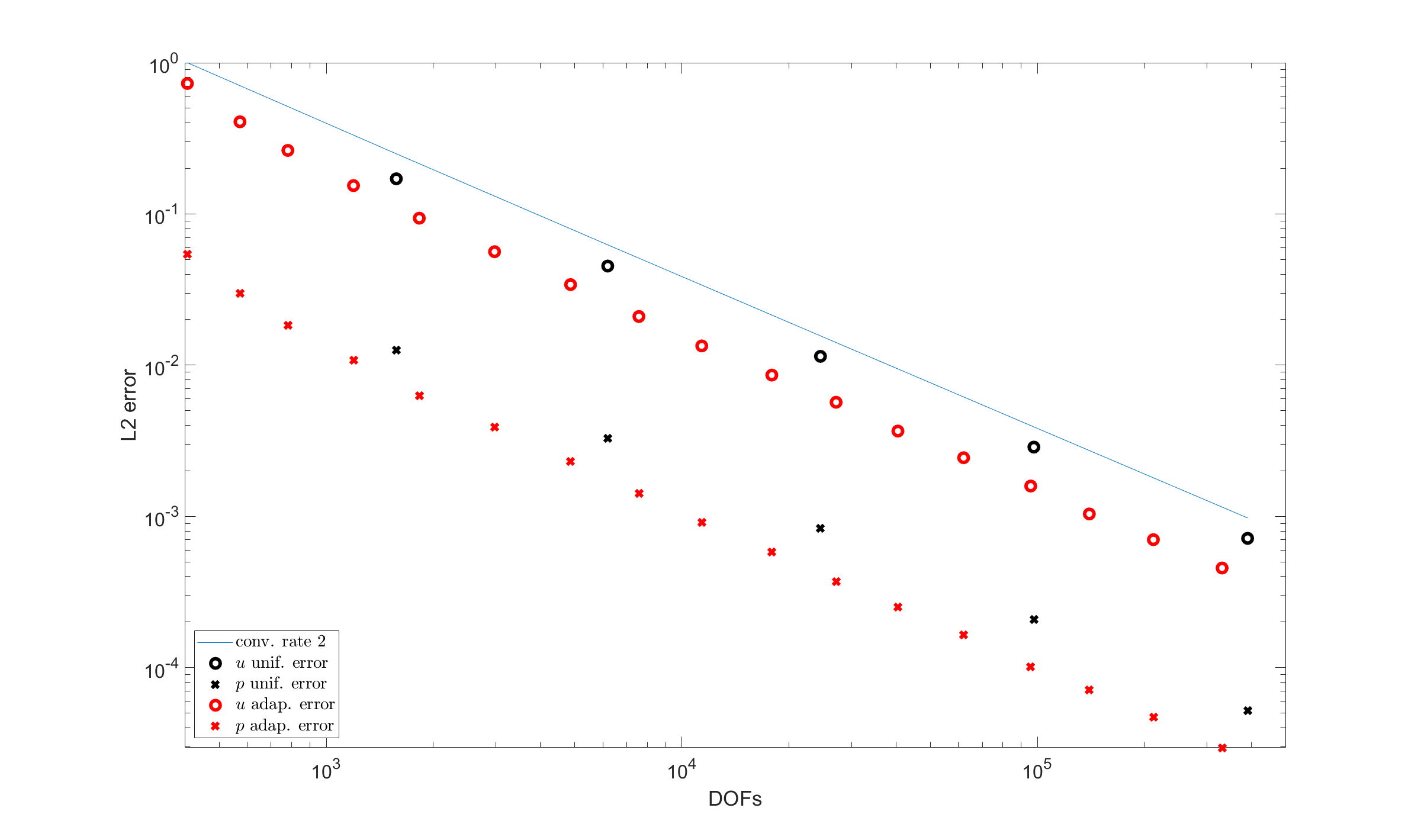}
 	\end{center}
	\caption{Comparison of convergence of errors for uniform and adaptive refinements. The errors are computed with the manufactured solution \eqref{eq:exact-pressure}. Pressure errors are computed with the post-processed pressure $p_h^*$.}
	\label{fig:eg1-convergence}
\end{figure}

\begin{figure}[h]  
	\hspace{-0.cm}
	\begin{subfigure}[b]{0.7\textwidth}        
		\includegraphics[width=.45\linewidth]{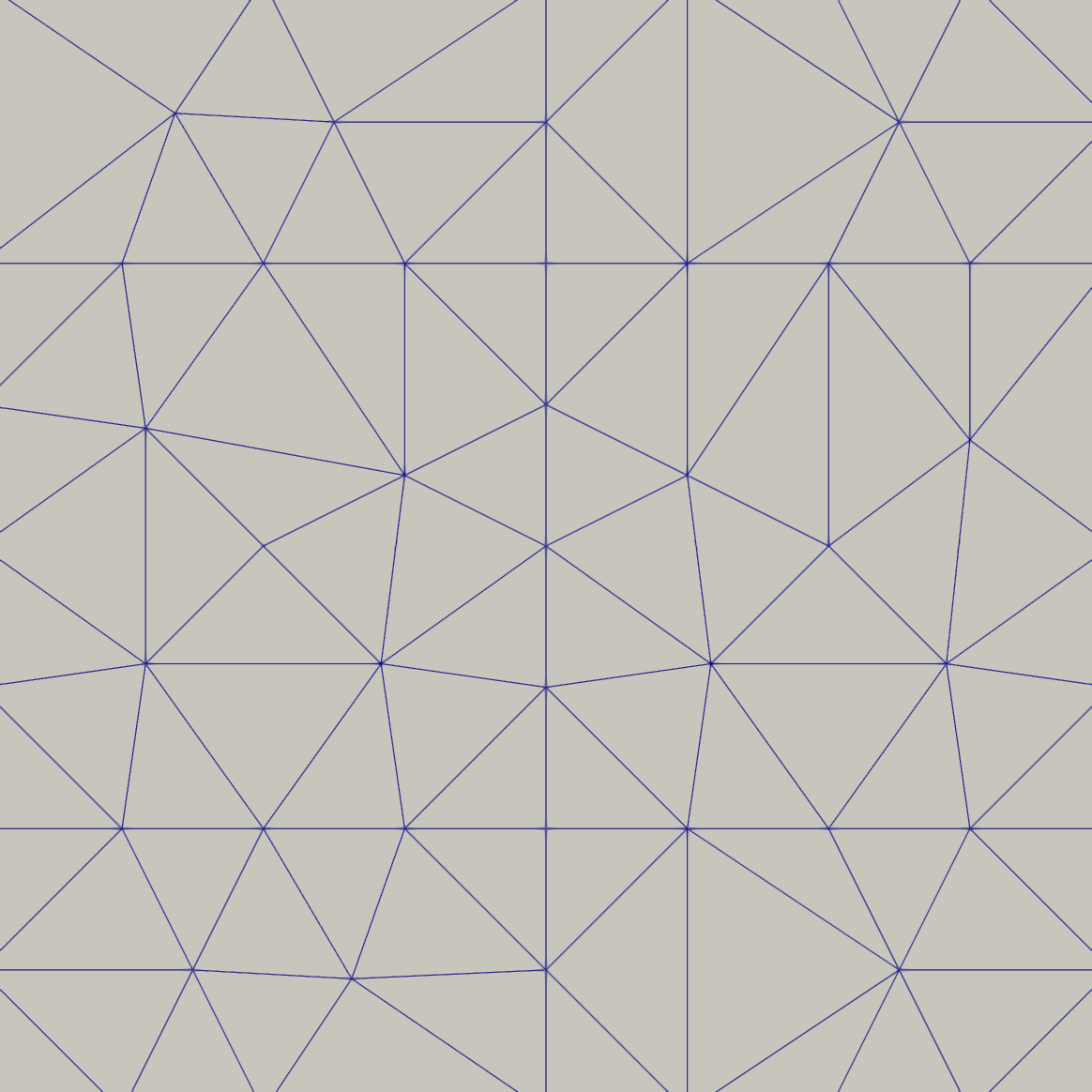} ~\hspace{3mm}
 		\includegraphics[height=.45\linewidth]{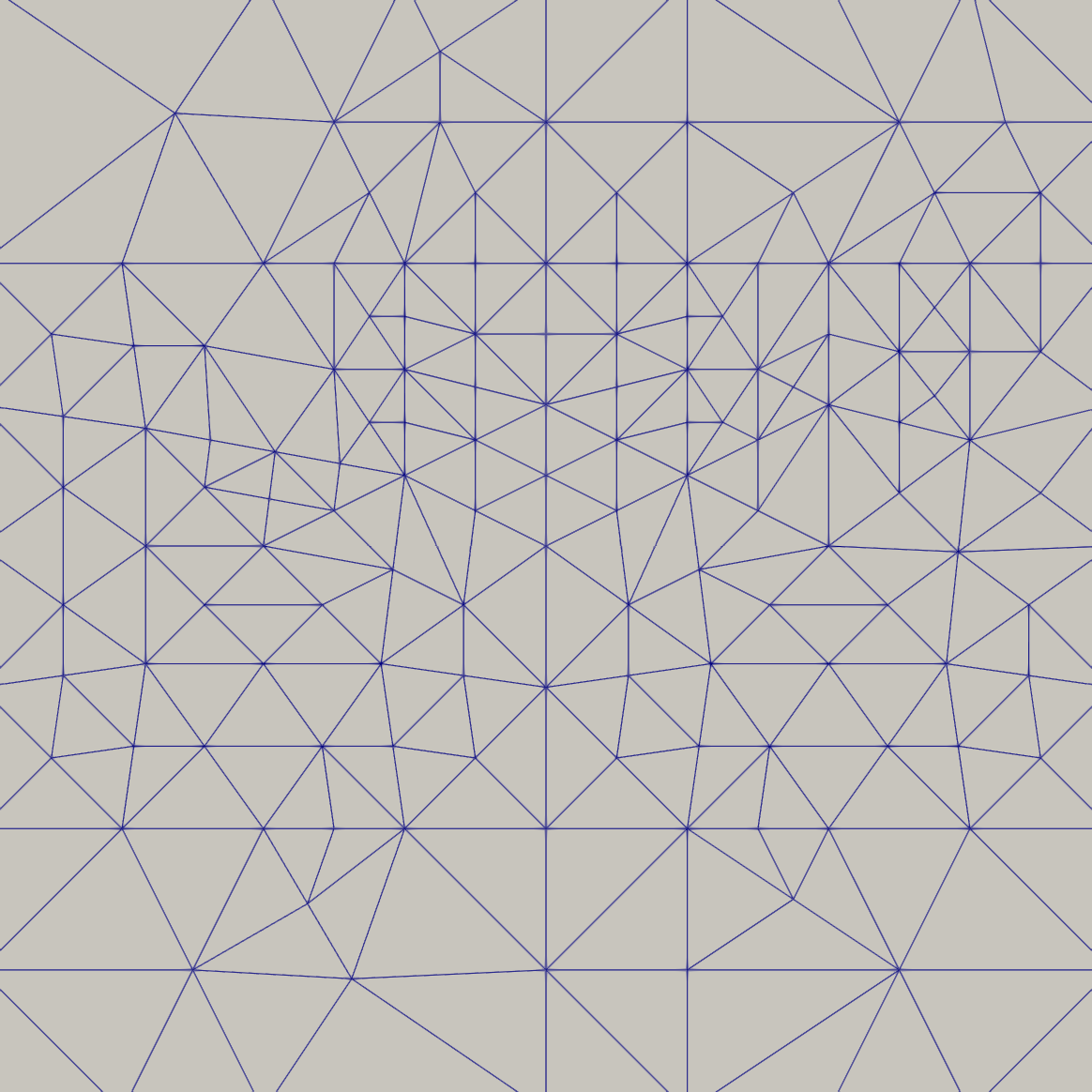}
    \end{subfigure} 
    \\
    \vspace{5mm}
	\hspace{-0.cm}
	\begin{subfigure}[b]{0.7\textwidth}
 		\includegraphics[height=.45\linewidth]{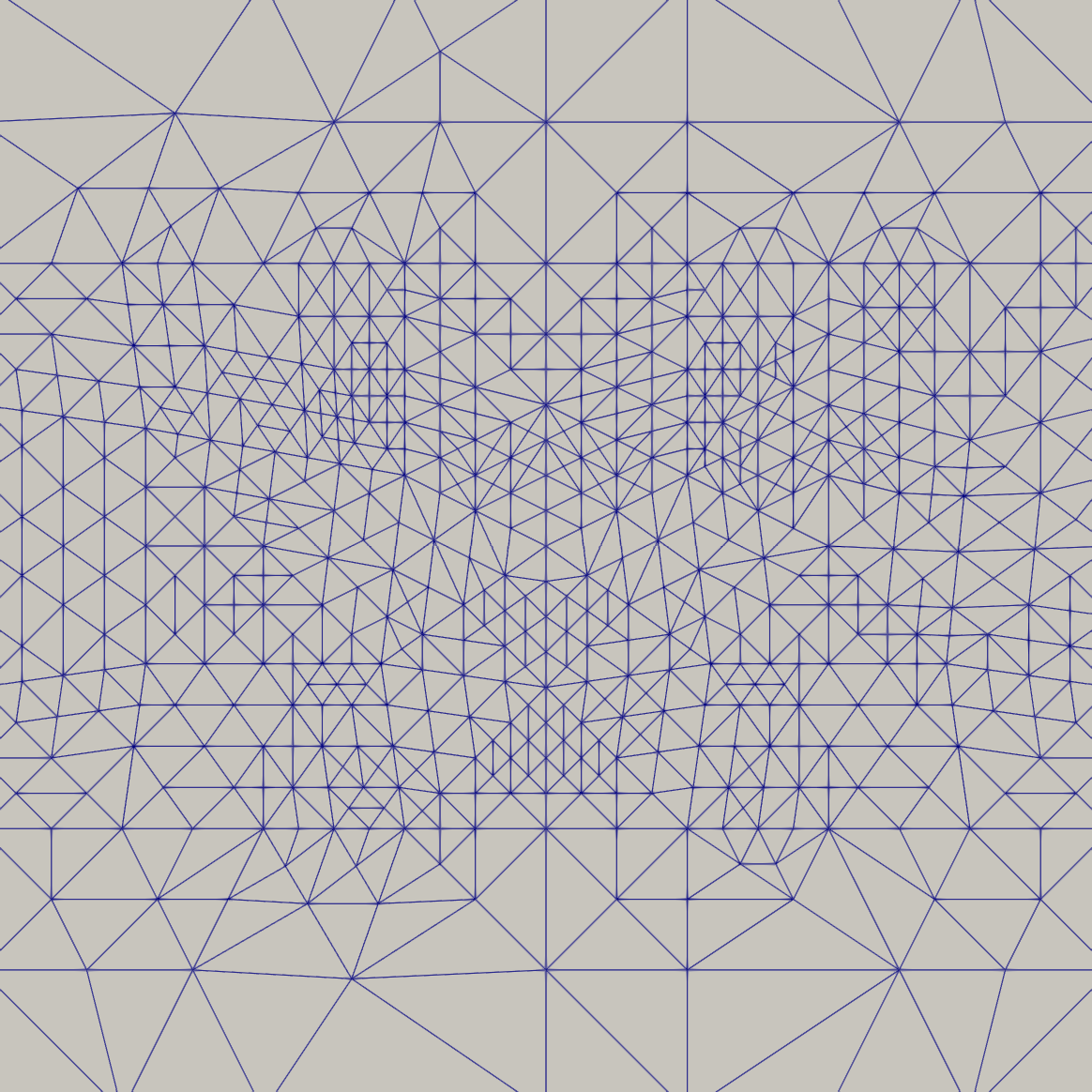} ~\hspace{3mm}
 		\includegraphics[height=.45\linewidth]{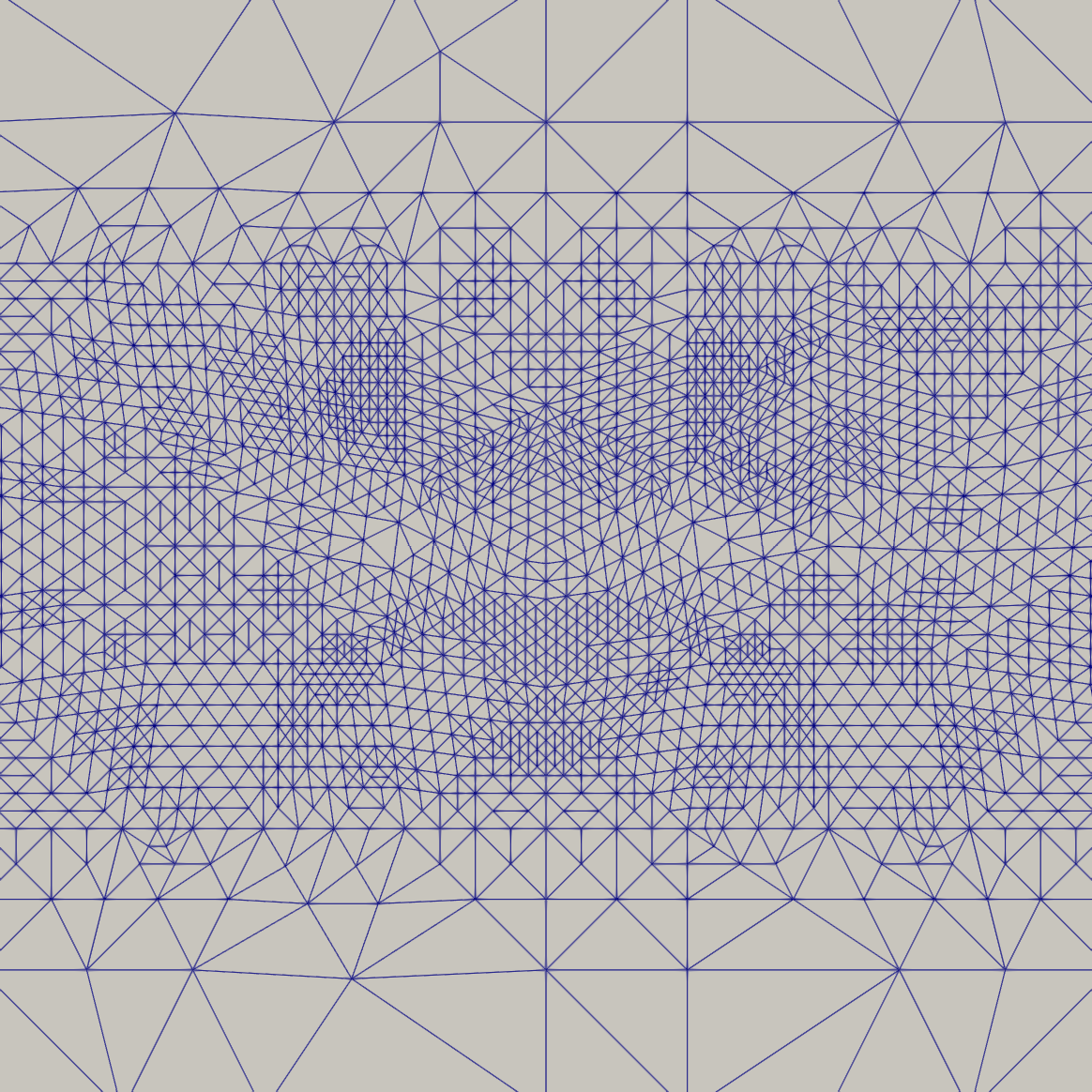}
  	\end{subfigure}
  	\caption{\tred{The initial, and the 3rd, 6th, 9th refined meshes in adaptive solves with the manufactured solution \eqref{eq:exact-pressure}.} }
  	\label{fig:ex1-stages}
\end{figure}

\begin{table}[!h]  
\begin{tabular}{>{\small}c >{\small}c >{\small}c >{\small}c >{\small}c >{\small}c >{\small}c >{\small}c >{\small}c >{\small}c >{\small}c >{\small}c >{\small}c >{\small}c >{\small}c >{\small}c >{\small}c >{\small}c >{\small}c >{\small}c}
\hline
{Dofs} & 407 & 572 & 780 & 1193 & 1827 & 2974 & 4864 & 7573 & 11375 & 17908 \\ 
\hline \hline
{Eff.} & 1.43 & 1.63 & 1.62 & 1.57& 1.63 & 1.57&  1.59& 1.59 & 1.55& 1.57 \\
\hline  
\end{tabular}        
\caption{ \tred{The numbers of degrees of freedom (Dofs) and effectivity indices (Eff.) in adaptive solves with the manufactured solution \eqref{eq:exact-pressure}.} }  
\label{effectivity} 
\end{table}


%
\begin{figure}[h]  
	\hspace{-0.cm}
	\begin{subfigure}[b]{0.7\textwidth}        
		\includegraphics[width=.45\linewidth]{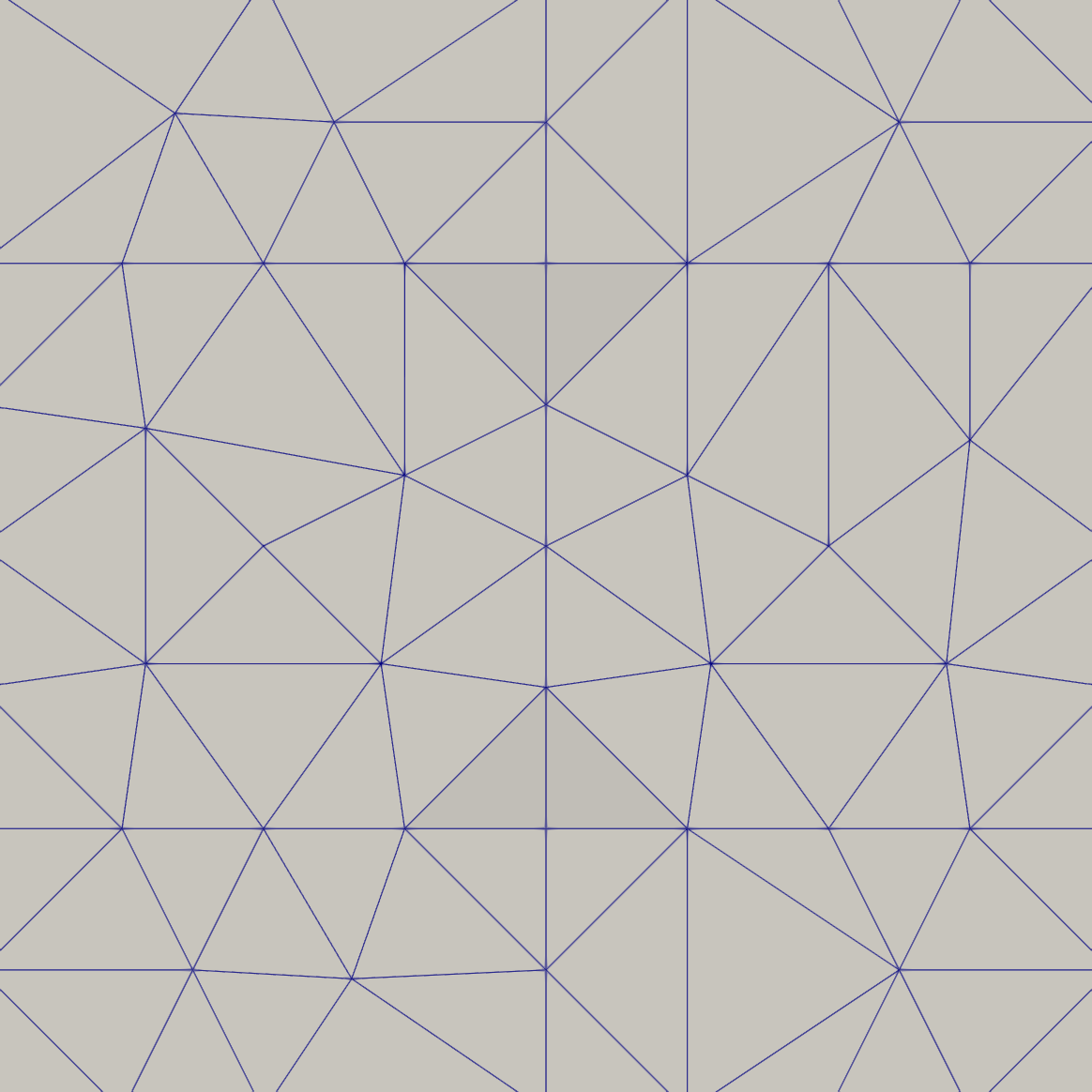} ~\hspace{3mm}
 		\includegraphics[height=.45\linewidth]{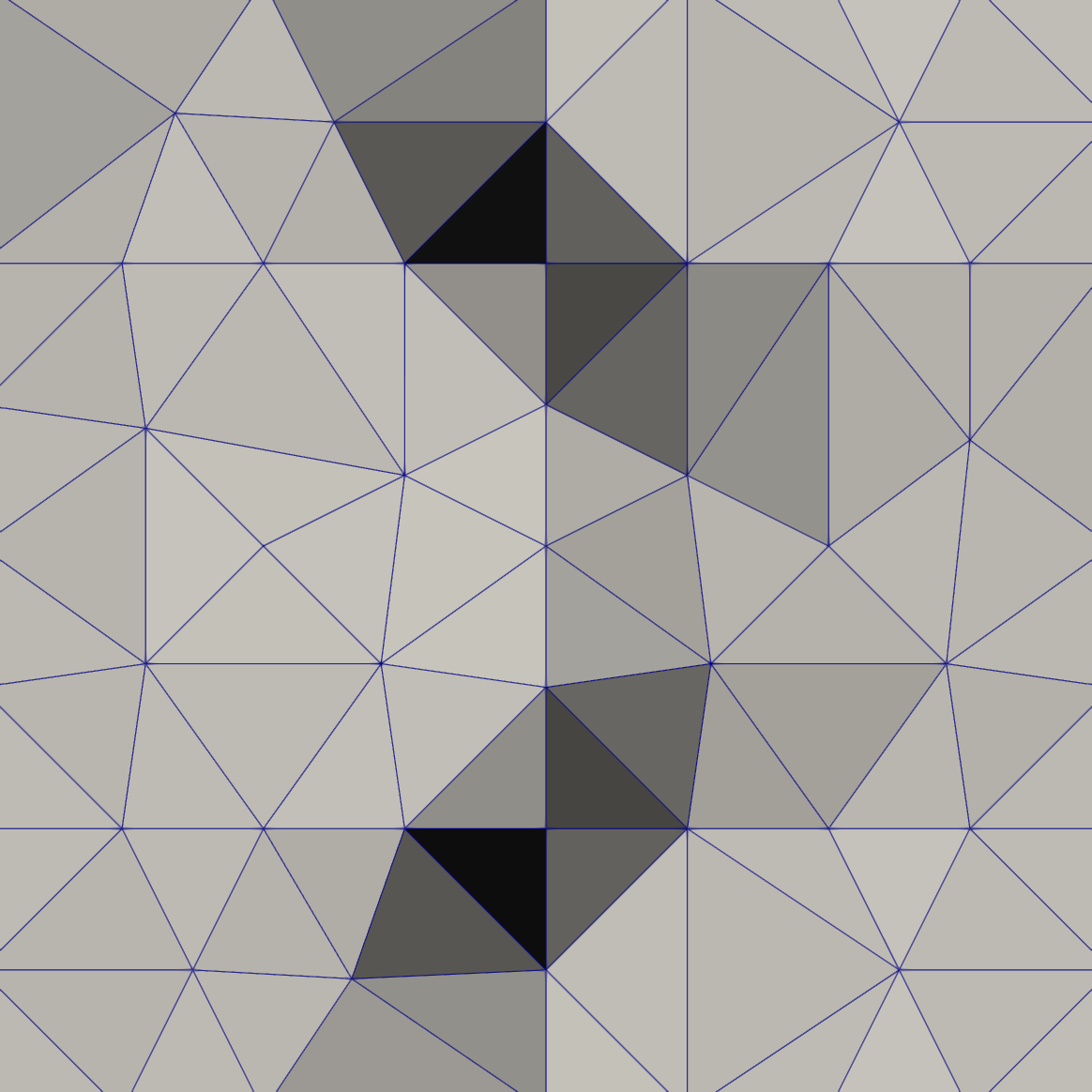}
    \end{subfigure} 
    \\
    \vspace{5mm}
	\hspace{-0.cm}
	\begin{subfigure}[b]{0.7\textwidth}
 		\includegraphics[height=.45\linewidth]{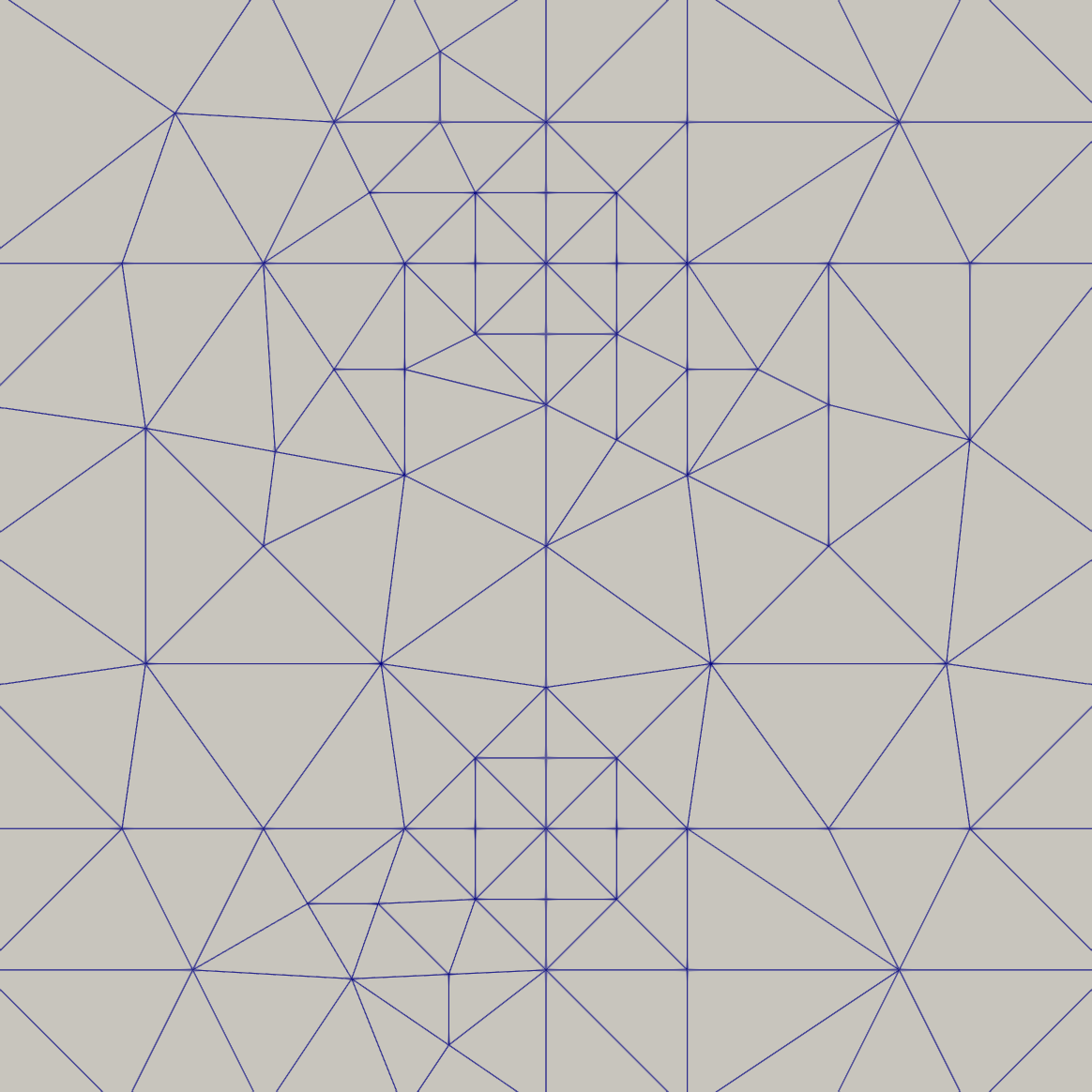} ~\hspace{3mm}
 		\includegraphics[height=.45\linewidth]{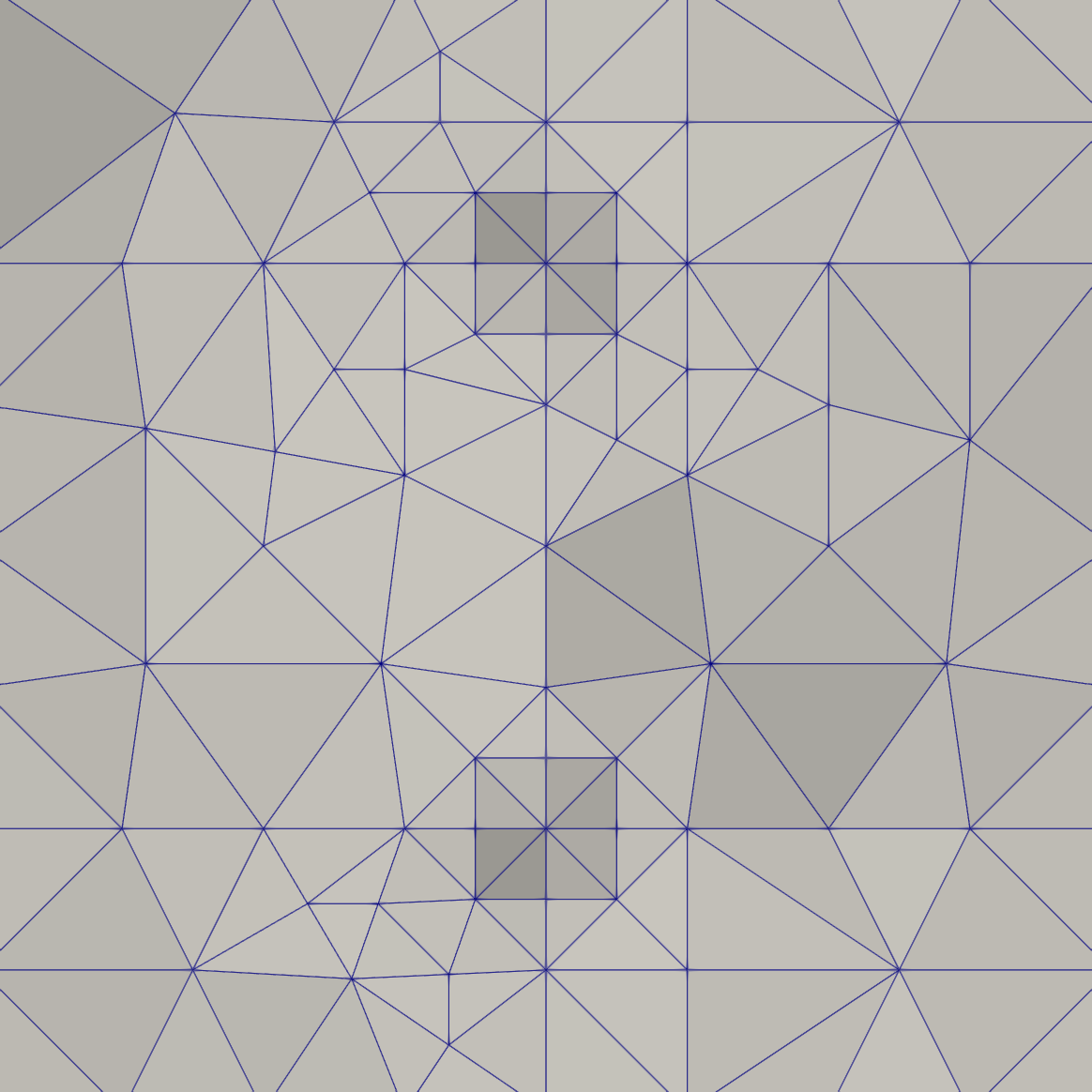}
  	\end{subfigure}
  	\caption{\tred{Distribution of $\{\tilde{\eta}_{\Gamma,T}\}$ (left), $\{\tilde{\eta}_{0,T}\}$ (right) in the initial and the first mesh refinement for $\alpha = 0.1$ (color scale: white = 0, black = 2.0e-4)} }
  	\label{fig:eta-alpha-tenth}
\end{figure}
\begin{figure}[h]  
	\hspace{-0.cm}
	\begin{subfigure}[b]{0.7\textwidth}        
		\includegraphics[width=.45\linewidth]{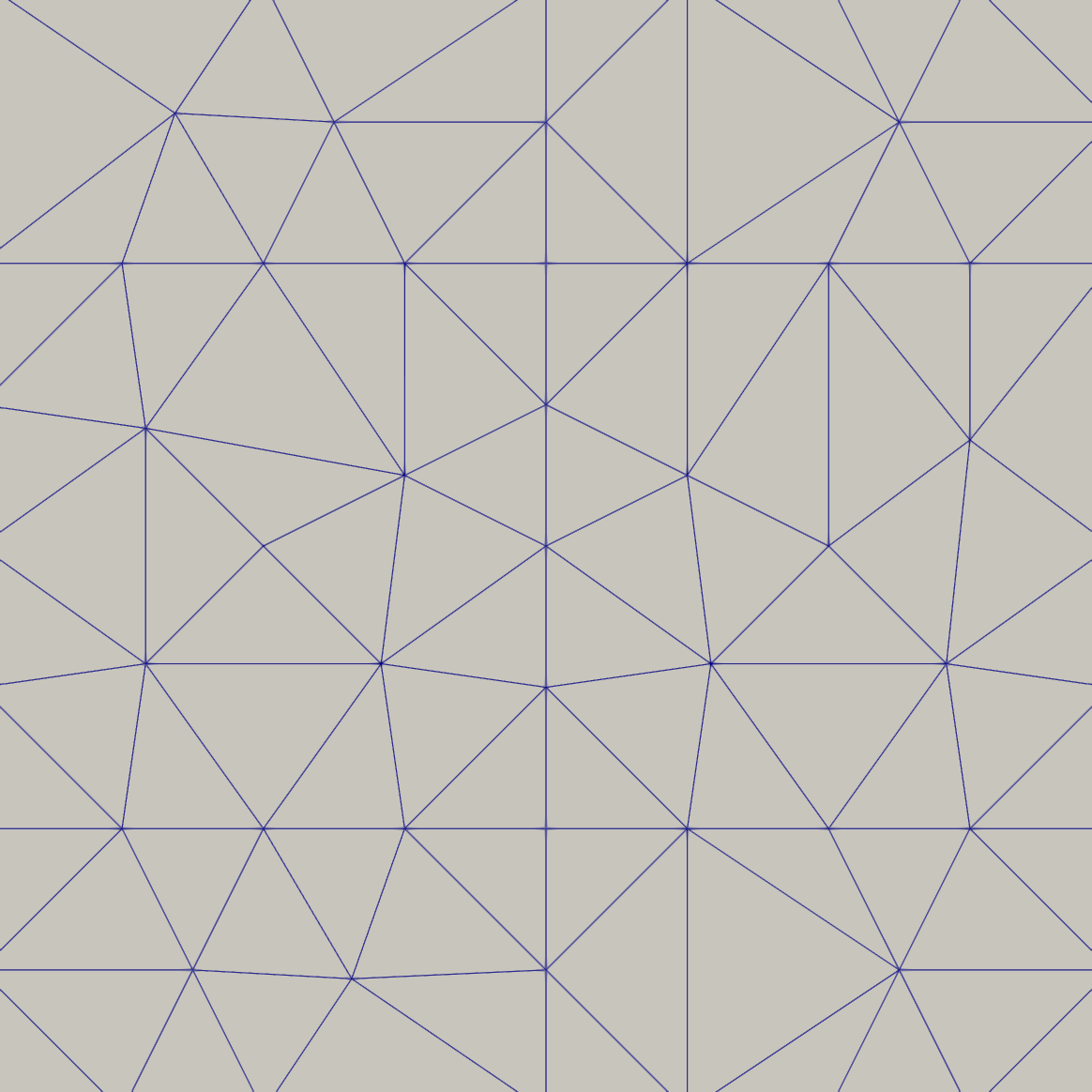} ~\hspace{3mm}
 		\includegraphics[height=.45\linewidth]{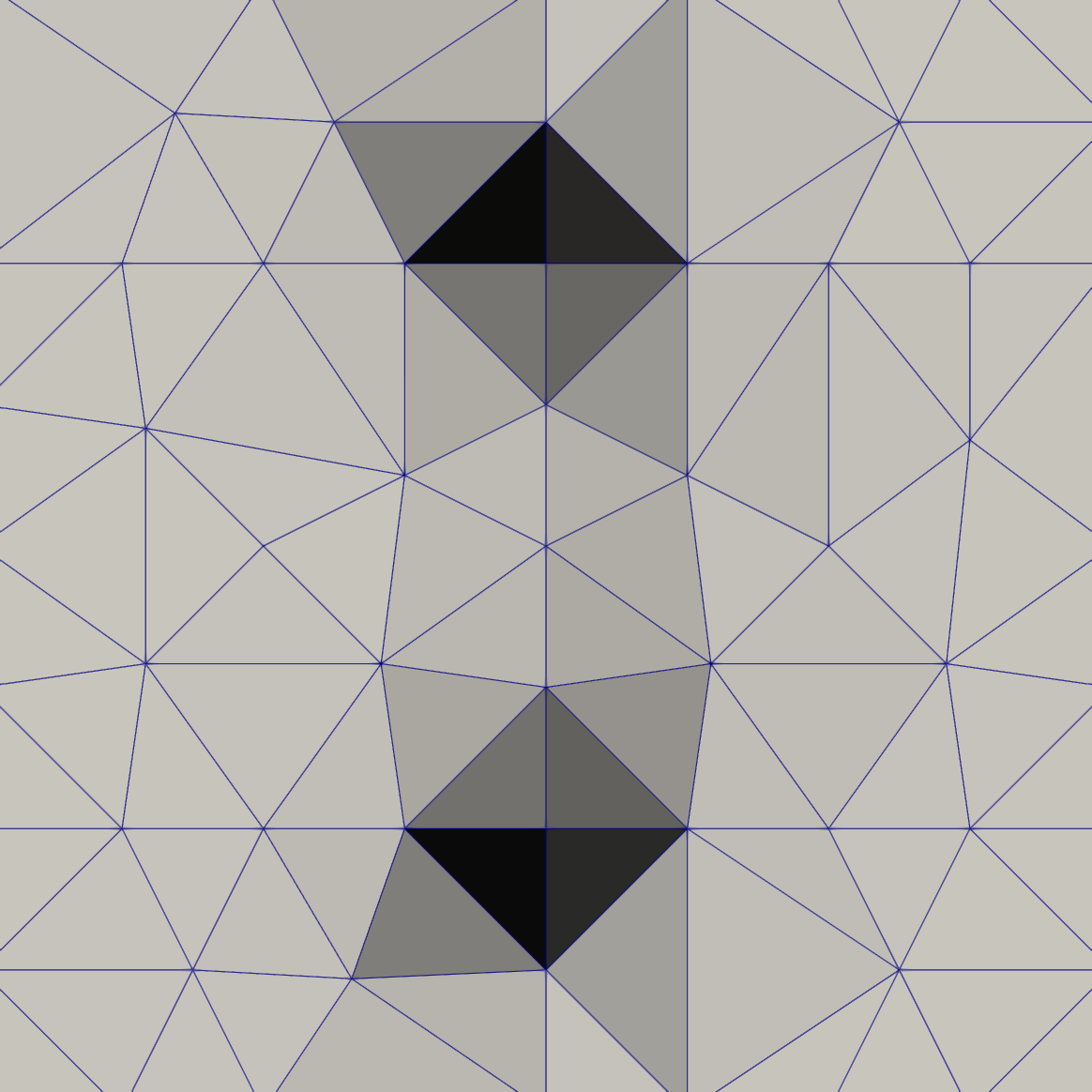}
    \end{subfigure} 
    \\
    \vspace{5mm}
	\hspace{-0.cm}
	\begin{subfigure}[b]{0.7\textwidth}
 		\includegraphics[height=.45\linewidth]{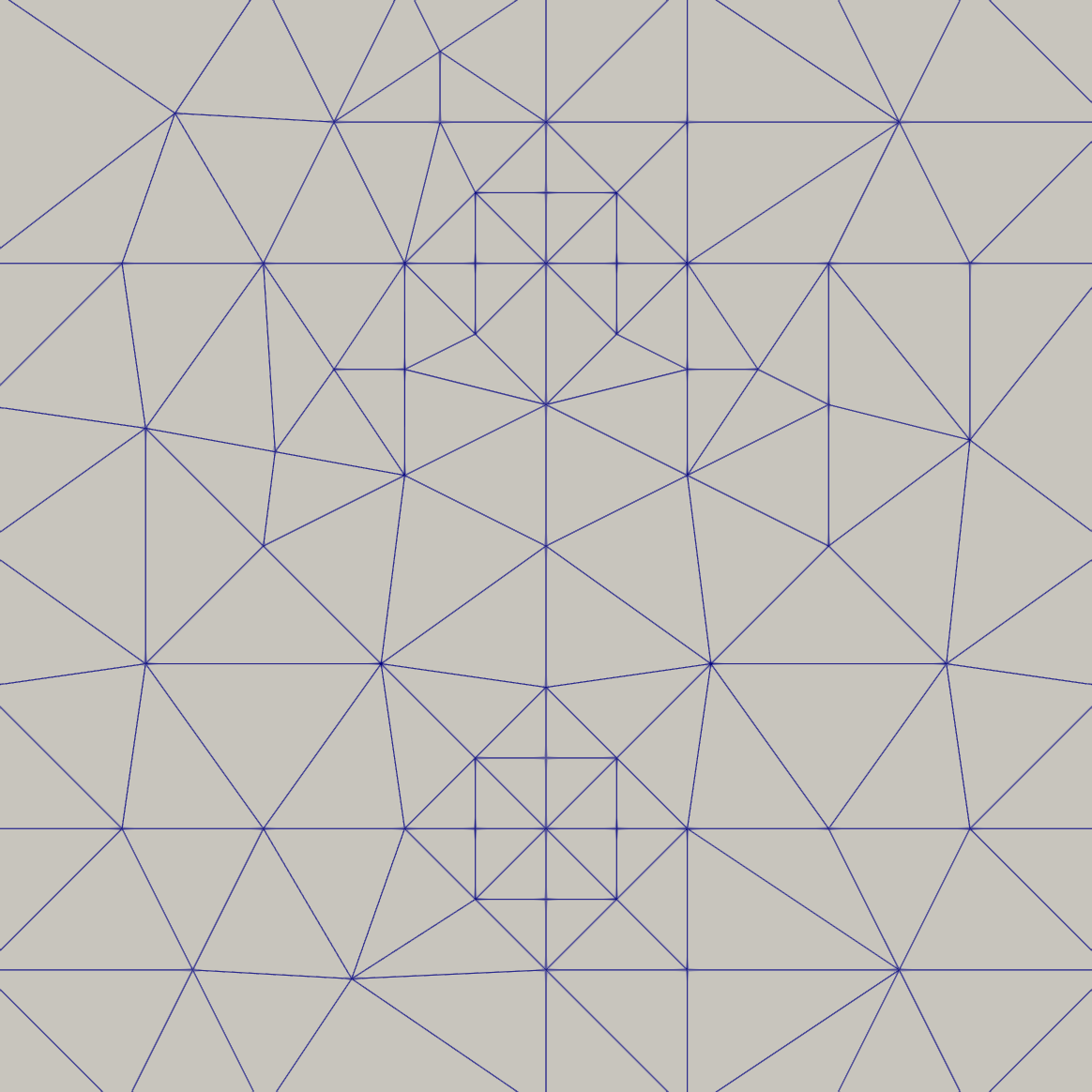} ~\hspace{3mm}
 		\includegraphics[height=.45\linewidth]{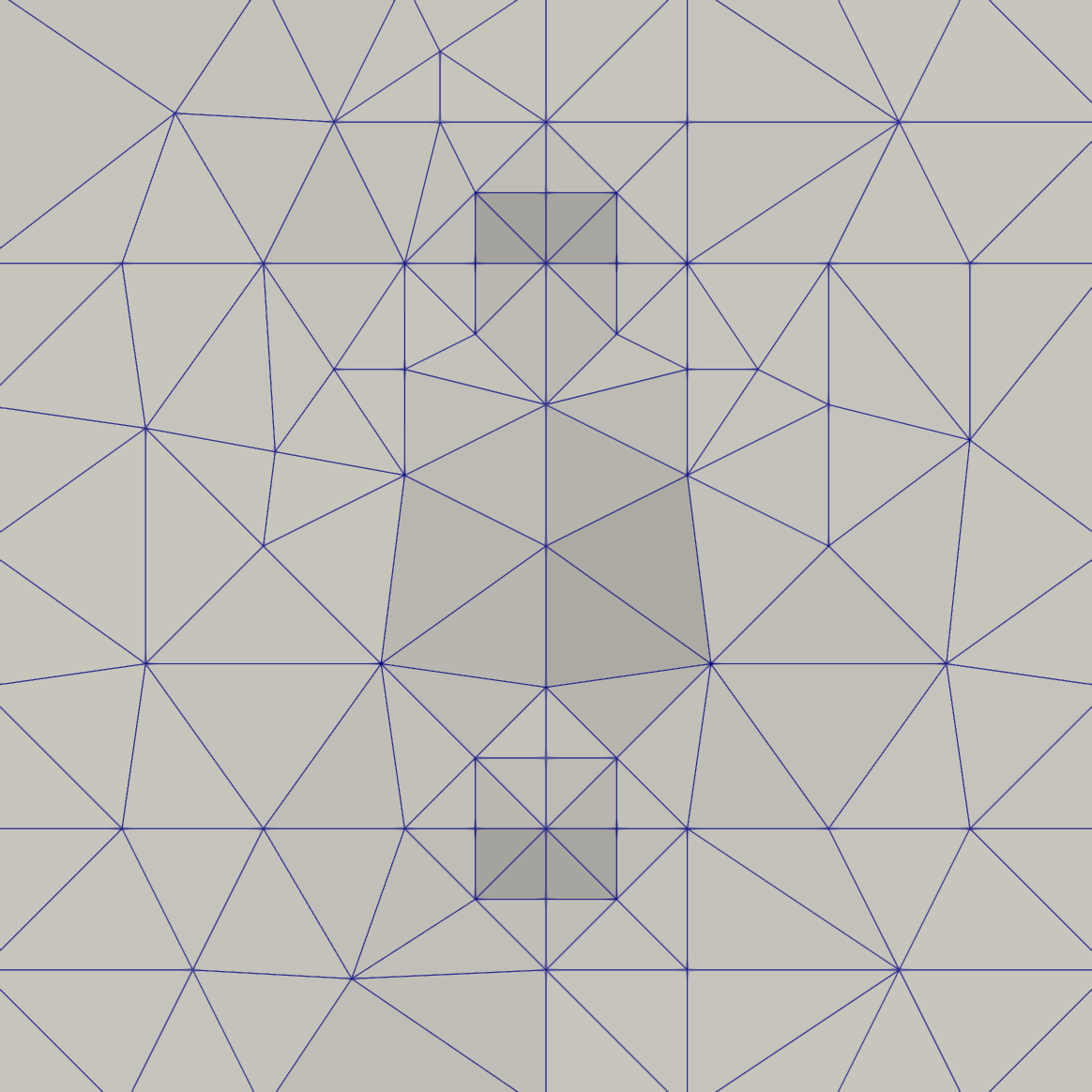}
  	\end{subfigure}
  	\caption{\tred{Distribution of $\{\tilde{\eta}_{\Gamma,T}\}$ (left), $\{\tilde{\eta}_{0,T}\}$ (right) in the initial and the first mesh refinement for $\alpha = 100$ (color scale: white = 0, black = 4.5e-4)} }
  	\label{fig:eta-alpha-100}
\end{figure}
\begin{figure}[h]  
	\hspace{-0.cm}
	\begin{center}
	\begin{subfigure}[b]{0.7\textwidth}        
		\includegraphics[width=.55\linewidth]{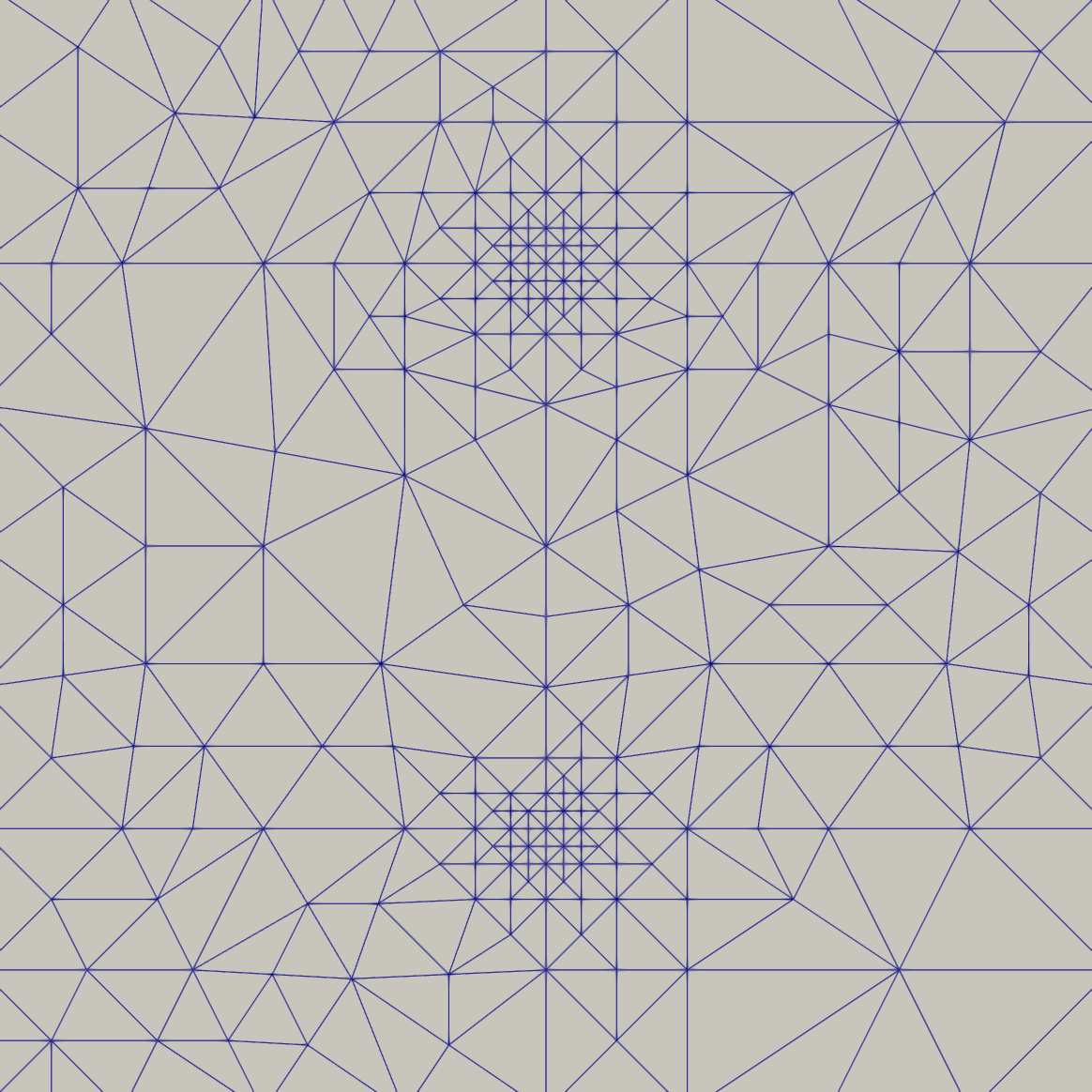} ~\hspace{3mm}
 		\includegraphics[width=.55\linewidth]{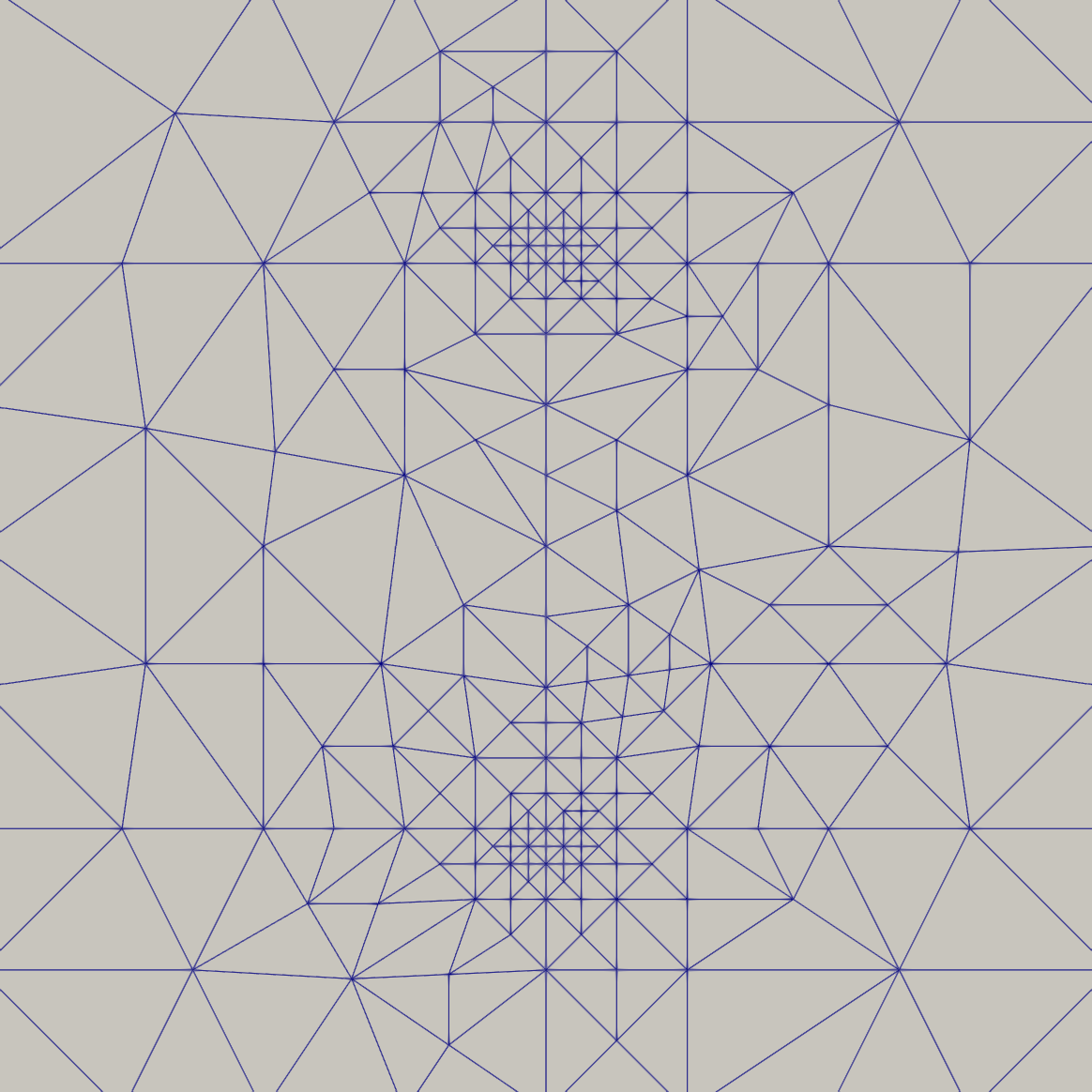} ~\hspace{3mm}
    \end{subfigure} 
    \\
	\begin{subfigure}[b]{0.7\textwidth}
 		\includegraphics[width=.55\linewidth]{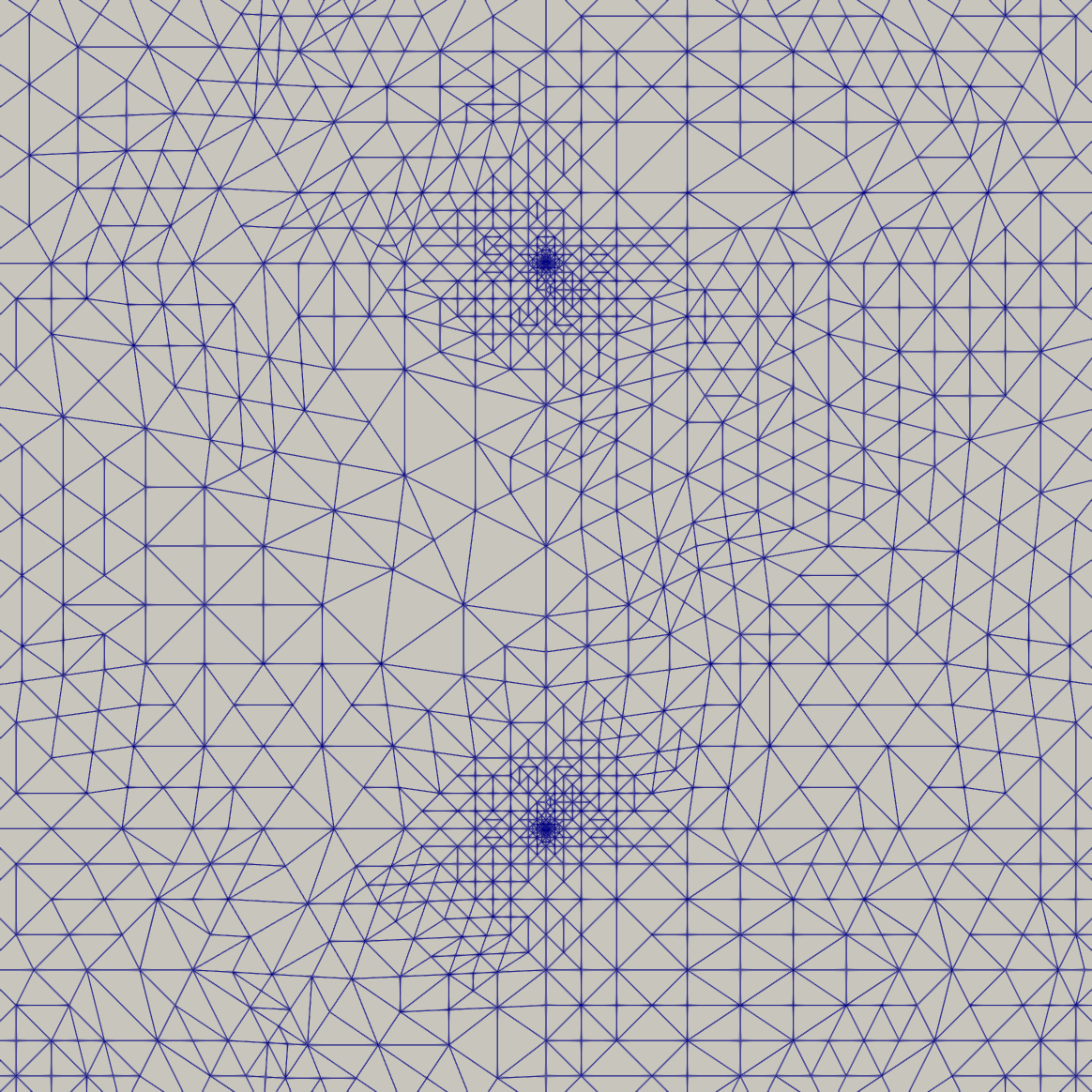} ~\hspace{3mm}
 		\includegraphics[width=.55\linewidth]{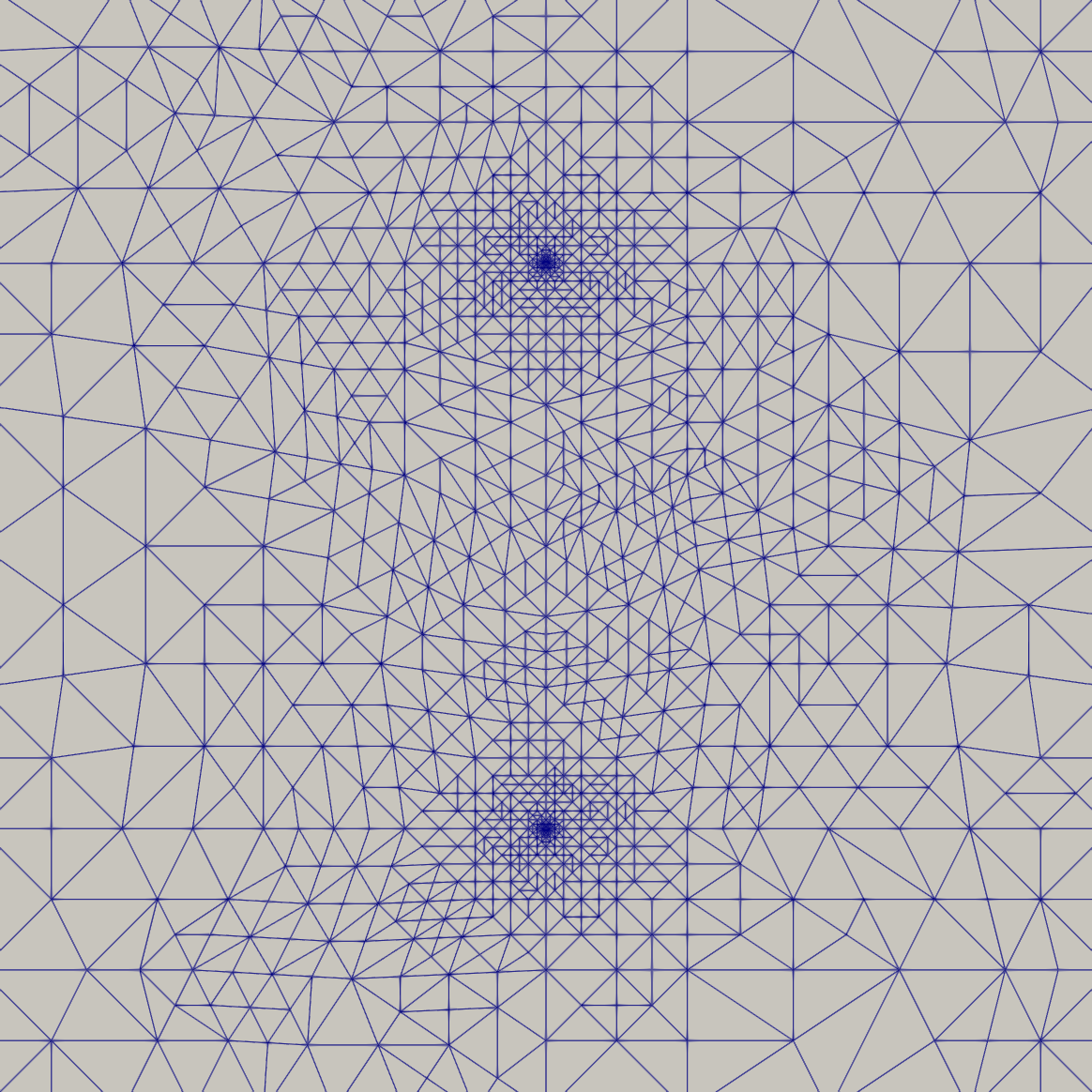} ~\hspace{3mm}
  	\end{subfigure}
  	\\
	\begin{subfigure}[b]{0.7\textwidth}
 		\includegraphics[width=.55\linewidth]{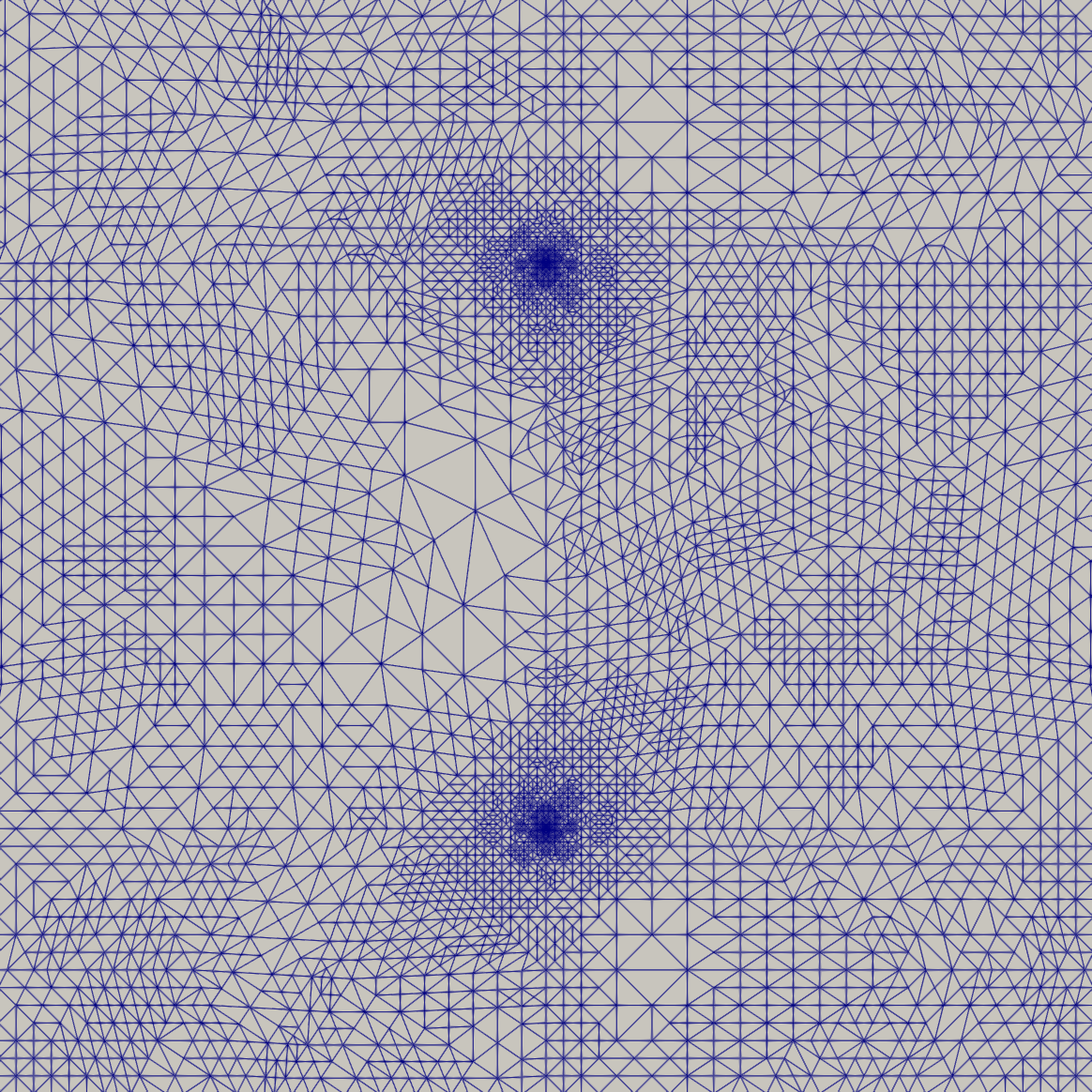} ~\hspace{3mm}
 		\includegraphics[width=.55\linewidth]{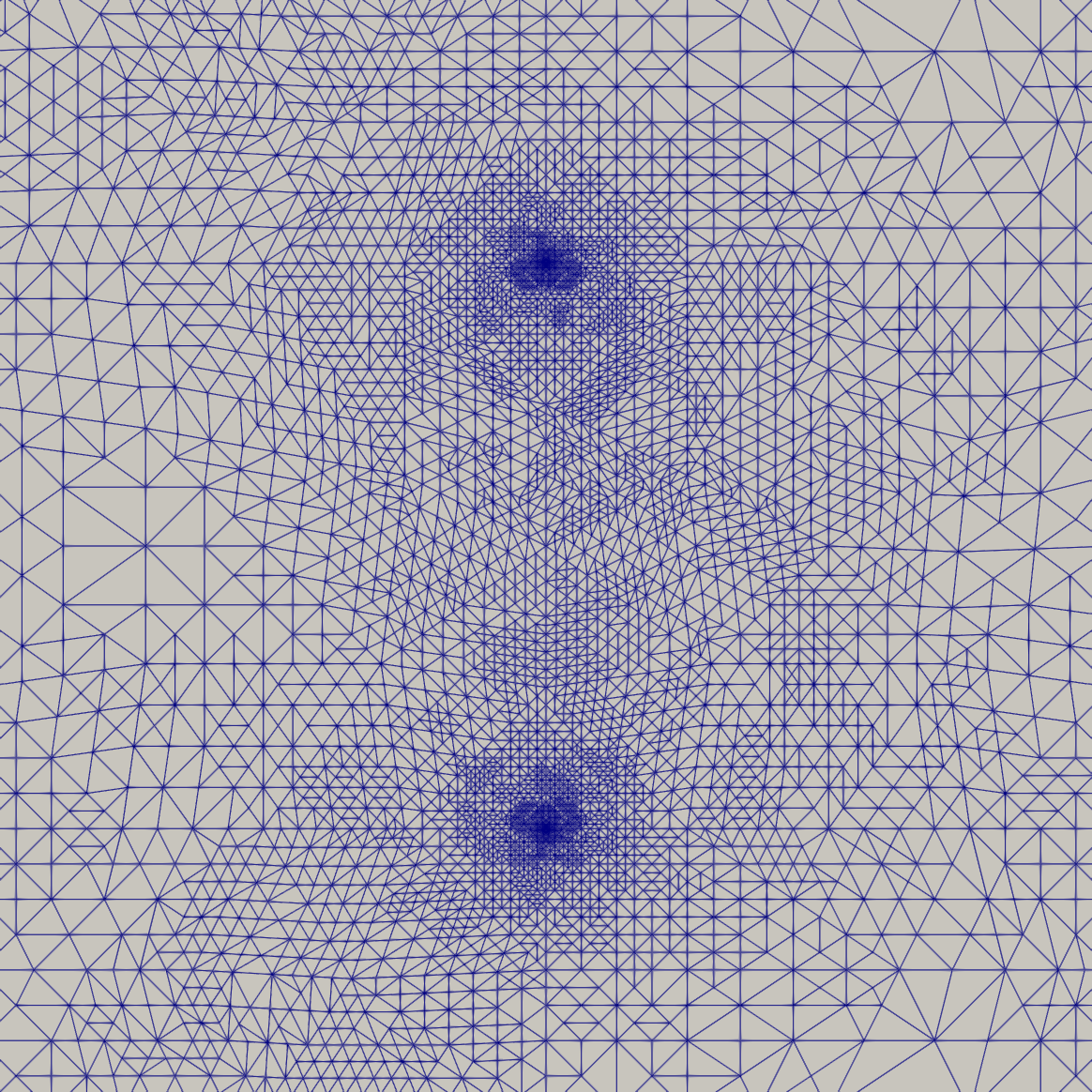} ~\hspace{3mm}
  	\end{subfigure}
  	\end{center}
  	\caption{The 3rd, 7th, 10th adaptive mesh refinements for $\alpha = 0.1$ (left) and for $\alpha= 100$ (right). }
  	\label{fig:ex3-adaptivity-comp}
\end{figure}
%

%

\begin{figure}[h]
	\begin{center}
 		\includegraphics[width=1.\linewidth]{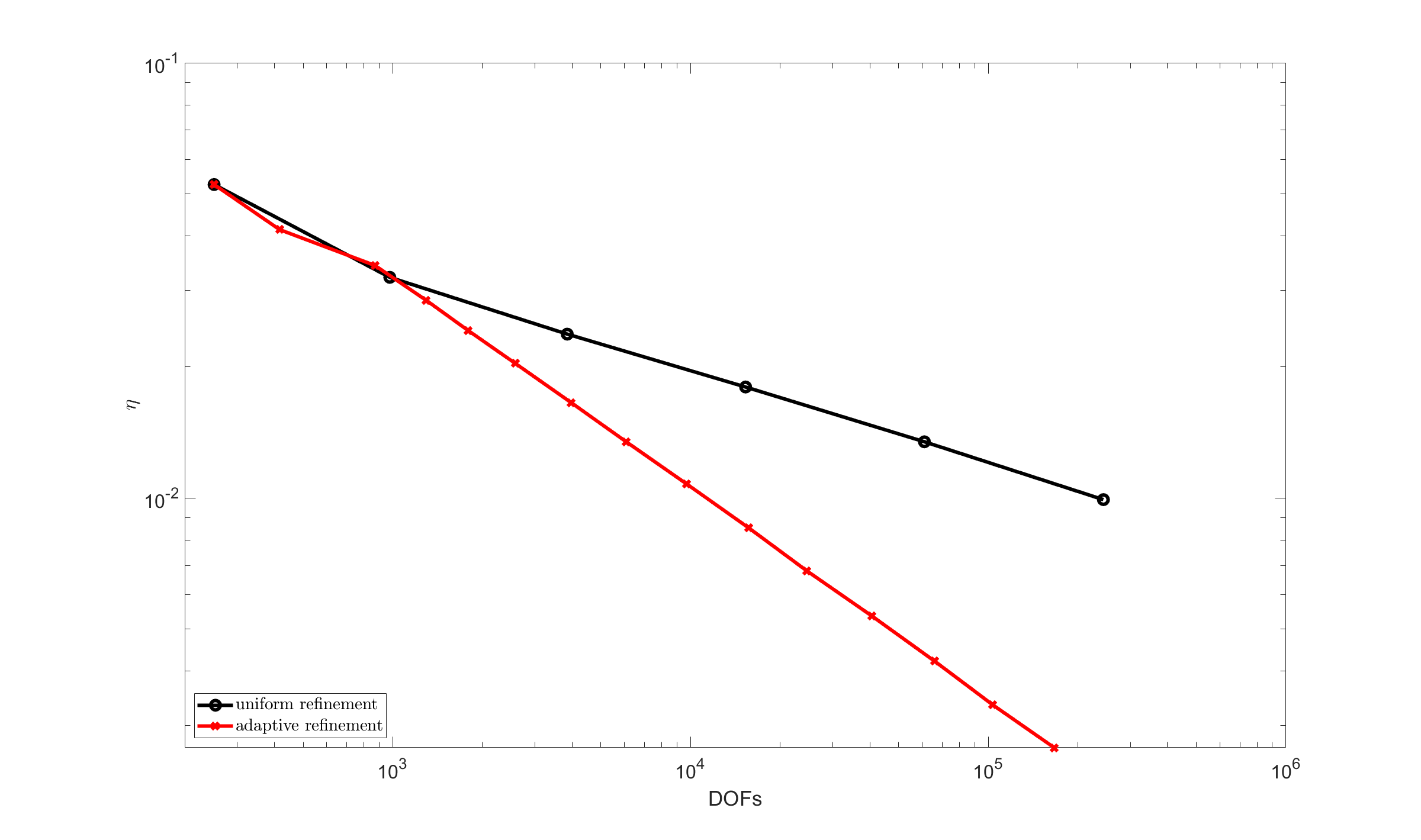} \\ \vspace{-5mm}
 		\includegraphics[width=1.\linewidth]{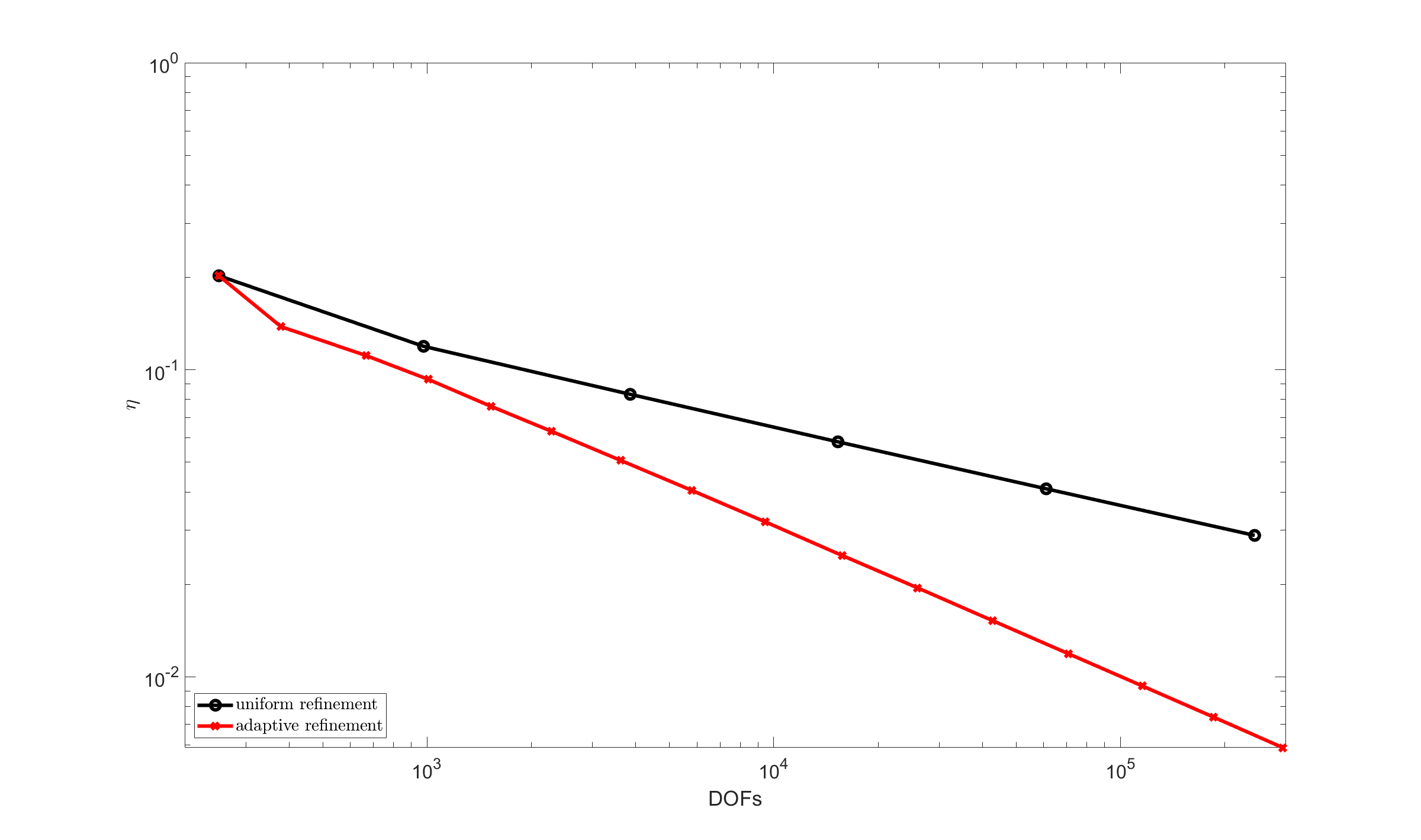} 
 	\end{center}
	\caption{Comparison of $\eta$ for uniform and adaptive mesh refinements for $\alpha = 0.1$ (top) and $\alpha=100$ (bottom). }
	\label{fig:ex3-eta-convergence}
\end{figure}
We can compute $\ub = - \nabla p$ and $f = \div \ub$ on $\Omega \setminus \Gamma$. 
Note that the manufactured solutions are not smooth on 
\algns{
	&\{ y= 3/4, 0 \le x \le 1 \,:\, (x, y) \in \Omega\}, 
	\\
	&\{ y= 1/4, 0 \le x \le 1 \,:\, (x, y) \in \Omega\}, 
	\\
	&\{ x= 1/2, 1/4 \le y \le 3/4 \,:\, (x, y) \in \Omega\}.
}
If we take an initial mesh which includes these segments in the set of edges, then the manufactured solution $(\ub, p)$ are smooth on every triangle. Therefore, the $L^2$ errors $\| \ub - \ub_h \|_0$ and $\| p - p_h \|_0$ will converge with optimal convergence rates for uniform mesh refinement. Since $p_h^*$ is an approximation of $p$ which is better than or as good as $p_h$, we only present $\| p - p_h^* \|_0$ in our numerical experiments. In this experiment, we use the lowest order BDM element for $\bs{V}_h$ and the piecewise constant element for \tred{$Q_h$, so} the optimal convergence rates of $\| \ub - \ub_h \|_0$ and $\| p - p_h \|_0$ are 2 and 1, respectively. 
The $L^2$ errors $\| \ub - \ub_h \|_0$ and $\| p - p_h^* \|_0$ up to degrees of freedom are given in Figure~\ref{fig:eg1-convergence} (black graph), and one can see that $\| p - p_h^* \|_0$ shows superconvergence. The errors for adaptive mesh refinement are also given in Figure~\ref{fig:eg1-convergence} (red graph). \tred{As can be seen in Figure~\ref{fig:ex1-stages}, mesh refinements are done mostly on the slab $0.25 < y < 0.75$ because the manufactured solution vanishes outside of this slab. Moreover, the manufacture solution is smooth on every triangle, so we do not see concentration of mesh refinements in this experiment. Nevertheless, one can see in Figure~\ref{fig:eg1-convergence} that adaptive mesh refinement gives more optimal convergence of errors up to the numbers of degrees of freedom.} The effectivity index is computed by 
\algns{
	\sqrt{\eta^2 + \frac 1{\pi^2} \sum_{T \in \mathcal{T}_h} \osc(f,T)^2} / \| \ub - \ub_h \|_0 , 
}
and \tred{the values of effectivity index up to adaptive mesh refinements are given in Table~\ref{effectivity}.}

In the second set of experiments, we present mesh adaptivity for nonsmooth solutions. 
Since it is difficult to construct nonsmooth manufactured solutions with the fault structure, we show adaptive mesh refinement by our a posteriori error estimator for numerical solutions with given boundary conditions, $\alpha=0.1,10,100$, and $f \equiv 1$. Assuming $\Omega = [0,1]\times [0,1]$ with the same $\Gamma$, zero flux boundary conditions are imposed on the top and bottom boundary components $\{(x,y) \in \Omega \,:\, 0 \le x \le 1, y=0 \text{ or } y = 1\}$ of $\partial \Omega$, and $p = 0$ on the left side, $p=-1$ on the right side, are imposed. 
\begin{figure}[h]  
	\hspace{-0.cm}
	\begin{center}
	\begin{subfigure}[b]{0.7\textwidth}        
		\includegraphics[width=.55\linewidth]{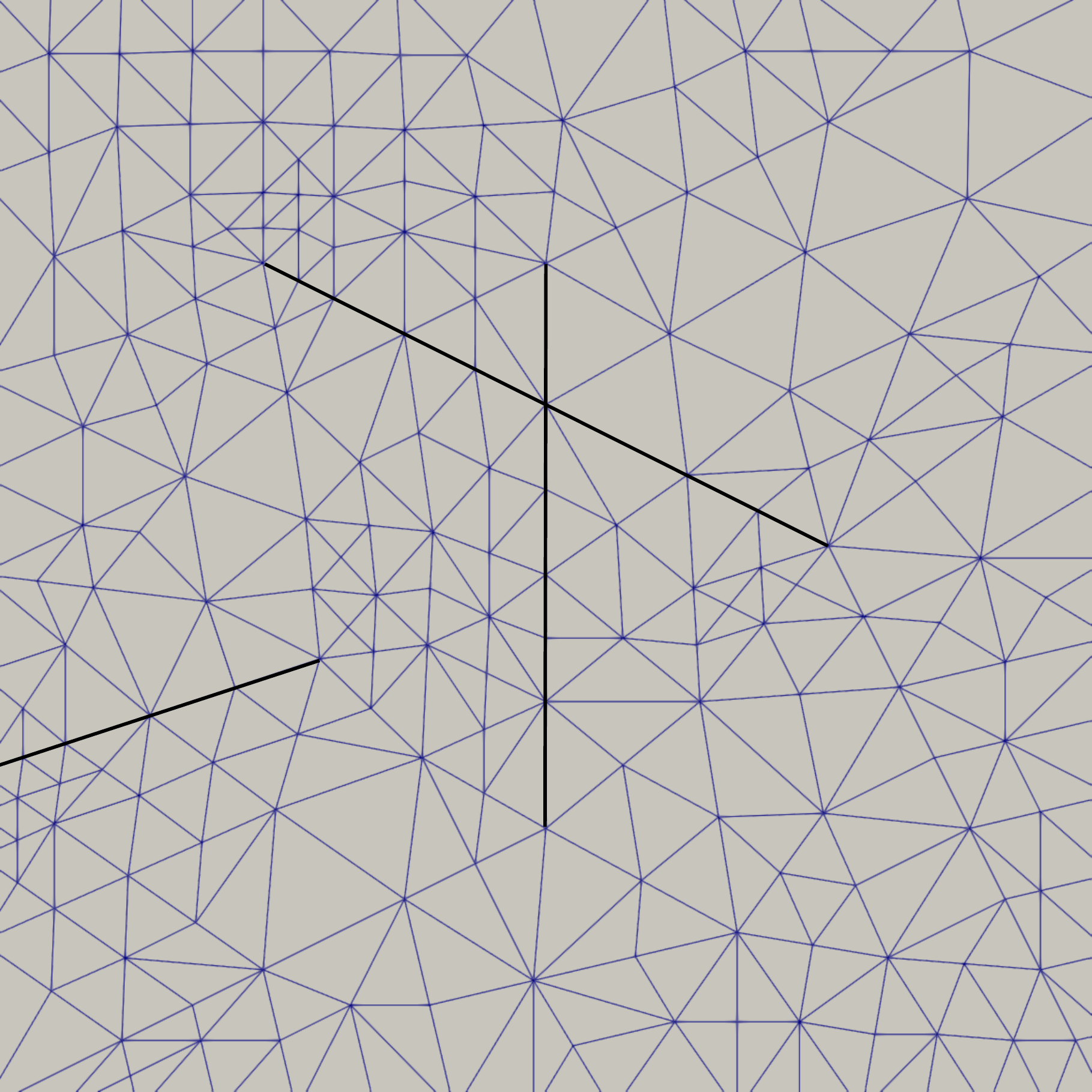} ~\hspace{3mm}
 		\includegraphics[width=.55\linewidth]{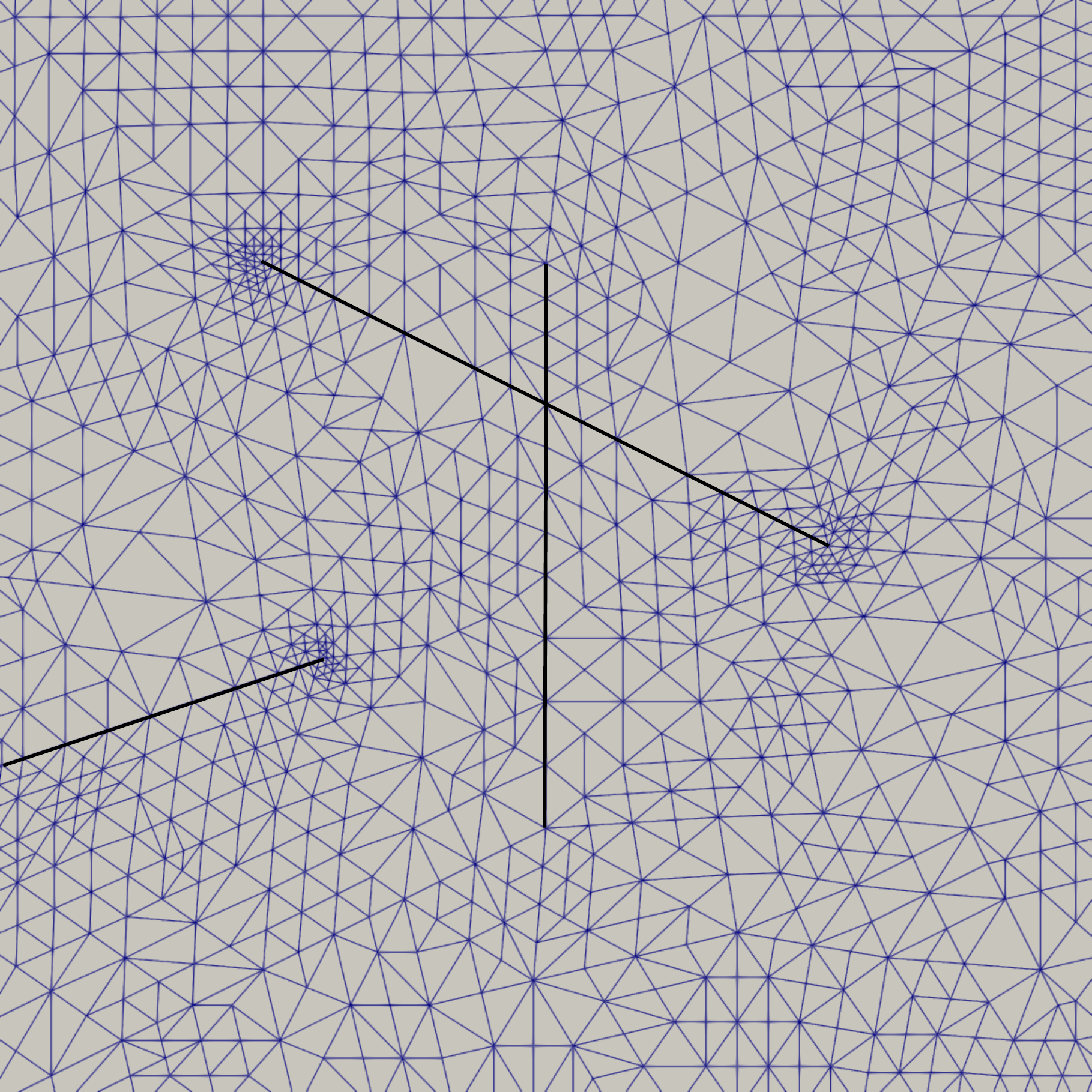} ~\hspace{3mm}
    \end{subfigure} 
    \\
	\begin{subfigure}[b]{0.7\textwidth}
 		\includegraphics[width=.55\linewidth]{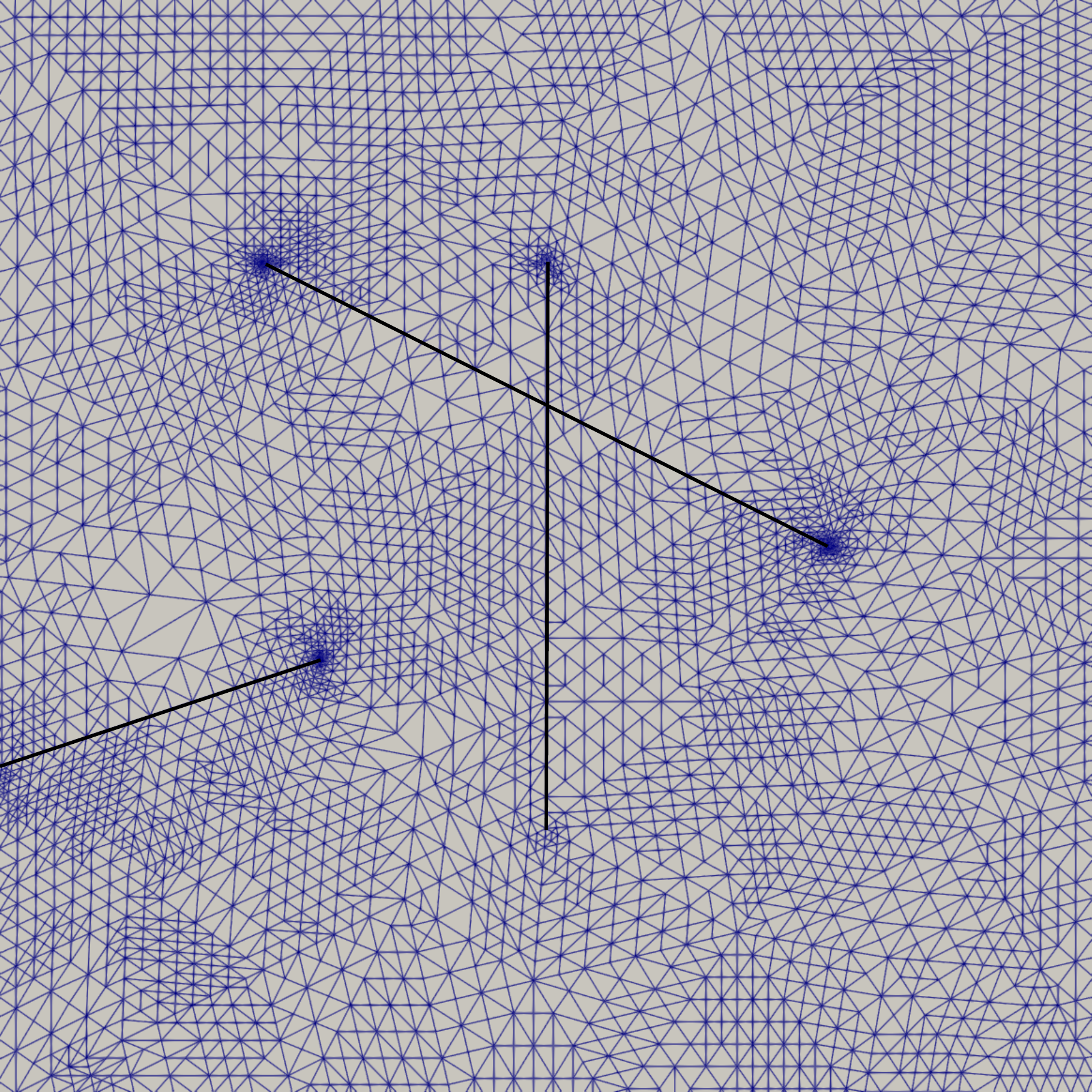} ~\hspace{3mm}
 		\includegraphics[width=.55\linewidth]{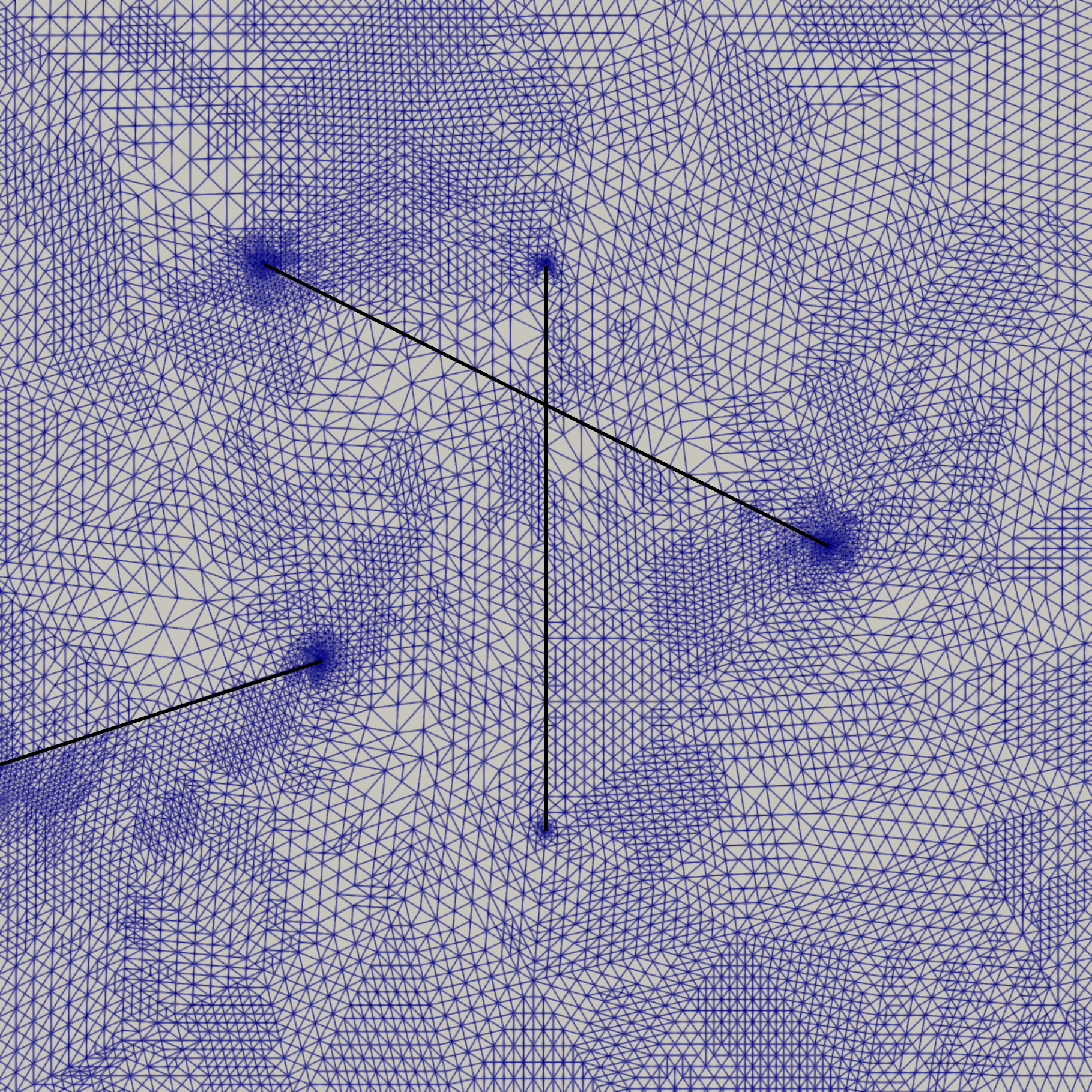} ~\hspace{3mm}
  	\end{subfigure}
  	\end{center}
  	\caption{The 2nd, 5th, 8th, and 10th adaptive mesh refinements for multiple faults. Fault 1 (vertical), fault 2 (upper oblique), fault 3 (touching boundary) have $\alpha$ values 100, 10, 50, respectively. }
  	\label{fig:ex4-adaptivity-comp}
\end{figure}

\begin{figure}[h]
	\begin{center}
 		\includegraphics[width=1.\linewidth]{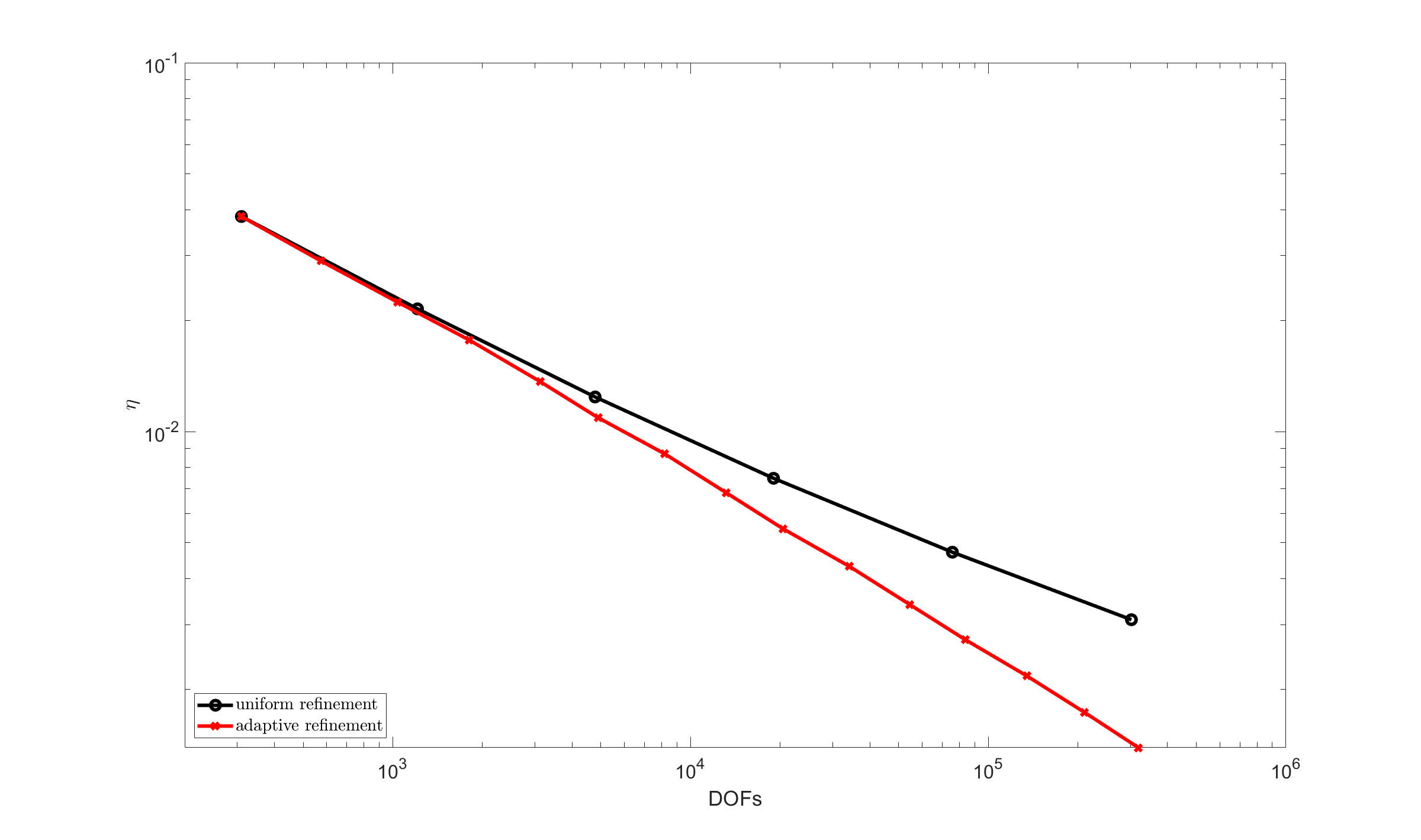}  
 	\end{center}
	\caption{\tred{Comparison of $\eta$ of uniform and adaptive mesh refinements for the multiple faults experiment.} }
	\label{fig:ex4-eta-convergence}
\end{figure}

\tred{Here we use the lowest-order Raviart--Thomas element for experiments. Since $f \equiv 1$, the data oscillation terms vanish. Moreover, $\eta_T$ vanishes for all $T \in \mathcal{T}_h$ because local  shape functions of the lowest order Raviart--Thomas element \eqref{eq:RTN-local} with $k=1$ is included in $\nabla Q_h^*(T)$ space (cf. \eqref{eq:Qhstar-def}) with $k=1$. Therefore, only $\{\eta_E\}_{E \in \mathcal{E}_h^0 \cup \mathcal{E}_h^{\Gamma}}$ gives meaningful values. Since $\eta_E$'s are quantities on edges, which are difficult to visualize, we define $\{\tilde{\eta}_{\Gamma,T} \}_{T \in \mathcal{T}_h}$ and $\{\tilde{\eta}_{0,T} \}_{T \in \mathcal{T}_h}$  by 
\begin{align*}
	\tilde{\eta}_{\Gamma,T}^2 &:= 
		\begin{cases}
			\frac 12 \sum_{E \subset \partial T \cap \Gamma} (\eta_E^2) \quad \text{ if } \pd T \cap \Gamma \not = \emptyset \\
			0 \quad \text{ otherwise} 
		\end{cases}
	, 
	\\
	\tilde{\eta}_{0,T}^2 &:= 
		\begin{cases}
			\frac 12 \sum_{E \subset \partial T \setminus \Gamma} \eta_E^2 \quad \text{ if } \pd T \cap \Gamma = \emptyset \\
			0 \quad \text{ otherwise} 
		\end{cases} ,
\end{align*}
and look at the distributions of $\{\tilde{\eta}_{\Gamma,T}\}$ and $\{\tilde{\eta}_{0,T}\}$. 
In Figures~\ref{fig:eta-alpha-tenth} and \ref{fig:eta-alpha-100}, distributions of $\{\tilde{\eta}_{\Gamma,T}\}$ (left) and $\{\tilde{\eta}_{0,T}\}$ (right) are presented for $\alpha=0.1,100$. In these results we can see that the quantities of $\eta_E$ on the faults are not significant whereas the quantities of $\eta_E$ near the fault are much larger. Moreover, the boundary conditions on the top and bottom boundaries do not give large a posteriori error estimator values near the boundaries. 
}
\tred{In Figure~\ref{fig:ex3-adaptivity-comp}, we presented the 3rd, 7th, 10th mesh refinements for $\alpha=0.1,100$. In both cases, mesh adaptivity is obvious concentrating near the two endpoints of the fault. However, one can see that the fault with $\alpha=100$ (low permeable fault) needs more refinements near the internal fault segment. We believe that this is because the low permeability fault $(\alpha=100)$ can cause more drastic pressure changes near the fault, so the solution regularities are lower than the ones of true solutions with $\alpha=0.1$. 
}
\tred{To see efficiency of adaptive schemes, we give two comparison graphs of $\eta$ and the degrees of freedom in Figure~\ref{fig:ex3-eta-convergence}. The results clearly show that adaptive methods are  more efficient, quantitatively about $6.9$ times for $\alpha=0.1$ and about $2.3$ times for $\alpha=100$. 
}

\tred{In the last experiment we solve the equation with 3 faults and present mesh refinement history. The 3 faults have different $\alpha$ values (see Figure~\ref{fig:ex4-adaptivity-comp} for details). We still observe that mesh refinements are concentrated at the ends of faults. The comparisons of $\eta$ and the degrees of freedom for uniform and adaptive meshes are given in Figure~\ref{fig:ex4-eta-convergence}. The efficiency of adaptive methods is not as high as the single fault examples. This is probably because irregular solutions due to the multiple faults give large $\eta$ values on most regions of the domain, so refined meshes by a posteriori error estimator are not so different from uniform refinements.   
}

 
\section{Conclusion}
In this work we studied a recovery type a posteriori error estimator of the Darcy flow model with Robin-type interface conditions. The reliability and the local efficiency of the estimator are proved. In contrast to the previous work in \cite{Konno-Schotzau-Stenberg:2011}, we developed a new $H(\div)$-based proof using a modified Helmholtz decomposition, a modified Scott--Zhang interpolation, edge/face-wise integration by parts. Moreover, we proved that the post-processed pressure is bounded by the estimator, and a superconvergent upper bound can be obtained under a (partial) elliptic regularity assumption of the dual problem. Numerical test results are included to illustrate the adaptivity results of our estimator.

\section{Appendix: integration by parts identities}
In this section we present identities from the integration by parts that we used in the paper.

In this section $(a_1 \; \cdots\; a_n)^t$ denotes the column vector with entries $a_1, \cdots, a_n$.
For differentiable functions $\phi : \mathbb{R}^2 \to \mathbb{R}$, $\Psi: \mathbb{R}^2 \to \mathbb{R}^2$ with 
$\Psi = (\Psi_1 \; \Psi_2)^t$, we define $\curl$ and $\rot$ by
\algn{ \label{eq:2d-curl-rot}
	\curl \phi = \pmat{- \pd_y \phi \\ \pd_x \phi}, \qquad \rot \Psi = -\pd_y \Psi_1 + \pd_x \Psi_2 .
}
For a triangle $T \subset \mathbb{R}^2$, $\bs{n}_{\pd T}$ is the outward unit normal vector field on $\pd T$ and $\bs{t}_{\pd T}$ is the unit tangential vector field along the counterclock-wise direction of $\pd T$. Denoting $\bs{n}_{\pd T}$ by $\bs{n}_{\pd T} = (n_1 \; n_2)^t$, note that $\bs{t}_{\pd T} = (-n_2 \; n_1)^t$. By the integration by parts,
\algn{
	\notag
	\int_T \curl \phi \cdot {\Psi} \,ds &= \int_T (- \pd_y \phi \Psi_1 + \pd_x \phi \Psi_2) \, ds
	\\
	\label{eq:IBP-2d-1}
	&= \int_T \nabla \phi \cdot \pmat{ \Psi_2 \\ - \Psi_1} \,ds 
	\\
	\notag
	&= \int_{\pd T} \phi \pmat{n_1 \\ n_2} \cdot \pmat{\Psi_2 \\ -\Psi_1} \,dl - \int_T \phi (\pd_x \Psi_2 - \pd_y \Psi_1) \, ds
	\\
	\notag
	&= \int_{\pd T} \phi \pmat{-n_2 \\ n_1} \cdot \pmat{\Psi_1 \\ \Psi_2} \,dl - \int_T \phi \rot {\Psi} \, ds
	\\
	\notag
	&= \int_{\pd T} \phi \bs{t}_{\pd T} \cdot \Psi \,dl - \int_T \phi \rot \Psi \, ds
}
and for $E \subset \pd T$, 
\algn{
	\notag
	\int_{E} \curl \phi \cdot \bs{n}_{\pd T} \,dl &= \int_{E} \pmat{- \pd_y \phi \\ \pd_x \phi} \cdot \pmat{n_1 \\ n_2} \, dl
	\\
	\label{eq:IBP-2d-2}
	&= \int_{E} \nabla \phi \cdot \pmat{ n_2 \\ - n_1} \,dl 
	\\
	\notag
	&= -\int_{E} \nabla \phi \cdot \bs{t}_{\pd T} \,dl .
}

Let $F$ be a triangle in the $xy$-plane in $\mathbb{R}^3$ and $\bs{n}_{\pd F} = (n_1 \; n_2\; 0)^t$ be the unit outward normal vector field of $F$ in $\mathbb{R}^3$. The tangential vector field on $\pd F$ in $\mathbb{R}^3$ is $\bs{t}_{\pd F} = (-n_2 \; n_1 \; 0)^t$. 
For differentiable functions $\phi:\mathbb{R}^3 \to \mathbb{R}$, $\Psi : \mathbb{R}^3 \to \mathbb{R}^3$ with 
$\Psi = ({\Psi_1 \; \Psi_2 \; \Psi_3})^t$, $\bs{n}_F := ({0 \; 0 \; 1})^t$, we get 
\algn{ \label{eq:3d-curl-identity}
	\bs{n}_F \times \Psi = \pmat{-\Psi_2 \\ \Psi_1 \\ 0} , \quad \curl \Psi \cdot \bs{n}_F = \pd_x \Psi_2 - \pd_y \Psi_1.
}
By these identities and the 3rd equality in \eqref{eq:IBP-2d-1}, 
\algn{
	\notag
	\int_F \curl \Psi \cdot \bs{n}_F \phi \, ds &= \int_F (\pd_x \Psi_2 - \pd_y \Psi_1 )\phi \,ds
	\\
	\label{eq:IBP-3d-1}
	&= - \int_F \pmat{-\pd_y \phi \\ \pd_x \phi} \cdot \pmat{\Psi_1 \\ \Psi_2} \,ds + \int_{\pd F} \phi \pmat{n_1 \\ n_2} \cdot \pmat{\Psi_2 \\ -\Psi_1} \,dl
	\\
	\notag
	&= \int_F \pmat{\pd_x \phi \\ \pd_y \phi} \cdot \pmat{-\Psi_2 \\ \Psi_1} \,ds + \int_{\pd F} \phi \pmat{-n_2 \\ n_1} \cdot \pmat{\Psi_1 \\ \Psi_2} \,dl
	\\
	\notag
	&= \int_F \nabla \phi \cdot (\bs{n}_F \times \Psi) \,ds - \int_{\pd F} \phi \bs{t}_{\pd F} \cdot \Psi \,dl .
}
%

\phantom{a}

\bigskip

\noindent {\bf Funding} Jeonghun J. Lee gratefully acknowledge support from the
National Science Foundation (DMS-2110781).


\end{document}